\documentclass[12pt]{amsart}
\usepackage[english]{babel}
\usepackage{bbm, amssymb, amsmath, amsfonts}

\usepackage[unicode]{hyperref}

\numberwithin{equation}{section}

\usepackage{graphicx, epsfig}

\usepackage{xcolor}

 \makeindex

\def\cal{\mathcal}
\def\frak{\mathfrak}
\def\Bbb{\mathbb}

\def\sgn{\text{\rm sgn\,}}

\def\ra{\rangle}

\def \supp {\text{\rm supp\,}}

\def \sgn {\text{\rm sgn\,}}
\def\trans{{\,}^t}

\def\c{{\frak c}}

\def\D{{\cal D}}

\def\C{{\cal C}}
\def\F{{\cal F}}

\def\L{{\cal L}}
\def\M{{\cal M}}

\def\S{{\cal S}}

\def\bC{{\Bbb C}}

\def\NN{{\Bbb N}}

\def\bR{{\Bbb R}}
\def\R{{\Bbb R}}
\def\RR{{\Bbb R}}

\def\ZZ{{\Bbb Z}}
\def\charac{{\mathbbm{1}}}

\def\vp{{\varphi}}
\def\al{{\alpha}}
\def\be{{\beta}}
\def\ga{{\gamma}}
\def\Ga{{\Gamma}}
\def\la{{\lambda}}
\def\om{{\omega}}

\def\e{{\varepsilon}}
\def\x{(x_1,x_2)}

\def\pa{{\partial}}

\def\ve{{\varepsilon}}
\def\si{{\sigma}}

\def\de{{\delta}}

\def\ka{{\kappa}}

\def\bGa{{\mathbf  \Gamma}}

\def\q{{\frak q}}

\def\Rwurz{R^{-1/2} }

\def\vth{{\vartheta}}
\def\tth{{\Theta}}

\def\transl{{\slash\hskip-0.2cm\slash\,}}
\def\translb{{\slash\hskip-0.1cm\slash\,}}

\def\bpm{\begin{pmatrix}}
\def\epm{\end{pmatrix}}

\def\noi{\noindent}
\def\bee{\begin{enumerate}}
\def\ee{\end{enumerate}}

\def\qed{\smallskip\hfill Q.E.D.\medskip}

\newtheorem{thm}{Theorem}[section]

\newtheorem{cor}[thm]{Corollary}
\newtheorem{lemma}[thm]{Lemma}
\newtheorem{lem}[thm]{Lemma}
\newtheorem{remark}[thm]{Remark}

\theoremstyle{plain}

\newtheorem{example}[thm]{Example}

\begin{document}


\title[$L^p$-estimates for FIO-cone  multipliers]
{ $L^p$-estimates for FIO-cone  multipliers}

\author[S. Buschenhenke]{Stefan Buschenhenke}
\address{Mathematisches Seminar, C.A.-Universit\"at Kiel,
Heinrich-Hecht-Platz 6, D-24118 Kiel, Germany} \email{{\tt
buschenhenke@math.uni-kiel.de}}

 \author[S. Dendrinos]{Spyridon  Dendrinos}
\address{School of Mathematical Sciences, University College Cork,
Cork, Ireland}
 \email{{\tt sd@ucc.ie}}

\author[I. A. Ikromov]{Isroil A. Ikromov}
\address{Institute  of Mathematics, 
University Boulevard 15, 140104, Samarkand, Uzbekistan}
 \email{{\tt ikromov1@rambler.ru}}

\author[D. M\"uller]{Detlef M\"uller}
\address{Mathematisches Seminar, C.A.-Universit\"at Kiel,
Heinrich-Hecht-Platz 6, D-24118 Kiel, Germany} \email{{\tt
mueller@math.uni-kiel.de}}




\begin{abstract}  
The classical cone multipliers are Fourier multiplier operators  which localize to narrow $1/R$-neighborhoods of the truncated light cone in frequency space. By composing such convolution operators with suitable translation invariant Fourier integral operators (FIOs), we obtain what we call FIO-cone multipliers. We introduce and study classes of such FIO-cone multipliers on $\RR^3$,  in which the  phase functions of the corresponding FIOs are adapted in a natural way  to the geometry of the cone and may even admit singularities at the light cone.

By building on methods developed by Guth, Wang and Zhang in their  proof of the cone multiplier conjecture in $\RR^3,$   we obtain $L^p$-estimates for FIO-cone multipliers in the range $4/3\le p\le 4$ which are  stronger  by the factor $R^{-|1/p-1/2|}$ than what a direct application of the method of Seeger, Sogge and Stein for estimating FIOs would give.

An  important application of  our theory is to  maximal  averages along smooth analytic surfaces  in $\RR^3.$ It allows to  confirm  a conjecture on the  the critical Lebesgue exponent  for a prototypical  surface from a small class of ``exceptional'' surfaces, for which this conjecture had remained open.

\end{abstract}


\maketitle


\tableofcontents

\thispagestyle{empty}


\section{Introduction}\label{intro}

If $m$ is an essentially bounded measurable function  on $\RR^3,$ we denote by $T_m$ the associated Fourier multiplier operator defined by 
$$
\widehat{T_m f}:= m\hat f, \qquad f\in \S(\RR^3),
$$
i.e.,  $T_mf=f* (\F^{-1}m),$ if $ \F^{-1}m\in \S'(\RR^3)$ denotes the inverse Fourier transform of $m.$

\smallskip

Let us denote by $\mathbf \Gamma:=\{ \xi_1^2+\xi_2^2=\xi_3^2\}$ the light cone in $\RR^3.$ 
The  classical  {\it cone multiplier} $m_R$ is a multiplier supported in  the  truncated $1/R$-neighborhood $\Ga_R$ of the light cone given by
$$ \Gamma_R:=\{\xi\in\R^3:  \Big|\frac{\xi_1^2+\xi_2^2}{\xi_3^2}-1\Big|<\frac{1}R,\ \frac12<\xi_3<2\},$$
that is,
$$
m_R(\xi):= \chi_0 (RF(\xi))\chi_1(\xi_3),
$$
where $\chi_0$ and $\chi_1$ are smooth bump functions supported near the origin and near $1$, respectively,  and where
$$F(\xi):=\frac{\xi_1^2+\xi^2_2}{\xi_3^2}-1.$$ 
We assume here that $R\ge R_0,$ where $R_0\gg1$ is fixed and sufficiently large.

The long standing cone multiplier conjecture, which has eventually  been proved  a few years ago by Guth, Wang and Zhang  \cite{GWZ}, states  that  for any $R\gg1,$
\begin{equation}\label{coneMult}
	\| T_{m_R} f\|_{L^4(\RR^3)}\le C_\epsilon R^\epsilon \|f\|_{L^4(\RR^3)},
\end{equation}
for any $\epsilon>0.$ This estimate is  sharp, up to the power $\epsilon$ loss in $R$.

Indeed, Mockenhaupt \cite{Mo} had already observed in 1993 that the estimate \eqref{coneMult} would be a consequence of the following square function estimate, which is in fact the key result in \cite{GWZ}:

Suppose we cover $\Ga_R$ be finitely overlapping sectors or ``caps'' $\theta$ of angular width $R^{-1/2},$ so that each  sector is then essentially a rectangular box of dimensions about $R^{-1}\times R^{-1/2}\times 1,$ and choose a smooth partition of unity 
$\{\chi_\theta\}_\theta$ subordinate to the cover. If $f$ is a function whose Fourier transform is supported in $\Ga_R,$ we can then decompose $f=\sum_\theta f_\theta,$  where $f_\theta:=T_{\chi_\theta}f.$ Then, for any $\epsilon>0$ and $R\ge R_0,$ 
\begin{equation}\label{squarest}
\|f\|_{L^4(\RR^3)}\le C_\epsilon R^\epsilon \Big\|\big (\sum_\theta |f_\theta|^2\big)^{\frac 12}\Big\|_{L^4(\RR^3)}.
\end{equation}

\medskip
Returning to cone multipliers, by composing such operators   with suitable translation invariant Fourier integral operators (FIOs), we obtain what we call FIO-cone multipliers.

In this article, we introduce classes $\F^{\kappa_1,\kappa_2,\gamma}$ of real and 1-homogeneous phase functions $\phi$ for suitable non-negative parameters $\kappa_1,\kappa_2,\gamma,$ by imposing a certain control on the  derivatives of $\phi$ at any point $\xi_0$ in $\Ga_{R_0}\setminus \mathbf \Ga$ in the directions of an  orthonormal frame $E_1(\xi_0),E_2(\xi_0),E_3(\xi_0)$  attached to   
$\xi_0$ in a natural way.
More precisely, if $\mathbf \Gamma(\xi_0)$ denotes the cone  given by 
$$
\mathbf \Gamma(\xi_0):=\{F(\xi)=F(\xi_0)\}
$$ 
away from the origin, then we choose $E_1(\xi_0)$ to be the outer unit normal vector to the cone $\mathbf \Gamma(\xi_0)$ at $\xi_0,$  $E_3(\xi_0):=\xi_0/|\xi_0|$ as the tangent vector to $\mathbf \Gamma(\xi_0)$ pointing in radial direction, and $E_2(\xi_0)$ as a suitable ``horizontal'' tangent vector. The  phases in $\F^{\kappa_1,\kappa_2,\gamma}$ are allowed to have singularities at the light-cone  $\mathbf \Ga,$ which turns out to be crucial for our applications to maximal averages studied in Section \ref{sec:maxop}.
For precise definitions,  we refer to Section \ref{sec:phases}.
\smallskip

Given a specific class $\F^{\kappa_1,\kappa_2,\gamma},$ our goal is to prove $L^p$-estimates for FIO-cone multipliers of the form 
$$ 
m_R^\lambda(\xi):=a_R(\xi)m_R(\xi) e^{ -i\lambda\phi(\xi)},
$$
with phases $\phi\in \F^{\kappa_1,\kappa_2,\gamma}$ and suitable ``amplitudes'' $a_R.$ 
\smallskip

To this end, adapting a key idea from \cite{SSS} to phases $\phi$  from our class $\F^{\kappa_1,\kappa_2,\gamma},$ we devise in Section \ref{sboxmainthm} suitable boxes  $\vartheta\subset \Gamma_R,$ called L-boxes, whose  dimensions $\rho_1\times\rho_2\times 1$ are determined by means of the  given parameters $\kappa_1,\kappa_2,\gamma,R$ and $\la$ (see \eqref{rho}), so that the phases $\phi$ from $\F^{\kappa_1,\kappa_2,\gamma}$ behave essentially like linear functions  over such L-boxes. These boxes are chosen in some sense maximal for the given class of phases.
The  axes  of  an L-boxes $\vth$ are chosen to point in the directions of the orthonormal frame attached to the center 
$\xi_\vth$ of $\vth,$ so that $\vth$ is in a sense tangential to the cone $\mathbf\Ga(\xi_\vth),$ with its long side of size  about 1 pointing in radial direction, $\rho_2$ corresponding to the angular width of $\vth,$ and $\rho_1$ being its  ``thickness'' 
with respect to the normal direction $n(\xi_\vth).$

\smallskip

The dimensions of the L-boxes are chosen to be no larger than the dimensions of the sectors $\theta$ appearing in \eqref{squarest}. We can thus obtain a typically finer decomposition of $\Ga_R$ by first decomposing $\Ga_R$ into the sectors $\theta$ of angular width $R^{-1/2},$ and then decomposing each sector $\theta$ into about 
$$
N(\la,R):=\frac {R^{-3/2}}{\rho_1\rho_2}
$$
L-boxes $\vth.$ This leads to a decomposition  of $\Ga_R$ into about $R^{1/2} N(\la,R)$  L-boxes $\vth.$

\smallskip

 Our main result then states that under an additional  ``small mixed derivative condition (SMD)'' (see Theorem \ref{mainthm} for a more precise formulation, and  Corollary \ref {mainlp} for more general $L^p$-estimates), the following estimates hold true:
\begin{thm}\label{FIOthm}
For any $\epsilon>0$, there exists a constant $C_\epsilon>0$ such that for any $R\gg 1$ and  $\lambda\geq R^\gamma,$ 
\begin{equation}\label{FIOconest}
	\| T^\la_R f\|_{L^4(\RR^3)}\leq C_\e R^\epsilon  N(\lambda,R)^{1/2} \|f\|_{L^4(\RR^3)}.
\end{equation}
\end{thm}
\noi {\bf Examples.} (a) Assume that  $\phi$ is any classical phase (cf. Section \ref{sec:phases}). Then $\phi\in \F^{\frac 12, \frac 12,1}$ (cf. Example \ref{Fclass}), and 
$\rho_1=\la^{-1/2}\wedge R^{-1}, \rho_2=\la^{-1/2}.$  For $\la\ge R^2, $ the (SMD) condition is no extra condition, and thus  we obtain  the estimate
$$
\| T^\la_R f\|_{L^4\to L^4}\le C_\epsilon R^\epsilon  \Big(\frac{\la}{R^{3/2}}\Big)^{1/2}.
$$

(b)  For $1\le\gamma\le 2,$ let
$$
\phi_{\gamma}(\xi):= \xi_3 \left|\frac{\xi_1^2+\xi_2^2-\xi_3^2}{\xi_3^2}\right|^\gamma, \qquad \xi\notin\mathbf \Ga.
$$
Then $\phi_{\gamma}$ satisfies the (SMD)-condition (cf. Example \ref{ExSMD1}).  And:
\smallskip
 \begin{itemize}
\item For $\ga=1,$ or $\ga=2,$ $\phi_\ga$ is a classical phase, so that (a) applies. Moreover, for $R\le \la <R^2,$  so that $\la^{-1/2}>R^{-1},$ we see that $N(\la,R)=(\la/R)^{1/2},$ and thus our main theorem implies that 
$$
\| T^\la_R f\|_{L^4\to L^4}\le C_\epsilon R^\epsilon  \Big(\frac{\la}{R}\Big)^{1/4}.
$$
\item  If  $1<\ga<2,$ then the  phase $\phi_\ga$ is singular at the light cone and does no longer belong to $\F^{\frac 12, \frac 12,1},$ but $\phi_\ga\in \F^{1, \frac 12,\ga}$  (compare also with Example \ref{mainex}).

For  such singular phases, we shall therefore rather be interested in FIO-cone multipliers $m_R^\lambda(\xi)=a_R(\xi)m_R(\xi) e^{ -i\lambda\phi(\xi)},$ where $m_R$ localizes to truncated conic regions on which $1/{(2R)}<\pm F(\xi)<1/R,$ whose distance to the light cone $\mathbf\Ga$ is  thus comparable to $1/R$ (compare Section \ref{sec:phases}).

 \end{itemize}
\medskip

\medskip

Let us compare our  estimate \eqref{FIOconest} with the estimate one would get from a direct adaption  of the method of Seeger, Sogge and Stein from \cite{SSS}: 
this would lead to an estimate 	$\| T^\la_R \|_{L^\infty\to L^\infty}\lesssim R^{1/2} N(\la,R),$ since the operator norm of $ T^\la_R$
on $L^\infty$ would be bounded,  up to a factor, by the number of L-boxes needed to cover $\Ga_R$ (cf. Section \ref{mainlpeqproof}). By interpolation with the trivial $L^2$-estimate $\| T^\la_R \|_{L^2\to L^2}\lesssim 1$ it would lead to the estimate 
\begin{equation}\label{SSSest}
\| T^\la_R \|_{L^4\to L^4}\lesssim R^{1/4} N(\la,R)^{1/2} ,
\end{equation}
which is worse essentially by the factor $R^{1/4}$ than our estimate \eqref{FIOconest}. 
\smallskip

We remark  also that if  $\phi=0$ is the trivial phase, then we can choose as L-boxes $\vth$ the caps $\theta,$ so that $N(\la,R)=1,$ and we retrieve the cone multiplier  estimate \eqref{coneMult}. What ultimately allows for the improvement in our  estimate \eqref{FIOconest}  are the curvature properties of the cone. 

\medskip

\noi{\bf An application to maximal averages.} As an  important application, we will show in Section \ref{sec:maxop} how  our FIO-cone multiplier estimates allow  to prove a  ``geometric conjecture'' on maximal  averages along smooth analytic surfaces  in $\RR^3$ for a prototypical  surface from a small class of ``exceptional'' surfaces, for which this conjecture had remained open (cf. \cite{BIM}). Example (b), which has also been a guiding  examples for us in developing our theory of FIO-cone multipliers, and a related phase (cf. Examples \ref{mainex}) will play a crucial role here.

We are confident that extensions of our theory will unable us to prove this geometric conjecture in full.
 
\medskip
\noi {\bf Conjecture in higher dimension.} There are natural analogues of our classes phase functions  and the associated  FIO-cone multipliers in higher dimension $d\ge 4.$ In view of the Bochner-Riesz conjecture for $\RR^d$ and the related cone multiplier conjecture, we conjecture that analogues of estimate  \eqref{FIOconest} should hold on $L^{2\frac{d-1}{d-2}}(\RR^d).$

\subsection{On the proof strategy}
To better understand the challenges to be  met when trying to prove Theorem \ref{FIOthm}, let us first consider the special case where the Fourier transform of $f$ is supported in one single sector $\theta$ of angular width $R^{-1/2}.$ Since this sector can be covered by about $N(\la,R)$ L-boxes $\vth,$ the same arguments based on an adaptation of the  method of Seeger, Sogge and Stein shows that then we can improve the estimate \eqref{SSSest} to the estimate 
\begin{equation}\label{SSSest2}
\| T^\la_R \|_{L^4\to L^4}\lesssim N(\la,R)^{1/2},
\end{equation}
so that estimate \eqref{FIOconest} does hold in this special case. 

For a general function $f$ which is  Fourier supported in $\Ga_R,$ to substantiate our philosophy that still only the number of L-boxes $\vth$ within one single cap $\theta$ should count,  we shall seek to "disentangle" the caps $\theta$ by exploiting some kind of ``almost orthogonality'' property of the contributions by different caps $\theta.$
Tools that might be used for this purpose are square function estimates, such as \eqref{squarest}, or  decoupling estimates \cite{BD}, \cite{DGW}.  
\smallskip

Let us briefly discuss some connections between those. Note first that  Minkowski's inequality shows that the square function  estimate \eqref{squarest}  is stronger  than the $\ell^2$-decoupling estimate
\begin{equation}\label{decouple}
	\|f\|_4 \leq C_\e R^\e   \Big(\sum_{\theta} \|f_\theta\|_4^2\Big)^{\frac12},
\end{equation}
which was established by  Bourgain and Demeter \cite{BD} before the cone multiplier conjecture was solved, and actually not only on $L^4$, but even on the  $L^p$-range  $2\leq p\leq 6$. Decoupling estimates had  already been suggested by Bourgain, and Wolff \cite{W},  who had also observed that the  $\ell^4$-decoupling estimate  
\begin{equation}\label{4decouple}
	\|f\|_4 \leq C_\e R^{\frac18+\e}   \Big(\sum_{\theta} \|f_\theta\|_4^4\Big)^{\frac14},
\end{equation}
which follows easily  from \eqref{decouple}  by Hölder's inequality, is essentially sharp. 

\smallskip
However, neither  of these decoupling estimates turned out to be  strong enough for proving estimate \eqref{FIOconest},
 not even "small cap decoupling" or "small cap square function" estimates for smaller caps (such as the  L-boxes $\vth$) than the ``natural''  caps $\theta$  seemed to suffice (cf.  papers by Demeter, Guth and Wang \cite{DGW} and Gan \cite{Gan}).
 \color{black}
 
\smallskip
Also the square function estimate \eqref{squarest}  as such was not  strong enough to prove Theorem \ref{FIOthm}. However, what eventually turned out to be one of the key tools for us are certain $L^2$-based semi-norms  which had been introduced in \cite{GWZ} and which we  shall  denote as Guth-Wang-Zhang semi-norms (cf. Subsection \ref{GWZ}). These semi-norms are somehow distilling what becomes  of the $L^4$-norm of  a function which is frequency supported in $\Ga_R$ after exploiting all  orthogonality properties related to the geometry of the light cone, and are well   amenable to the induction on scales arguments used in the proof of  \eqref{squarest} in \cite{GWZ}.

\medskip

 {\bf Notation and  Conventions:}  Throughout this article, we shall use the ``variable constant'' notation, i.e.,  many constants appearing in the course of our arguments, often  denoted by  $C,$  will typically have different values  at different lines.  Moreover, we shall use  in a standard way symbols such as  $\sim, \lesssim$ or $\ll$ in order to avoid  writing down constants in inequalities. 
 
By $\chi_0 $ and $\chi_1$ we shall  typically denote  smooth cut-off functions on $\RR^n$ with   small compact supports, where $\chi_1(\xi) $ vanishes near the origin and is identically 1 on a neighborhood of the unit sphere $|\xi|=1,$ whereas 
$\chi_0$ is identically $1$ on a small neighborhood of the origin. 
  These cut-off functions may also vary from line to line, and may  in some instances, where several of  such functions of different variables appear within the same formula, even   designate different functions.
  
  We shall consider vectors in the space $\RR^n$ as row vectors and denote them by latin letters  such as $x,y,\dots.$ It will be convenient to consider vectors in the dual space as column vectors and denote them by greek letters, such as $\xi, \eta, \dots.$ Identifying otherwise as usually the dual space with $\RR^n,$ the  Fourier transform of a function $f$ on $\RR^n$ can be written as 
$$
\hat f(\xi):=\int_{\RR^n} f(x) e^{-i x\xi} dx, \qquad \xi\in \RR^n.
$$
Occasionally we shall denote the dual pairing $x\xi$ also by $\langle x, \xi\ra.$
The  gradient $\nabla f$ of a function $f=f(\xi) $ will be considered as a row vector.


\section{Classes of phase functions adapted to the light cone}
\label{sec:phases}

Denote by 
$$
\mathbf \Gamma:=\{\xi\in\RR^3: \xi_1^2+\xi_2^2=\xi_3^2\}
$$
 the (boundary of the) light cone in $\R^3,$  and by $\RR^3_\times:=\{\xi\in\RR^3: \xi_3\ne 0\}.$ Note that, by introducing the 0-homogeneous ``level function''
$$
F(\xi):=\frac{\xi_1^2+\xi^2_2}{\xi_3^2}-1
$$
on $\RR^3_\times,$  we may write $\mathbf \Gamma\setminus\{0\}=\{\xi\in\RR^3_\times: F(\xi)=0\}.$
 
For any parameter $R\geq R_0$ (where $R_0\gg 1$ will denote a sufficiently large constant) we introduce the following associated conic open sets of  aperture $\sim 1/R:$
 \begin{eqnarray*}
\mathbf \Gamma_R:&=&\{\xi\in\R^3_\times:  |F(\xi)|<\frac{1}R\}, \\
\mathbf \Gamma^\pm_R&:=&\{\xi\in\R^3_\times: \frac{1}{2R}<\pm F(\xi)<\frac{1}R\}.
 \end{eqnarray*}
The sets $\mathbf \Ga^-_R$ lie inside the light cone, and the sets $\mathbf \Ga^+_R$ outside.
We will mainly restrict our attention to the following shells, i.e., truncated subsets of those conic sets, on which $\xi_3\sim 1:$
 \begin{eqnarray*}
\Gamma_R&:=&\{\xi\in\R^3:  |F(\xi)|<\frac{1}R\ , \frac 12<\xi_3<2\}, \\
\Gamma^\pm_R&:=&\{\xi\in\R^3: \frac{1}{2R}<\pm F(\xi)<\frac{1}R, \  \frac 12<\xi_3<2\}.
\end{eqnarray*}
Note that $\Gamma_R$ is essentially the $1/R$-neighborhood of $\Gamma:=\{\xi\in \mathbf \Gamma: \frac 12<\xi_3<2\},$ and that 
$\Ga^\pm_R\subset \Gamma_R.$ 
\medskip

 \noi {\bf Definitions.} By $\F_{hom}=\F_{hom}(\mathbf \Ga_{R_0}\setminus\mathbf \Gamma)$ we denote the space of all smooth, real-valued phase functions $\phi$ on $\mathbf \Ga_{R_0}\setminus\mathbf \Gamma$ which are homogeneous of degree $1,$ i.e., which satisfy $\phi(r\xi)=r\phi(\xi)$ for every $r>0$ and every $\xi$ in its domain of definition.

In a similar way, we define the spaces  $\F_{hom}(\mathbf \Ga_R)$ and $\F_{hom}(\mathbf \Ga^\pm_R)$  as the space of  all smooth, real-valued phase functions $\phi$ on $\mathbf \Ga_R$ and  $\mathbf \Gamma^\pm_R,$ respectively, which are homogeneous of degree 1.

If $\phi\in\F_{hom},$ then obviously  $\phi\vert_{\mathbf \Ga^\pm_R}\in\F_{hom}(\mathbf \Ga^\pm_R).$

\smallskip

\medskip

We shall devise various classes of sub-families of these spaces of phase functions by requiring suitable  controls on their derivatives in the directions of  the outer unit normals and the  ``horizontal'' derivatives with respect to the cone. Due to the homogeneity of our phases, it will suffice to postulate such controls on derivatives on the truncated conic sets where $\xi_3\sim 1.$ 
We first need a bit of notation.
\smallskip

For $n,k\in \NN,$ we shall denote by   $\NN^n_{\ge k}$  the set of all multi-indices $\al=(\al_1,\dots,\al_n)\in\NN^n$ of length $|\al|:= \al_1+\cdots+\al_n\ge k.$ 
\smallskip

\noi{\bf Admissible frames:}
Given any point $\xi_0\in \Gamma_{R_0},$  there is a unique radially symmetric (with respect to the $\xi_3$-axis) cone $\mathbf \Gamma(\xi_0)$ containing $\xi_0,$ namely the cone given by
$$
\mathbf \Gamma(\xi_0):=\{F(\xi)=F(\xi_0)\}
$$ 
away from the origin.
We then associate to $\xi_0$ the following {\it orthonormal frame at $\xi_0$}: 
\begin{equation}\label{frame}
	E_1=n(\xi_0),\quad E_2=t(\xi_0),\quad E_3=\xi_0/|\xi_0|,
\end{equation}
where  $n(\xi_0)=\trans\nabla F(\xi_0)/|\trans\nabla F(\xi_0)|$ is the outer unit normal vector to the cone $\mathbf \Gamma(\xi_0)$ at $\xi_0,$ and $t(\xi_0)$ the ``horizontal'' (i.e., lying in the plane $(\xi_1,\xi_2)$-plane) unit tangent vector to $\mathbf \Gamma(\xi_0)$ at $\xi_0$ which is oriented in such a way that the frame $E_1, E_2,E_3$ is  positively oriented. 
With a slight abuse of notation, we shall also call the corresponding orthogonal matrix  
$$
E=E_{\xi_0}=\left(E_1|E_2|E_3\right),
$$
whose columns are given by this frame, the {\it orthonormal frame at $\xi_0.$}

Note  that, by homogeneity, this frame is constant along the ray through $\xi_0,$ i.e., $E_{r\xi_0}=E_{\xi_0}$ for every $r>0.$

\smallskip

For instance, if $\xi_0=(1,0,1)$, then
$$
E_1=\frac1{\sqrt{2}}\trans(1,0,-1),\quad E_2=\trans(0,1,0),\quad E_3=\frac1{\sqrt{2}}\trans(1,0,1).
$$

It will sometimes also be useful to work with slightly more general frames: any orthogonal frame of the form
 $$
 \tilde E_1=s_1E_1,\ \tilde E_2=s_2E_2,\ \tilde E_3=s_3E_3, \quad\text{with } 1/10\le s_1,s_2,s_3\le 10,
 $$ 
 with associated matrix 
 $$
\tilde E=\tilde E_{\xi_0}=\left(\tilde E_1|\tilde E_2|\tilde E_3\right),
$$
 will be called an {\it admissible frame at $\xi_0.$}

\medskip
 \noindent  {\bf Families of constants:} By $\Sigma$ we shall denote the set of all families $\C=\{C_\alpha\}_{\alpha\in \NN^2_{\ge 2}}$ of  constants $C_\alpha\ge 0.$ 
Suppose  
$\C=\{C_\alpha\}_{\alpha\in \NN^2_{\ge 2}},\C'=\{C'_\alpha\}_{\alpha\in \NN^2_{\ge 2}}\in \Sigma,$ and $r\ge 0.$
We can then partially order $\Sigma$ by writing 
$\C\le\C',$ if $C_\al\le \C'_\al$ for all $\al,$  and define 
$$
\C+\C':=\{C_\alpha+C'_\al\}_{\alpha\in \NN^2_{\ge 2}}\quad \text {and} \quad r\C:=\{rC_\alpha\}_{\alpha\in \NN^2_{\ge 2}}.
$$
In this way,   $\Sigma$ is endowed with the structure of an ordered, additive  ``conic‘‘ semigroup.

\medskip
\noi{\bf The parameters $\ka_1,\ka_2$ and $\ga$:} Our families of phase functions  will be parametrized by means of parameters $\kappa_1,\kappa_2\ge 0$ and $\gamma> 0$ which  will always be assumed to satisfy 
\begin{equation}\label{param}
\kappa_2\geq \frac12\quad \text{and}\quad |\kappa_1-\kappa_2|\leq \frac 12.
\end{equation}

\medskip
\noi{\bf Corresponding classes of phase functions:}  Let  $\C=\{C_\alpha\}_{\alpha\in \NN^2_{\ge 2}}\in \Sigma$  and parameters $\ka_1,\ga_2$ and $\ga$ be given as above. In the sequel,  $\mathbf\Ga_R^\circ$ and $\Ga_R^\circ$ shall  often stand for either  $\mathbf \Ga_R$  and   $\Ga_R,$ or for $\mathbf \Ga^\pm_R$ and $ \Ga^\pm_R.$ 
\smallskip

We then denote by
 $\F^{\kappa_1,\kappa_2,\gamma,\C}(\mathbf \Ga^\circ_R)$
  the subfamily of all phase functions $\phi$ in $\F_{hom}(\mathbf  \Gamma^\circ_R)$ satisfying  the following bounds on their derivatives:
\begin{equation}\label{deriv}
	 |\partial_{n(\xi_0)}^{\alpha_1}\partial_{t(\xi_0)}^{\alpha_2} \phi(\xi_0)|\leq C_{\alpha} R^{\alpha_1\kappa_1+\alpha_2\kappa_2-\gamma}
	 \quad \text{for all}\  \xi_0\in\Gamma_R^\circ, \ \alpha_1+\alpha_2\geq 2.
\end{equation}

 Observe that if $\phi_1\in  \F^{\kappa_1,\kappa_2,\gamma,\C}(\mathbf \Ga^\circ_R),$ $\phi_2\in  \F^{\kappa_1,\kappa_2,\gamma,\C'}(\mathbf \Ga^\circ_R)$
 and $a_1,a_2\in \RR,$ then 
 \begin{equation}\label{sumphi}
  a_1\phi_1+a_2\phi_2\in  \F^{\kappa_1,\kappa_2,\gamma,|a_1|\C+|a_2|\C'}(\mathbf \Ga^\circ_R).
\end{equation}

 \medskip

By $\F^{\kappa_1,\kappa_2,\gamma,\C}=\F^{\kappa_1,\kappa_2,\gamma,\C}(\mathbf \Ga_{R_0}\setminus\mathbf \Gamma)$ we denote the family   of all phase functions $\phi\in\F_{hom}$ such that $\phi\vert_{\mathbf \Ga^\pm_R}\in\F^{\kappa_1,\kappa_2,\gamma,\C}(\mathbf \Ga^\pm_R)$ for all $R\ge R_0,$ and both choices of sign $+,-.$
  
  Finally, we put 
  $$
  \F^{\kappa_1,\kappa_2,\gamma}:=\bigcup\limits_{\C\in \Sigma}\F^{\kappa_1,\kappa_2,\gamma,\C}.
  $$ 
  By \eqref{sumphi},   $\F^{\kappa_1,\kappa_2,\gamma}$ is a real vector space.
  
  \smallskip
  Note that trivially 
\begin{equation}\label{inclusion}
 \F^{\kappa_1,\kappa_2,\gamma^\pm,\C}(\mathbf \Ga^\circ_R)\subset \F^{\ka'_1,\ka'_2,\gamma',\C'}(\mathbf \Ga^\circ_R),\quad  \text{if}\  \ka_1\le\ka'_1, \ka_2\le\ka'_2,
  \ga\ge \ga' \  \text{and}\  \C\le\C',
\end{equation}
and that analogous inclusions hold for  the  families  of phase functions $\F^{\kappa_1,\kappa_2,\gamma,\C}$ and $\F^{\kappa_1,\kappa_2,\gamma}.$

\medskip
\noi{\bf Classical phases:} By a {\it classical phase function} we  shall mean any  real-valued  and 1-homogeneous   function $\phi$ defined on an open conic neighborhood  of $\overline{{\bf\Gamma}_{R_0}}\setminus\{0\}$ which is smooth away from the origin. The space  of all restrictions of such  phases to  ${\bf\Gamma}_{R_0}$ will be denoted by 
$\F_{class}.$ 
\smallskip

By $\F_{class}^{\ka_1,\ka_2,\ga,\C}=\F_{class}^{\kappa_1,\kappa_2,\gamma,\C}(\mathbf \Ga_{R_0})$ we denote the family   of all phases $\phi\in\F_{class}$ such that $\phi\vert_{\mathbf \Ga_R}\in\F^{\kappa_1,\kappa_2,\gamma,\C}(\mathbf \Ga_R)$ for all $R\ge R_0,$ and correspondingly we put
  $$
  \F_{class}^{\kappa_1,\kappa_2,\gamma}:=\bigcup\limits_{\C\in \Sigma}\F_{class}^{\kappa_1,\kappa_2,\gamma,\C}.
  $$ 
 $\F_{class}^{\kappa_1,\kappa_2,\gamma}$ is then a real subspace of $\F_{class}.$ 
\medskip

We will explore such classes of phase functions and give examples in Section \ref{sec:stable}.
\medskip

\begin{remark}\label{unibo}
If $\ka_1\ge 1,$ then one can easily construct non-trivial ``global'' phase functions $\phi\in \F^{\kappa_1,\kappa_2,\gamma}$ 
(cf. Remark \ref{globalphi}). By definition, for such a phase $\phi$ there  exists a family of constants  $\C\in\Sigma$  so that  
$\phi\vert_{\mathbf \Ga^\pm_R}\in\F^{\kappa_1,\kappa_2,\gamma,\C}(\mathbf \Ga^\pm_R)$ for  all $R\ge R_0$ and both choices of sign $+,-.$

If $\ka_1<1,$ then for certain choices of parameters $\ka_1,\ka_2,\ga,$ classical phases still  provide non-trivial examples  of phases in $\F^{\kappa_1,\kappa_2,\gamma}$ (cf. Example \ref{Fclass}), but it seems unclear whether  $\F^{\kappa_1,\kappa_2,\gamma}\ne \{0\}$ in general.

\smallskip
On the other hand, one can always find a family $\{\phi_R\}_{R\ge R_0}$ of non-trivial phases 
$\phi_R\in\F^{\kappa_1,\kappa_2,\gamma,\C}(\mathbf \Ga^\pm_R)$ for some suitable family of constants $\C$ (cf. Remark \ref {existphi}).
\smallskip

Our ultimate goal will therefore be to prove ``uniform'' $L^p$-bounds  for all FIO-cone multipliers $T^\la_R$ associated to all phases $\phi_R\in F^{\kappa_1,\kappa_2,\gamma,\C}(\mathbf \Ga^\circ_R),$ of the form
\begin{equation}\label{unibd3}
\|T^\la_R\|_{L^p\to L^p}\le C(p,\kappa_1,\kappa_2,\gamma,\C)\,M_p(\la,R) \qquad \text{for all} \ R\ge R_0,
\end{equation}
where the constant $C(p,\kappa_1,\kappa_2,\gamma,\C)$ may depend on $p,$  on  the constants  $C_\al\in \C$ (actually only a finite number of them) and the parameters $\kappa_1,\kappa_2,\gamma,$ but neither on the phase $\phi_R,$ nor on $R.$ 

For any phase $\phi\in \F^{\kappa_1,\kappa_2,\gamma},$ this will in particular lead to ``uniform'' estimates of the form \eqref{unibd3} for the family of phases $\phi_R, R\ge R_0, $ defined by 
$\phi_R:=\phi\vert_{\mathbf \Ga^\pm_R}.$ 
\end{remark}

\begin{remark}\label{subcone}
We have been and shall  usually assume that our phase functions $\phi$ (respectively $\phi_R$) are defined on the whole of $\mathbf \Ga_{R_0}\setminus\mathbf \Gamma$ (respectively on $\mathbf \Ga^\circ_R$). However, the entire theory that we  develop  applies as well when  the amplitude and the phase $\phi$ are defined  in an angular sector of these sets only (cf. Example \ref{mainex}).
 \end{remark}

\begin{remark}\label{GaRGapm}
By means of linear scalings of the form  $(\xi_1,\xi_2,\xi_3)\mapsto  (r\xi_1,r\xi_2,\xi_3),$ with suitable scaling factors $r=1+\mathcal O(R^{-1}),$ one can obviously  essentially switch between  phases in $\F^{\kappa_1,\kappa_2,\gamma,\C}(\mathbf \Ga_R),$ and phases in $\F^{\kappa_1,\kappa_2,\gamma,\C}(\mathbf \Ga^\pm_R).$ Therefore,  from here on we shall restrict many considerations to phase functions  defined on  the regions $\mathbf \Ga^\pm_R $;  the analogous results will then hold for corresponding phase functions on $\mathbf \Ga_R$ as well.
\end{remark}

\section{Stability results and examples}
\label{sec:stable}
\color{black}

\subsection{Stability results for the defining estimates of our classes of phases}\label{stabilityres}

In the subsequent lemmata,  given a family $\C\in \Sigma,$  various new, but related  families $\C'$ of $R$- independent constants will come up, and in view of   Remark \ref{unibo},   we shall put some emphasis on understanding how  the new constants in these estimates will be controlled by the constants from our initial family $\C.$ 

\smallskip

 We will see that estimate \eqref{deriv} in the definition of the class $\F^{\kappa_1,\kappa_2,\gamma,\C}(\mathbf \Ga^\pm_R)$ is quite flexible with respect to small changes of the directions for the derivatives.
 \smallskip

To this end, let us compare the orthonormal frames $E=E_\xi$ and $\check E=E_{\check \xi}$ at two points  $\xi,\check\xi$ which are separated by a certain angle of $\alpha.$ More precisely, we define the  {\it sector of angular width (or aperture) }$\alpha$ around $\xi\in\Gamma_R$ by

\begin{equation}\label{theta}
	\theta_{\alpha}(\xi):=\{\check\xi\in\Gamma_R:|\angle((\check\xi_1,\check\xi_2),(\xi_1,\xi_2))|\leq c_0\alpha \},
\end{equation}
where the constant   $0<c_0\ll1$ will be assumed to be sufficiently small, but fixed. 

\begin{lem}\label{EEcheck}
Let $\alpha\in[\Rwurz,1]$, $\xi\in\Gamma_R$, and $\check\xi\in\theta_{\alpha}(\xi)$, and let $E=E_\xi$ and $\check E=E_{\check \xi}$ be the  orthonormal frames at $\xi$ and $\check \xi,$ respectively. Then we have
$$
|\langle E_i, \check E_j\rangle|\lesssim \alpha^{|i-j|}
$$
for all $i,j=1,2,3$. Furthermore, we have
\begin{equation}\label{E2}
	|1-\langle E_2,\check E_2\rangle| \lesssim \alpha^2.
\end{equation}
\end{lem}
\begin{proof}
By homogeneity, we may assume without loss of generality that $\xi_3=\check \xi_3.$ Moreover, after rotation around the axis of $\bGa,$ we may assume that  $E_3=\xi/|\xi|$ and $E_1=n(\xi)$ lie  in the plane $\xi_2=0,$ say. Then, since $\xi\in\Gamma_R,$
$$
E_3=s\trans(a,0,1),\qquad E_1=s\trans(1,0,-a),
$$
where $a=1+\sigma_1$, with $|\sigma_1|\lesssim R^{-1}\leq\alpha^2,$ and $s=(1+a^2)^{-1/2}\sim 1.$ Clearly, by orthogonality, 
 $E_2=\trans(0,1,0)$, and thus
$$
\trans E= \left(\begin{array}{ccc} s & 0 & -sa \\ 0 & 1 & 0 \\ sa & 0 & s \end{array}\right).
$$
Since $\check\xi\in\theta_{\alpha}(\xi)$, we have 
 $$
 \check E_3 = \check\xi/| \check\xi|= s_3\trans(b,c,1),
 $$
with $b^2+c^2=1+\sigma_2$, $|\sigma_2|\lesssim R^{-1}\leq\alpha^2$, $|c|\lesssim c_0 \alpha,$ and $s_3=(b^2+c^2+1)^{-1/2}\sim 1.$ Note that by choosing $c_0$ sufficiently small, we  may also assume that  $b=1+\sigma_3$ for some $|\sigma_3|\lesssim \alpha^{2}.$ Then we compute that 
$$
\check E_1 = n(\check\xi) = s_1 \,\trans(b,c,-(b^2+c^2)),\qquad \check E_2 = t(\check\xi)=s_2\trans(-c,b,0),
$$
with $s_1\sim 1$, $s_2=(b^2+c^2)^{-1/2}=1+\sigma_4$, $|\sigma_4|\lesssim \alpha^2$.
To simplify the computation, observe that $\check E_j=s_j\tilde E_j,j=1,2,3,$ where 
$$
\tilde E_1=\trans(b,c,-(b^2+c^2)), \tilde E_2=\trans(-c,b,0), \tilde E_3=\trans(b,c,1)
$$
is an admissible frame at $\check \xi.$

Putting $\tilde E=\left(\tilde E_1|\tilde E_2|\tilde E_3\right),$ we observe that 
\begin{equation}\label{angles}
|\langle E_i, \check E_j\rangle|= |(\trans E \check E)_{i,j}|\sim |(\trans E\tilde E)_{i,j}|.
\end{equation}
Moreover, one computes that 
$$
\trans E\tilde E = s\left(\begin{array}{ccc} b+a(b^2+c^2) & -c & b-a \\ c/s & b/s & c/s \\ ab-(b^2+c^2) & -ac & 1+ab \end{array}\right).
$$
Note that the entries $c/s,-c,-ac$ are all of order at most $\alpha$, since $c$ is. Furthermore, 
$$
b-a = 1+\sigma_3-1-\sigma_1 = \sigma_3-\sigma_1 =\mathcal{O}(\alpha^{2})
$$
and 
$$
ab-(b^2+c^2 )= 1+\sigma_1+\sigma_3+\sigma_1\sigma_3-1-\sigma_2 = \mathcal{O}(\alpha^{2}).
$$
In  combination with \eqref{angles}, these estimates  imply the estimates claimed in the lemma.\\
Furthermore,
$$
 \langle E_2,\check E_2\rangle = s_2 \langle E_2,\tilde E_2\rangle =s_2b =(1+\sigma_4)(1+\sigma_3) =1+\mathcal{O}(\alpha^{2}).
$$
\end{proof}

In the definition of $\F^{\kappa_1,\kappa_2,\gamma,\C}(\mathbf \Ga^\pm_R),$  the condition \eqref{deriv}
\begin{equation*}
|\partial_{n(\xi_0)}^{\alpha_1}\partial_{t(\xi_0)}^{\alpha_2} \phi(\xi_0)|\leq C_\alpha R^{\alpha_1\kappa_1+\alpha_2\kappa_2-\gamma}
\end{equation*}
looks as if it would lack control of the derivatives with respects to $E_3.$  This is, however, automatic for homogeneous phase functions. In the subsequent lemmata, we again always assume that $R\ge R_0.$ The following notation will be useful:
 if $\C=\{C_\alpha\}_{\alpha\in \NN^2_{\ge 2}}\in \Sigma,$  then we put 
 $$
 m_\C(l):=\max\{ C_\al: \al\in \NN^2_{\ge 2}, |\al|\le l \}, \quad l\in \NN_{\ge2}.
$$

\begin{lem}\label{addkappa3}
Assume that $\kappa_3\geq 0.$ Then the following hold true:
\begin{itemize}
\item [(i)] For every family $\C=\{C_\alpha\}_{\alpha\in \NN^2_{\ge 2}}\in \Sigma$  there exists a family $C'_\beta\ge 0, \beta\in \NN^3_{\ge 2},$ satisfying 
\begin{equation}\label{cc'}
C'_{\beta}\le |\beta|!\, m_\C(\beta_1+\beta_2), 
\end{equation}
so that for every $\phi\in\F^{\kappa_1,\kappa_2,\gamma,\C}(\mathbf \Gamma^\pm_R)$  we have 
\begin{equation}\label{deriv2}
	|\partial_{n(\xi_0)}^{\beta_1}\partial_{t(\xi_0)}^{\beta_2} \partial_{\xi_0}^{\beta_3}\phi(\xi_0)|\leq C'_{\beta} R^{\beta_1\kappa_1+\beta_2\kappa_2+\beta_3\kappa_3-\gamma},\qquad \text{for every} \ \xi_0\in \Gamma^\pm_R, \ \beta\in \NN^3_{\ge 2}.
\end{equation}

\item  [(ii)]   Conversely, suppose that  $\phi$ is in  $\F_{hom}(\mathbf \Gamma^\pm_R),$  and that $\phi$ satisfies the estimates \eqref{deriv2}. Then $\phi\in\F^{\kappa_1,\kappa_2,\gamma,\C}(\mathbf \Gamma^\pm_R),$   with $\C:=\{C'_{(\alpha_1,\alpha_2,0)}\}_{\alpha\in \NN^2_{\ge 2}}.$
\end{itemize}
\end{lem}

\begin{proof}
(ii) is trivial - just choose $\beta_3=0$ in \eqref{deriv2}. 

\smallskip

To prove (i),  it will be useful to work with the admissible frame at $\xi_0$ given by $\tilde E_1=n(\xi_0), \tilde E_2=t(\xi_0), \tilde E_3=\xi_0.$  For $\beta=(\beta_1,\beta_2,\beta_3)\in \NN^3,$  we shall  use the short-hand notation  $\partial_{\tilde E}^\beta=\partial_{\tilde E_1}^{\beta_1}\partial_{\tilde E_2}^{\beta_2}\partial_{\tilde E_3}^{\beta_3}.$ Moreover, by $e_1,e_2,e_3$ we shall  denote the canonical basis of $\RR^3.$
\smallskip

We shall  prove \eqref{deriv2} and \eqref{cc'} by induction on the size of $\beta_3$. So, let $\beta\in\NN^3_{\ge 2}.$ In the base case $\beta_3=0$, \eqref{deriv2} is immediate from \eqref{deriv}.
\smallskip

Assume next that  $\beta_3\geq 1.$ Let us then write $\beta=\beta'+e_3,$  and put $m:=|\beta'|=\beta_1+\beta_2+\beta_3-1.$
 Since $\phi$ is homogeneous of degree 1, $\psi(\xi):=\partial_E^{\beta'}\phi(\xi)$ is homogeneous of degree $1-m$. By Euler's homogeneity relation, we thus have
$$
 \nabla\psi(\xi)\cdot\xi = (1-m)\psi(\xi).
 $$
Note that $\tilde E_3=\xi_0,$ and thus
 \begin{equation}\label{16sep1345}
	 \partial_{\tilde E}^\beta\phi(\xi_0) = \partial_{\tilde E_3}\psi(\xi_0) = \nabla\psi(\xi_0)\cdot \tilde E_3 = (1-m)\psi(\xi_0) = (1-m)\partial_{\tilde E}^{\beta'}\phi(\xi_0) .
 \end{equation}
 
In the case where $|\beta'|=m\geq 2$, we have $\beta_1'=\beta_1$, $\beta_2'=\beta_2,$ and by   the induction hypothesis we can assume that 
$$
|\partial_{\tilde E}^{\beta'}\phi(\xi_0) |\leq C'_{\beta'} R^{\beta_1\kappa_1+\beta_2\kappa_2+\beta'_3\kappa_3-\gamma}.
$$
Since  $\kappa_3\geq 0$, we conclude by \eqref{16sep1345} that 
$$
 |\partial_{\tilde E}^\beta\phi(\xi_0)|\le  (m-1)C'_{\beta'} R^{\beta_1\kappa_1+\beta_2\kappa_2+\beta_3\kappa_3-\gamma}.
$$
In particular, we can choose $C'_\beta:=|\beta|C'_{\beta'}.$ 
 \smallskip

In the case $m\le 1$, we cannot use the inductive hypothesis anymore, but since $m=|\beta'|=|\beta|-1\geq 1$, we actually have $m=1$ and thus \eqref{16sep1345} reads $ \partial_{\tilde E}^\beta\phi(\xi_0)=0,$ so that we may trivially again choose $C'_\beta:=|\beta|C'_{\beta'}.$ 

\smallskip
The induction hypothesis also implies that 
$$
C'_\beta:=|\beta|C'_{\beta'}\le |\beta||\beta'|!\, m_\C(\beta_1+ \beta_2)=|\beta|!\,m_\C(\beta_1+ \beta_2),
$$ 
which concludes the proof.
\end{proof}

In the previous lemma,   we could have just taken $\ka_3=0.$ However,  a similar statement is true even if we take derivatives in slightly different directions, given by  the orthonormal frame associated to another point in a sector of angular width $R^{-1/2}$ around $\xi_0,$  but then we need to increase $\kappa_3:$

\begin{lem}\label{addkappa3n}  
Assume that $\kappa_3\geq \kappa_2-\frac12.$ 
Then the following hold true:
\begin{itemize}
\item [(i)] There exists a  constant $A\ge1$ such that if we associate to  any family $\C=\{C_\alpha\}_{\alpha\in \NN^2_{\ge 2}}\in \Sigma$ the family of constants 
\begin{equation}\label{cc'p}
C'_{\beta}:= A^{\beta_3} |\beta|!\, m_\C(|\beta|), \qquad \beta\in \NN^3_{\ge 2},
 \end{equation}
then the following holds true: if $\xi_0\in\Gamma^\pm_R$ and $\check\xi\in\theta_{\Rwurz}(\xi_0),$ and if $\phi\in\F_{hom}(\mathbf \Gamma^\pm_R)$ satisfies the estimates
\begin{equation}\label{deriv3}
	|\partial_{n(\check\xi)}^{\al_1}\partial_{t(\check\xi)}^{\al_2} \phi(\xi_0)|\leq C_{\alpha} R^{\alpha_1\kappa_1+\alpha_2\kappa_2-\gamma}
	\qquad \text{for every}\ \al\in\NN^2_{\ge 2},
\end{equation}
then $\phi$ satisfies the following estimates as well:
\begin{equation}\label{deriv4}
	|\partial_{n(\check\xi)}^{\beta_1}\partial_{t(\check\xi)}^{\beta_2} \partial_{\check\xi}^{\beta_3}\phi(\xi_0)|\leq C'_{\beta} R^{\beta_1\kappa_1+\beta_2\kappa_2+\beta_3\kappa_3-\gamma}\qquad \text{for every}\ \beta\in\NN^3_{\ge 2}.
\end{equation}

\item [(ii)] Conversely, if $\phi$ satisfies the estimates of the form \eqref{deriv4}, with constants $\C'_\beta\ge 0,$ then $\phi$ satisfies the estimates  \eqref{deriv3},  with $\C:=\{C'_{(\alpha_1,\alpha_2,0)}\}_{\alpha\in \NN^2_{\ge 2}}.$
\end{itemize}
\end{lem}

\begin{proof} The proof of (ii) is again trivial, and the proof of (i)  will follow the same inductive scheme on the size of $\beta_3$ as in the proof of the previous lemma, whose notation we adapt.

  \smallskip
Let again $\tilde E_1=n(\xi_0), \tilde E_2=t(\xi_0), \tilde E_3=\xi_0$ denote the admissible frame at  $\xi_0$ of the previous proof, and let 
 $\check E_1=n(\check \xi), \check E_2=t(\check \xi), \check E_3=\check\xi/|\check \xi|$ be the  orthonormal frame at $\check \xi.$ 
 Then, by applying Lemma \ref{EEcheck} with $\alpha:=R^{-1/2}$, 
 
\begin{equation}\label{30sep241610}\
|\langle \tilde E_i, \check E_j\rangle|\lesssim R^{-\frac12 |i-j|},\qquad i,j=1,2,3,
\end{equation}
since $|\tilde E_3|\sim |E_3|=1.$ Given $\beta\in \NN^3_{\ge 2},$  we may  again easily reduce to the case $\beta_3\ge 1.$ Then, setting again $\beta':=\beta-e_3$ and putting here  $\psi(\xi):=\partial_{\check E}^{\beta'} \phi(\xi),$   by Euler's identity we have 
$$
\nabla\psi(\xi_0)\cdot \tilde E_3 = \nabla\psi(\xi_0)\cdot\xi_0 = (1-m)\psi(\xi_0).
$$
In the basis given by the orthonormal frame $\check E,$ we thus have
$$\sum_{j=1}^3 \langle \tilde E_3, \check E_j\rangle \partial_{\check E_j}\psi(\xi_0)  = (1-m)\psi(\xi_0).$$
From this identity, making use of \eqref{30sep241610}, we deduce that there is a constant $A\ge 1$ (depending only on the implicit constants in \eqref{30sep241610}) such that 
\begin{eqnarray*}
	|\partial_{\check E}^\beta\phi(\xi_0)| = |\partial_{\check E_3}\psi(\xi_0) | &\le& \frac A3 (|m-1| |\psi(\xi_0)| + R^{-1}|\partial_{\check E_1}\psi(\xi_0) |+R^{-1/2}| \partial_{\check E_2}\psi(\xi_0)|) \\
	&&\hskip-1.5cm= \frac A3(|m-1|| \partial_{\check E}^{\beta'} \phi(\xi_0)| + R^{-1}|\partial_{\check E}^{\beta'+e_1} \phi(\xi_0) |+R^{-1/2}| \partial_{\check E}^{\beta'+e_2} \phi(\xi_0)|).
\end{eqnarray*}
Since again either $|m-1|=0$ or $|\beta'|\geq 2$, by the induction hypothesis we may  estimate 
\begin{eqnarray*}
	|\partial_{\check E}^\beta\phi(\xi_0)| 	&\le& \frac A3R^{\beta_1\kappa_1+\beta_2\kappa_2+(\beta_3-1)\kappa_3-\gamma} (|m-1|C'_{\beta'} + C'_{\beta'+e_1}R^{-1} R^{\kappa_1} + C'_{\beta'+e_2}R^{-1/2}R^{\kappa_2} )\\
	&\leq& A \max\{|\beta|C'_{\beta'},C'_{\beta'+e_1},C'_{\beta'+e_2}\}R^{\beta_1\kappa_1+\beta_2\kappa_2+\beta_3\kappa_3-\gamma},
\end{eqnarray*}
since,  by \eqref{param}, we have $\kappa_1-1\leq \kappa_2-\frac12\leq\kappa_3$ and $\kappa_3\geq0$. 
\smallskip
And, the induction hypothesis easily implies that 
$$
A\max\{|\beta|C'_{\beta'},C'_{\beta'+e_1},C'_{\beta'+e_2}\}\le A A^{\beta_3-1} |\beta|!\, m_\C(|\beta|)=A^{\beta_3} |\beta|!\, m_\C(|\beta|).
$$
 This concludes the inductive argument.
\end{proof}
   
We now prove that we can vary the directions of the derivatives in estimates \eqref{deriv2}, \eqref{deriv4}.

Actually, the directions need not necessarily be given by the   orthonormal frame at some  point $\xi_0\in\Gamma_R$ as in \eqref{frame}, but can be chosen in more flexible ways:

\begin{lem}\label{stability} Let $\kappa_1,\kappa_2,\kappa_3\geq 0$, $\beta_1,\beta_2,\beta_3\in\NN$, and 
let $E,\check E\in \mathop{GL}(3,\R)$ be  such that, for  some constant $A\ge 0,$ 
 
 \begin{equation}\label{pert1}
 |(E^{-1}\check E)_{i,j}|\le A R^{\kappa_j-\kappa_i} \text{ for all}\ i,j=1,2,3.
\end{equation}
Then the following holds true:  if $m\in\NN,$ if $\phi$ is $m$-times differentiable at $\xi_0\in\R^3,$ and if  
\begin{equation}\label{derivlemeq}
	|\partial_E^\beta\phi(\xi_0)|\leq C_m R^{\beta_1\kappa_1+\beta_2\kappa_2+\beta_3\kappa_3-\gamma} 
	\quad \text{for every }\ \beta\  \text {with} \  |\beta|=m,
\end{equation}
then  we also have
\begin{equation}\label{17sep1125}
|\partial_{\check E}^\beta\phi(\xi_0)|\leq C'_m R^{\beta_1\kappa_1+\beta_2\kappa_2+\beta_3\kappa_3-\gamma}\quad \text{for every }\ \beta\  \text {with} \  |\beta|=m,
\end{equation}
with a constant $C'_m\le (3A)^mC_m.$ 
\end{lem}

\begin{proof}
Let us introduce the coordinates  $\eta$ which are given by $\xi=\eta_1E_1+\eta_2 E_2+\eta_3 E_3= E\eta,$ and set 
$\eta_0:=E^{-1}\xi_0.$ Then  \eqref{derivlemeq} is equivalent to
$$
|\partial_{\eta}^{\beta}(\phi\circ E)(\eta_0)|\le C_{\beta} R^{\beta_1\kappa_1+\beta_2\kappa_2+\kappa_3\beta_3-\gamma},\qquad |\beta|=m.
$$
Next, if we introduce the diagonal  matrix $D_R$ with  diagonal entries $R^{-\kappa_j}, j=1,2,3$, then these estimates, hence  \eqref{derivlemeq}, are equivalent to  
\begin{equation}\label{normalest1}
|\partial_{\eta'}^{\beta} \tilde \phi(\eta_0') )|\leq C_{\beta}R^{-\gamma},\qquad |\beta|=m,
\end{equation}
where $\tilde \phi:=\phi\circ (ED_R)$ and  $\eta'_0=:(ED_R)^{-1} \xi_0.$

Similarly, \eqref{17sep1125} is equivalent to 
\begin{equation}\label{normalest2}
|\partial_{\eta''}^{\beta}(\tilde \phi\circ[D_R^{-1}E^{-1}\check ED_R])(\eta''_0) |=|\partial_{\eta''}^{\beta}( \phi\circ[\check ED_R])(\eta''_0) )|\leq C'_{\beta}R^{-\gamma},\qquad |\beta|=m,
\end{equation}
with $\eta''_0:=(\check E D_R)^{-1}\xi_0.$

Now, observe that the $(i,j)$-th entry of  the matrix $D_R^{-1}E^{-1}\check ED_R$ is given by
$$
\trans e_iD_R^{-1}E^{-1}\check ED_Re_j =R^{\kappa_i-\kappa_j}\,\trans e_iE^{-1}\check Ee_j,
$$
if $e_1,e_2,e_3$ denotes again the canonical basis of $\RR^3.$ Thus, by \eqref{pert1}, 
$$
|(D_R^{-1}E^{-1}\check ED_R)_{i,j}|\le A, \text{ for all}\ i,j=1,2,3.
$$

The lemma now follows by chain rule.
\end{proof}

We are now in the position to show that the estimates \eqref{deriv} in the definition of the class 
$\F^{\kappa_1,\kappa_2,\gamma,\C}(\mathbf \Gamma^\pm_R)$ are stable on sectors of angular width $\Rwurz,$ in the sense that we are allowed  to replace, for instance, the orthonormal frame $E_{\xi_0}$ in them by the orthonormal frame at any other point in a sector of  angular width  $\Rwurz$ around $\xi_0.$

\begin{lem}\label{stability2}
Let $\xi_0\in \Gamma_R,$  and  let $\xi,\check\xi\in\theta_{\Rwurz}(\xi_0).$ Suppose that 
\begin{eqnarray}\label{30sep1631}
	|\partial_{n(\xi)}^{\alpha_1}\partial_{t(\xi)}^{\alpha_2} \phi(\xi_0)|\leq C_{\alpha} R^{\alpha_1\kappa_1+\alpha_2\kappa_2-\gamma}\quad\text{for all}\ \alpha\in \NN^2_{\ge 2}.
\end{eqnarray}
Then 
\begin{eqnarray}\label{30sep1632}
	|\partial_{n(\check\xi)}^{\alpha_1}\partial_{t(\check\xi)}^{\alpha_2} \phi(\xi_0)|\leq C'_{\alpha} R^{\alpha_1\kappa_1+\alpha_2\kappa_2-\gamma} \quad\text{for all}\ \alpha\in \NN^2_{\ge 2},
\end{eqnarray}
where we may bound the constants $C'_\alpha$ by $C'_{\alpha}\le A^{|\alpha|} |\alpha|!\, m_\C(|\alpha|),$ with a universal  constant $A.$
\end{lem}

\begin{proof} 
By Lemma \ref{addkappa3n}, we can superficially improve \eqref{30sep1631} to
$$
 |\partial_{n(\xi)}^{\beta_1}\partial_{t(\xi)}^{\beta_2}\partial_{\xi}^{\beta_3} \phi(\xi_0 )|\leq C'_{\beta} R^{\beta_1\kappa_1+\beta_2\kappa_2+\beta_3\kappa_3-\gamma}\quad\text{for all}\ \beta\in \NN^3_{\ge 2},
 $$
if  we choose  $\kappa_3:=\kappa_2-\frac12 \geq0$.
Let $E,\check E$ be the orthonormal frames at  $\xi$ and $\check\xi,$ respectively. 
Then, by Lemma \ref{EEcheck}, $|\langle E_i, \check E_j\rangle|\lesssim R^{-\frac12 |i-j|}.$
\smallskip

 But, by \eqref{param},  we have 
$$
|\kappa_1-\kappa_2|\leq \frac12,
$$
and by our choice of $\kappa_3$, we also have
$$
|\kappa_2-\kappa_3|=\frac12,
$$
so that 
$$
|\kappa_1-\kappa_3|\leq|\kappa_1-\kappa_2|+|\kappa_2-\kappa_3|\leq 1.
$$

In other words,
$$
\kappa_i-\kappa_j\leq |\kappa_i-\kappa_j|\leq \frac12 |i-j|
$$
for all $i,j=1,2,3$. 
Since $E$ is an orthogonal matrix, $E^{-1}=\trans E,$  and thus 
$$
|(E^{-1}\check E)_{i,j}|=   |(\trans E\check E)_{i,j}|   =	|\langle E_i, \check E_j\rangle| 
		\lesssim R^{-\frac12 |i-j|} \leq R^{\kappa_j-\kappa_i}.
$$
By Lemma \ref{stability}, we find that
$$
 |\partial_{n(\check\xi)}^{\beta_1}\partial_{t(\check\xi)}^{\beta_2}\partial_{\check\xi}^{\beta_3} \phi(\xi_0 )|\leq C''_{\beta} R^{\beta_1\kappa_1+\beta_2\kappa_2+\beta_3\kappa_3-\gamma}\quad\text{for all}\ \beta\in \NN^3_{\ge 2},
 $$
 and choosing then $\beta_3=0,$ we arrive at \eqref{30sep1632}. The claimed bounds on the constants $C'_\alpha$ easily follow from the corresponding bounds in Lemma \ref{addkappa3n} and  Lemma \ref{stability}.
\end{proof}

\noi {\bf Definition:} Let $I,J$ be two index sets and $\F$ a parameter set. Moreover, suppose $a:\F\times I\to \RR_{\ge 0},\,  b:\F\times J\to \RR_{\ge 0}$ and $g: I\to \RR_{\ge 0}, \, h: J\to \RR_{\ge 0}$ are functions such that the following holds:

 If $\C=\{C_\al\}_{\al\in I}$ is any family of constants $C_\al\ge 0,$ then  there are constants $C'_\be\ge 0, \beta\in J,$  so that the following holds:  if $\phi\in\F$ is such that 
\begin{equation}\label{control1}
|a(\phi,\al)| \le C_\al g(\al) \qquad\text{for all} \ \al\in I,
\end{equation}
then
\begin{equation}\label{control2}
|b(\phi, \beta)| \le C'_\beta h(\beta) \qquad\text{for all} \ \beta\in J.
\end{equation}
Then we say that {\it the constants $C'_\beta $ in \eqref{control2} are controlled  by the constants $C_\al$ in \eqref{control1}}, if for every $\beta\in J$ there is a constant $A_\beta\ge 0$ and a finite subset $I_\beta\subset I$ such that
$$
C'_\beta\le A_\beta \sum_{\al\in I_\beta} C_\al.
$$

\smallskip
For us, the most important  cases in Corollary \ref{stability2} are where either $\xi=\xi_0,$ or $\check\xi=\xi_0$, so that the assumption \eqref{30sep1631} respectively the conclusion \eqref{30sep1632} is precisely the defining conditions \eqref{deriv} for the class 
 $\F^{\kappa_1,\kappa_2,\gamma,\C}(\mathbf \Ga^\pm_R):$

\begin{cor}\label{stability3}
Let $\xi\in\Gamma_R^\pm $ and $\phi\in\F^{\kappa_1,\kappa_2,\gamma,\C}(\mathbf \Ga_R^\pm).$  If  $\check\xi\in\theta_{\Rwurz}(\xi),$ then 
$$
 |\partial_{n(\check\xi)}^{\alpha_1}\partial_{t(\check\xi)}^{\alpha_2} \phi(\xi)|\le C'_{\alpha} R^{\alpha_1\kappa_1+\alpha_2\kappa_2-\gamma}
  \quad\text{for all}\ \alpha\in \NN^2_{\ge 2},
 $$
 where the constants $C'_\al$ are controlled by the constants in $\C$  from estimate \eqref{deriv}. The analogous statement holds for $\Gamma_R$ in place of 
 $\Gamma^\pm_R$
\end{cor}
\begin{proof} Take $\xi=\xi_0$ in Lemma \ref{stability2}.
\end{proof}

\begin{cor}\label{suff_for_deriv}
Let  $\check \xi\in \Gamma_R$  and  $\xi\in\theta_{\Rwurz}(\check\xi),$ and assume  that 
$$
|\partial_{n(\xi)}^{\alpha_1}\partial_{t(\xi)}^{\alpha_2} \phi(\check\xi)|\leq C_{\beta} R^{\alpha_1\kappa_1+\alpha_2\kappa_2-\gamma}
 \quad\text{for all}\ \alpha\in \NN^2_{\ge 2}.
$$
Then
$$
 |\partial_{n(\check\xi)}^{\alpha_1}\partial_{t(\check\xi)}^{\alpha_2} \phi(\check\xi)|\leq C'_{\alpha} R^{\alpha_1\kappa_1+\alpha_2\kappa_2-\gamma} \quad\text{for all}\ \alpha\in \NN^2_{\ge 2},
$$
where the constants $C'_\al$ are controlled by the constants in  $\{C_\alpha\}_{\alpha\in \NN^2_{\ge 2}}.$
\end{cor}

\begin{proof} Take $\check\xi=\xi_0$ in Lemma \ref{stability2}.
\end{proof}

\subsection{Rotated and dilated light cones}\label{rotatedcones}
Sometimes it will be useful to consider the light cone in rotated coordinates.  In particular, we shall sometimes  like to work in tilted coordinates, in which  the axis of the light cone becomes the line generated by $(1,0,1)$ (see, e.g., \cite{GWZ}, and Example \ref{mainex}).  More explicitly, this can be achieved by means of the following  orthogonal linear change of coordinates: 
\begin{equation}\label{tiltedc}
\eta_1:=\frac{\xi_3+\xi_1}{\sqrt{2}}, \ \eta_2:=\xi_2,\ \eta_3:=\frac{\xi_3-\xi_1}{\sqrt{2}},
\end{equation}
in which $\xi_1^2+\xi_2^2-\xi_3^2=\eta_2^2-2\eta_1\eta_3,$ so that the light-cone becomes 
$$
\tilde {\bf\Gamma}:=\{\eta\in\RR^3: \eta_2^2/2-\eta_1\eta_3=0\}.
$$
In this context (see Example \ref{mainex}), it will be more natural to work with the 0-homogeneous ``level function'' 
$$
G(\eta):=\frac{\eta_2^2/2-\eta_1\eta_3}{\eta_3^2}
$$
on $\RR^3_{\times}$ in place of $F(\xi)=(\xi_1^2+\xi^2_2-\xi_3^2)/\xi_3^2.$  The  function $F$  assumes the form 
\begin{equation}\label{tildeF}
\tilde F(\eta)=\frac{\eta_2^2/2-\eta_1\eta_3}{(\eta_3+\eta_1)^2}
\end{equation}
in the coordinates $\eta$.  
However, in the $\eta$-coordinates, we shall only be interested in regions where, say, 
\begin{equation}\label{etaregion}
1/2<\eta_3<2,  \ \text{and }\   |\eta_1|, |\eta_2|< 1/4,
\end{equation}
in which $|G(\eta)|/|\tilde F(\eta)| \sim 1.$ Thus, even though the analogues of the truncated conic sets $\Gamma^\pm_R$ would here more naturally be, say, the sets 
$$
\tilde\Gamma^\pm_R:=\{\eta\in\R^3: \frac{1}{2R}<\pm G(\eta)<\frac{1}R, \  \frac 12<\eta_3<2\},
 $$
 we can cover such a region by the three  regions corresponding to $\Gamma^\pm_{R'}$ in the $\eta$-coordinates, with $R'\in \{R/2,R,2R\}.$
 This still allows us to work with the regions $\Gamma^\pm_R$ in the $\xi$-coordinates, even when working with the level function $G(\eta).$
 \medskip
 
 Let us next see how our classes of phase functions   $\F^{\kappa_1,\kappa_2,\gamma,\C}(\mathbf \Ga^\pm_R)$ will transform under a general orthogonal linear transformation 
 $$
 \xi=T^{-1}\eta,\quad \text{with} \ T\in O(3,\RR).
 $$ 

 \smallskip
 
  Given $T,$ consider the rotated  light-cone ${\bf\Gamma}_T:=T(\bf\Gamma).$ Note that $F_T(\eta):= F(T^{-1}\eta)$ is the level function
 for  ${\bf\Gamma}_T$ corresponding to $F.$
 \smallskip
 
 Moreover, if  $\phi\in \F^{\kappa_1,\kappa_2,\gamma},$ then in the $\eta$-coordinates $\phi$ corresponds to $\phi_T(\eta):=\phi(T^{-1}\eta).$
   With a slight abuse of notation, we shall then also say that   $\phi_T\in \F^{\kappa_1,\kappa_2,\gamma}.$
   \smallskip

Next, if $\xi_0\in \Ga_R,$  then let  $E=E_{\xi_0}$ be the orthonormal frame at $\xi_0.$ We claim that this frame corresponds to the frame 
\begin{equation}\label{ET}
E_T=E_T(\eta_0):=TE
\end{equation}
associated  to $F_T$ at  $\eta_0:= T\xi_0.$  Indeed, we have 
$$
\pa_{E_j}\phi(\xi_0)=\pa_{t}[\phi(\xi_0+tE_j)]\vert_{t=0}=\pa_{t}[\phi_T(\eta_0+ t TE_j)]\vert_{t=0}=\pa_{TE_j}\phi_T(\eta_0).
$$
Note: since  $\nabla F(\xi) =\nabla F_T(\eta) T,$ we also have 
$$
n(\xi_0)=\trans T\, \trans(\nabla F_T(\eta_0))/|\trans(\nabla F_T(\eta_0))|,
$$ 
so that $Tn(\xi_0)=\trans(\nabla F_T(\eta_0))/|\trans(\nabla F_T(\eta_0))|=:n_T(\eta_0),$ which matches with \eqref{ET}.
Moreover, $TE_3=T(\xi_0/|\xi_0|)=\eta_0/|\eta_0|,$ which again matches with \eqref{ET}.

\begin{example}\label{Tex}
Let $T\in O(3,\RR)$ be defined by the orthogonal transformation  $\eta=T\xi$ in  \eqref{tiltedc}. 
\end{example}
Here, $F_T(\eta)=\tilde F(\eta),$ so if $\eta=\eta_0,$ with $(\eta_0)_2=0,(\eta_0)_3\sim 1$  and $1/R\sim |G(\eta_0)|\sim |\tilde F(\eta_0)|,$ then 
$|(\eta_0)_1|\sim 1/R.$  Therefore $e_3=(0,0,1)\in\theta_{R^{-1/2}}(\eta_0)$. One computes that 
$$
\nabla \tilde F(0,0,1)= (-1,0,0),\
$$
which shows that (in the $\eta$-coordinates)  the orthonormal frame  $\tilde E$ at $e_3$ is given by 
$$
 \tilde E_1:=\trans(-1,0,0),  \tilde E_2:=\trans(0,-1,0),  \tilde E_3:=\trans(0,0,1).
$$

{\it Thus, to check that a given 1-homogeneous phase $\phi$ satisfies the estimates \eqref{deriv} at a point $\eta_0$ with $(\eta_0)_2=0,$ it will suffice to prove this for the partial coordinate derivatives
$\pa_{\eta_1}^{\alpha_1}\pa_{\eta_2}^{\alpha_2}.$}
\medskip

More generally, assume that we only require $|(\eta_0)_2|\ll 1$ in place of $(\eta_0)_2=0.$ Then, since $(\xi_0)_2=(\eta_0)_2,$ we also have
 $|(\xi_0)_2|\ll 1.$ By means of a rotation by a small angle $\om$ of size $|\omega|\ll 1$ of the light cone around its axis, we may then again reduce to the previous situation where $(\eta_0)_2=(\xi_0)_2=0.$ Note that since the level function $F$ is invariant under such a rotation, we still end up with the same frame $\tilde E$ as before, after this rotation.
 \smallskip
 
  However, the level function $G$ is  not quite  invariant under such a rotation, but passing back to the $\eta$-coordinates after performing this small rotation in the $\xi$-coordinates, on easily finds that  $G$ assumes the form
\begin{equation}\label{Gom}
G_\om(\eta)=\frac{\eta_2^2/2-\eta_1\eta_3}{\big(\eta_3(1+\cos\om)-\eta_1(1-\cos\om)-\sqrt{2}\eta_2\sin\om\big)^2},
\end{equation}
where $|\eta_2^2/2-\eta_1\eta_3|\sim 1/R.$ As we shall see in the subsequent discussion of Example \ref{mainex}, the modification of the denominator  which appears here  is not problematic.

\medskip

\subsection{Examples of phase functions}\label{examples}  
We begin with an explicit  key example, in which we adapt the previous notation:

\begin{example}\label{mainex}
 Let, for $\gamma\ge 0,$ 
$$
\phi^\gamma(\eta):= \eta_3 \left|\frac{\eta_2^2/2-\eta_1\eta_3}{\eta_3^2}\right|^\gamma=\eta_3 |G(\eta)|^\gamma,\qquad \eta\notin \tilde {\bf\Gamma},
$$
which is a priori well-defined and smooth outside the cone $\tilde {\bf\Gamma}.$ If $\ga=0,$ then $\phi^\ga=\eta_3,$ and the estimation of the corresponding FIO-cone multipliers can easily be reduced to the ones for cone multipliers, by means of  translations. But, if $\ga>0,$ then  
 outside the plane $\eta_3=0,$ $\phi^\ga$ extends to a continuous function which, however, will  have  singularities along the cone $\tilde \Gamma$ for non-integer values of $\ga.$
 
 We shall actually consider these phases  only on homogeneous  subdomains 
of the form $|\eta_1|,|\eta_2| \ll  \eta_3$   (which are covered by  \eqref{etaregion}). We claim that  for $\ga>0$ 
\begin{equation}\label{phigamma}
\phi^\gamma\in \F^{1,1/2,\gamma}.
\end{equation}
\end{example}
\begin{proof}

 Assume $R\ge R_0,$ and fix a point $\eta_0\in \tilde\Gamma^\pm_R.$ By our previous discussion, after a rotation of the light cone around its axis by a small angle $\omega$ and  making again use of homogeneity, we may assume that $\eta_0=(a,0,1),$ with $|a|\sim R^{-1},$ and that $\phi^\ga$ assumes the form
$$
\phi^\ga_\om(\eta)=\big(\eta_3(1+\cos\om)-\eta_1(1-\cos\om)-\sqrt{2}\eta_2\sin\om\big)|G_\om(\eta)|^\ga,
$$
with $G_\om$ given by \eqref{Gom}. We need to show that there is a family of constants $C_\al,|\al|\ge 2,$ which are independent of $\om$ and $R\ge R_0,$ so that the estimates \eqref{deriv} holds true. But, as we have just  seen,  these estimates can be  reduced to showing estimates of the form
$$
|\partial_{\eta_1}^{\al_1}\partial_{\eta_2}^{\al_2} \phi_\om^\ga(\eta_0)| \le C_{\alpha}R^{\alpha_1\kappa_1+\alpha_2\kappa_2-\gamma},
$$
here with  $\ka_1=1$ and $\ka_2=1/2.$ To this end, consider the re-scaled function
$$
\psi(\eta'):=R^{\gamma} \phi_\om^\gamma(R^{-\ka_1}\eta'_1,R^{-\ka_2}\eta'_2,\eta'_3).
$$
Then the estimates above are equivalent to the estimates
$$
|\partial_{\eta'_1}^{\al_1}\partial_{\eta'_2}^{\al_2} \psi(\eta'_0)| \le C_{\alpha},
$$
where  $\eta'_0:=(Ra,0,1).$
But, for $\eta'$ close to $\eta'_0,$ we have
$$
\psi(\eta')=\big(\eta'_3(1+\cos\om)-R^{-1}\eta'_1(1-\cos\om)-\sqrt{2}R^{-1/2}\eta'_2\sin\om\big)^{1-2\ga} |(\eta'_2)^2/2-\eta'_1\eta'_3|^\ga,
$$
and since  $|Ra|\sim 1,$  it becomes obvious that $\psi$ will satisfy the required estimates.
\end{proof}

\begin{example}[Embeddings of $\F_{class}$] \label{Fclass}
We  claim that if $\gamma\ge 1, \ka_1\ge \gamma/2, \ka_2\ge \gamma/2$ and $|\ka_1-\ka_2|\le 1/2,$ then
\begin{equation}\label{FclassF}
\F_{class}=\F_{class}^{\ka_1,\ka_2,\gamma} \subset \F^{\ka_1,\ka_2,\gamma},
\end{equation}
 in  the sense that the restriction of any $\phi\in \F_{class}$ to ${\bf\Gamma}_{R_0}\setminus \bf\Gamma$ lies in 
$\F^{\ka_1,\ka_2,\gamma},$ and that there is a $\C\in\Sigma$ so that $\phi\vert_{\mathbf\Gamma_R}\in \F^{\ka_1,\ka_2,\gamma,\C}(\mathbf\Gamma_R)$ and $\phi\vert_{\mathbf\Gamma^\pm_R}\in \F^{\ka_1,\ka_2,\gamma,\C}(\mathbf\Gamma^\pm_R)$ for every $R\ge R_0.$  
\end{example}

Indeed, note that if $\phi\in \F_{class},$ then by compactness of $\overline{\Gamma_{R_0}},$ there are constants $C_\beta, \beta\in \NN^3,$ so that 
$|\pa_\xi^\beta \phi(\xi)|\le C_\beta$ for every $\xi\in \Gamma_{R_0}.$ Our claim easily follows from these estimates by noting that
 $R^{\ka_1\al_1+\ka_2\al_2-\gamma}\ge 1$ whenever  $\al_1+\al_2\ge 2.$

\medskip
\begin{example}\label{phi+phi}
If $1\le\gamma\le 2$, and if  $\phi_0\in\F_{class}$ and $\phi_1\in\F^{1,\frac12,\gamma},$ then
$$
\phi:=\phi_0+\phi_1\in\F^{1,\frac\gamma 2,\gamma}.
$$
\end{example}
This is indeed an immediate consequence of the statement in the previous example in combination with \eqref{inclusion}.

\smallskip
We remark that by choosing $\ga=1$ in \eqref{FclassF}, we also have 
$$
\phi=\phi_0+\phi_1\in\F^{1,\frac12,1},
$$
but the former class  $\F^{1,\frac\gamma 2,\gamma}$  will lead to stronger estimates for the associated FIO-cone multipliers when $\ga>1$ than for the latter class.

\medskip

 As before, let again $\mathbf\Ga_R^\circ$ and $\Ga_R^\circ$ denote either  $\mathbf \Ga_R$  and   $\Ga_R,$ or  $\mathbf \Ga^\pm_R$ and $ \Ga^\pm_R.$

\begin{example}\label{existphi} Assume that $\kappa_1,\kappa_2,\gamma>0$ satisfy \eqref{param}. Then there exists some $\C\in \Sigma$ and a family 
$\{\phi^\circ_R\}_{ R\ge R_0}$ of non-trivial phase functions $\phi^\circ_R\in \F^{\ka_1,\ka_2,\gamma,\C}({\mathbf \Gamma}^\circ_R).$  These phases $\phi^\circ_R$ can even be chosen so that $ \phi^\circ_R$ does not vanish anywhere in $\mathbf\Ga_R^\circ$.
 \end{example}

 It will suffice to construct such functions in the upper half-space where $\xi_3>0.$  Given  $R\ge R_0,$ we choose $N_R\sim R^{1/2}$ 
 equally spaced and roughly  $R^{-1/2}$-separated points  $\om_i$ on the unit circle $S^1$ in the $(\xi_1,\xi_2)$-plane and an associated smooth partition of unity $\{\Psi_{\om_i}\}_i,$ with the usual control of derivatives corresponding to the scale $R^{-1/2},$ so that 
 $|\angle( \om_i, (\xi_1,\xi_2))|\lesssim R^{-1/2}$ for all $(\xi_1,\xi_2)$ in the support of $\Psi_{\om_i}.$ By putting
 $$
 \chi_{R,i}(\xi):=\chi_0(R^{\ka_1\wedge 1} F(\xi)/2)\Psi_{\om_i}((\xi_1,\xi_2)/|(\xi_1,\xi_2)|),\qquad i=1,\dots, N_R,
 $$
 where $\chi_0$ is any non-trivial smooth function supported in $[-1,2],$ 
 we then get a $0$-homogeneous partition of unity on ${\bf\Gamma}_R$ consisting of about $R^{1/2}$ functions  $\chi_{R,i},$ each of them being supported in a sector of  angular width  $\sim R^{-1/2}.$ 
 Note that we may  assume that for every $i=1,\dots, N_R,$  $\supp \chi_{R,i}\cap \Ga_{R}$ is contained in a sector
 $$
\theta_{\Rwurz}(\xi^i)=\{\xi\in\Gamma_{R}:|\angle((\xi^i_1,\xi^i_2),(\xi_1,\xi_2))|\leq c_0\Rwurz \},
$$
where we may and shall assume that $\xi^i\in \Gamma^\circ_R.$ Let $E=E_{\xi^i}$ be the orthonormal frame at $\xi^i.$ It is then easily seen that we can estimate
\begin{equation}\label{chiRi}
|\pa_{E_1}^{\al_1}\pa_{E_2}^{\al_2} \chi_{R,i}(\xi)|\le B_\al R^{\ka_1\al_1+\ka_2 \al_2}, \qquad \text{if}\  \xi_3\sim 1,
\end{equation}
uniformly in $i.$ 
Indeed,  heuristically, one easily sees that 
$$
|\pa_{E_1}^{\al_1}\pa_{E_2}^{\al_2}\chi_0(R^{\ka_1\wedge 1} F(\xi)/2)|\le B_\al R^{(\ka_1\wedge 1)(\al_1+\frac 12 \al_2)}\le B_\al R^{\ka_1\al_1+\ka_2 \al_2},
$$
(for a solid proof of the first inequality,  we refer to Lemma \ref{admissaex}, and the second inequality follows since $\ka_2\ge 1/2$).
Moreover,  easy estimates show that
$$
|\pa_{E_1}^{\al_1}\pa_{E_2}^{\al_2}\Psi_{\om_i}((\xi_1,\xi_2)/|(\xi_1,\xi_2)|\le B_\al R^{\frac 12 \al_2}\le B_\al R^{\ka_1\al_1+\ka_2 \al_2},
$$
and thus estimate \eqref {chiRi} follows by the product rule. 

It is important to note that  the constants $B_\al$ can be chosen to be independent of $R$ and $i$.

Given $i,$ let us  denote by $\eta$ the corresponding coordinates given by
$\xi=\eta E.$  If $\tilde \chi_{R,i}$ represents $\chi_{R,i}$ in these coordinates $\eta,$ then the previous estimates  \eqref{chiRi} are equivalent to
\begin{equation}\label{chiRi2}
|\pa_{\eta_1}^{\al_1}\pa_{\eta_2}^{\al_2} \tilde \chi_{R,i}(\eta)|\le B_\al R^{\ka_1\al_1+\ka_2 \al_2}, \qquad \text{if }\  \eta_3\sim 1.
\end{equation}

In view of Example \ref{Fclass},  let us next  choose  any  non-trivial classical phase functions $\psi_i, i=1,\dots,N_R,$  sharing the  same bounds 
 $|\pa_\eta^\beta \psi_i(\eta)|\le C_\beta, \, \eta\in \Ga_{R_0},$ on their derivatives, for a fixed family of constants $C_\beta, \beta\in \NN^3.$ 
We re-scale them by putting 
 $$
 \tilde\phi_{R,i}(\eta):=R^{-\gamma} \psi_i(R^{\ka_1}\eta_1,R^{\ka_2}\eta_2, \eta_3),
 $$
 and  denote by $\phi_{R,i}(\xi)$ the corresponding function in the original $\xi$-coordinates. 
 
 Then   $\phi_{R,i}$ satisfies estimates of the form
 $$
 |\pa_{\eta_1}^{\al_1}\pa_{\eta_2}^{\al_2} \tilde \phi_{R,i}(\eta)|\le C'_\al R^{\ka_1\al_1+\ka_2 \al_2-\ga}
 $$
 for suitable constants  $C'_\al,$ for every $i.$ 
 
 Then the phase $\phi_{R,i} \chi_{R,i}$ is 1-homogeneous and satisfies 
 $$
 |\pa_{\eta_1}^{\al_1}\pa_{\eta_2}^{\al_2} \big(\tilde \phi_{R,i} \tilde \chi_{R,i}\big)(\eta)|\le C''_\al R^{\ka_1\al_1+\ka_2 \al_2-\ga}
 $$
 in the $\eta$-coordinates, 
 i.e., 
 \begin{equation}\label{chiRi6}
|\pa_{E_1}^{\al_1}\pa_{E_2}^{\al_2} \big(\phi_{R,i}  \chi_{R,i}\big)(\xi)|\le C''_\al R^{\ka_1\al_1+\ka_2 \al_2-\ga}, \qquad \text{if }\  \xi_3\sim 1,
\end{equation}
with constants $C''_\al$ which are independent of $i$ and $R$.

By our perturbation lemmata,   the same kind of estimates will then even hold for any other orthonormal frame $E_{\check \xi}$ with 
$$
\check \xi\in \theta_{\Rwurz}(\xi^i)=\{\check\xi\in\Gamma_R:|\angle(\check\xi_1,\check\xi_2),(\xi^i_1,\xi^i_2))|\leq c_0\Rwurz \},
$$ with possibly larger constants  which, however,  are controlled by the constants  $C''_\al.$
 
    Finally, we set
\begin{equation}\label{phiRs}
 \phi^\circ_R:=\big( \sum\limits_{i=1}^{N_R}\phi_{R,i} \chi_{R,i}\big)\vert_{\mathbf \Ga^\circ_R}.
 \end{equation}

It is now evident  that there is a $\C\in\Sigma$ so that 
$\phi^\circ_R\in \F^{\ka_1,\ka_2,\gamma,\C}({\bf \Gamma}^\circ_R)$ for every $R\ge R_0.$

 \begin{remark}\label{globalphi}
Suppose ${\bf\Ga}^\circ_R={\bf\Ga}^\pm_R.$ If $\ka_1\geq1,$ then we can ``sum'' the  functions $\phi^\pm_R$ in \eqref{phiRs}  over dyadic values of $R$ to obtain examples of functions 
$\phi\in \F^{\ka_1,\ka_2,\gamma}$ such  that $\phi\vert_{{\bf \Gamma}^\pm_R}$ is non-trivial for every $R\ge R_0.$
We do not know wether  such functions $\phi$ can be constructed in general also when $\ka_1<1.$ 
\end{remark}

Indeed, suppose that  $\ka_1\ge 1.$ Then by choosing a  non-trivial  smooth cut-off function $\chi_1$ supported in the interval $(1/4,1)$ and satisfying  
$\sum_j\chi_1(2^j x)=1$ for all $x>0,$ one easily checks that  the function 
\begin{equation}\label{sumsoverR}
\phi(\xi):=\sum\limits_{\pm} \sum\limits_{j\ge \log_2(R_0)}\chi_1(\pm2^jF(\xi))\, \phi^\pm_{2^j}(\xi),
\end{equation} 
with $\phi^\pm_{2^j}$ defined as in \eqref{phiRs}, gives such an example (note that $2^j\le 2^{\ka_1j}$ and $2^{j/2}\le 2^{\ka_2j}$).

\color{black}


\section{L-boxes and the main theorem}\label{sboxmainthm}

\subsection{L-boxes associated to phases in $\F^{\ka_1,\ka_2,\gamma,\C}$ }\label{Sboxes}
In a next step, by adapting one of the key ideas from the paper \cite{SSS} by Seeger, Sogge and Stein to our setting, we shall decompose   any sector $\theta$  of angular width $\Rwurz$ within $\Ga^\pm_R$  into   smaller rectangular boxes  $\vartheta$ in such a way that any phase $\phi$ in $\F^{\kappa_1,\kappa_2,\gamma,\C}(\mathbf\Gamma_R^\pm)$ can ``essentially'' be replaced by an affine-linear phase $\phi_{\rm lin}$ over $\vartheta.$ The corresponding boxes  $\vartheta$, which shall be denoted as {\it L-boxes},  will essentially be chosen as large as possible for the given class $\F^{\kappa_1,\kappa_2,\gamma,\C}(\mathbf\Gamma_R^\pm).$
\medskip

To be more precise: if  $\phi_{\rm nl}:=\phi-\phi_{\rm lin}$ denotes the non-linear part  of  $\phi$ over $\vth,$ and if $\la\gg 1$ is the given ``frequency scale'' of any associated FIO-cone multiplier $T^\la_R$ with phase $\phi$ whose amplitude is in addition essentially  localized to $\vartheta$ by means of a suitable smooth cut-off  function $\chi_\vth,$ then  the factor 
$e^{i\la\phi_{\rm nl}}$ contributed  by  the non-linear part of $\phi$ can be included into this amplitude without changing the control on the  derivatives of this amplitude.
\smallskip

\medskip
To construct such boxes $\vth,$ recall that if $\phi\in \F^{\kappa_1,\kappa_2,\gamma,\C}(\mathbf\Gamma_R^\pm),$ where 
$\C=\{C_\al\}_{\al\in \NN^2_{\ge 2}},$ then by \eqref{deriv} we have in particular that
\begin{eqnarray}\label{2deriv}
\begin{split}
|\partial_{n(\xi)}^2\la \phi(\xi)| &\le& C_{(2,0)} \la R^{2\kappa_1-\gamma} \\
	|\partial_{t(\xi)}^2\la\phi(\xi)| &\le& C_{(0,2)} \la R^{2\kappa_2-\gamma}
\end{split}
	\end{eqnarray}
for any $\xi\in \theta.$ For $j=1,2,$ let us accordingly put 
\begin{eqnarray}\label{rho}
\begin{split}
	\tilde\rho_j&:=&(\lambda R^{2\kappa_j-\gamma})^{-1/2}, \\
	\rho_j&:=&\min\{\tilde\rho_j,R^{-1/j}\}. 
	\end{split}
\end{eqnarray}

We shall  always assume that

\begin{equation}\label{lambdaR}
\tilde \rho_2\le R^{-1/2}, \quad \text {i.e., that  } \  \la \ge R^{\ga+1-2\ka_2},
\end{equation}
so that 
 \begin{equation}\label{rho2tilde}
	\rho_2=\tilde\rho_2= (\lambda R^{2\kappa_2-\gamma})^{-1/2}.
\end{equation}

Note that  since $\kappa_2\geq 1/2,$  \eqref{lambdaR} holds in particular when $\la\ge R^\gamma.$ 
\medskip

The idea will then be to choose the  boxes $\vth$ essentially of dimensions $\rho_1\times\rho_2\times \rho_3,$  where we have put 
$\rho_3:=1.$

More precisely, given a point  $\xi_0\in\theta$,   let  again $E=E_{\xi_0}$ be the orthonormal frame at $\xi_0.$  Since this frame is constant along the ray through $\xi_0,$ we may indeed assume without loss of generality that $|\xi_0|=1.$ We then basically define the  {\it L-box} $\vth$ centered at $\xi_0$ by 
$$
\vartheta=\vartheta(\xi_0):=\{\xi= \xi_0+ E\zeta:\ |\zeta_j|\leq\rho_j,\ j=1,2,3 \}
$$
(for the precise definition, we refer to \eqref{decompvth}).

 Moreover,  by choosing the centers $\xi_0$ appropriately,  \eqref{rho} shows that we may essentially assume without loss  of generality that $\vth\subset \theta.$

 For  our next discussions, it will be more convenient to work  again with the linear coordinates $\eta$ given by $\xi=E\eta$ in place of the coordinates $\zeta$ above, i.e., $\zeta=\eta-\eta_0,$ where $\eta_0=:E^{-1}\xi_0.$ Note that by  the  definition of $E,$ we have $\eta_0=\trans (0,0,1).$

\smallskip
Given $\xi_0,$ we shall decompose the phase as 
$$
\phi(\xi) = \phi_{\rm lin}(\xi) + \phi_{\rm nl}(\xi),
$$
where $\phi_{\rm lin}$ denotes  the affine-linear part of $\phi$ in its Taylor expansion around $\xi_0.$ 
Since $\phi$ is  homogeneous of degree 1,  Euler's homogeneity relation implies that 
$$
\phi_{\rm lin} (\xi)=\phi(\xi_0)+\nabla\phi(\xi_0)(\xi-\xi_0)=\nabla\phi(\xi_0)\xi,
$$
so that $\phi_{\rm lin} $ is actually linear and hence also homogeneous of degree 1. Thus,   $\phi_{\rm nl}$ is homogeneous of degree 1 as well.

Note that in the coordinates $\eta,$ we can localize to the L-box $\vth$ by means of a smooth bump function 
\begin{equation}\label{thetabump}
\tilde\chi_\vth(\eta):=\chi_0(\rho_1^{-1}\eta_1)\chi_0(\rho_2^{-1}\eta_2) \chi_0(\eta_3-1).
\end{equation}

By $\chi_\vth$ we shall denote the corresponding function when expressed in the original coordinates $\xi=E\eta,$ i.e., 
$$
\chi_\vth(\xi):= \tilde\chi_\vth(E^{-1}\xi).
$$

\smallskip

The next lemma shows that the $\eta$-derivatives  of the  contribution $e^{i\la\phi_{\rm nl}}$  by the non-linear part of $\phi$ do indeed satisfy the same kind of estimates as the corresponding derivates of  the bump function $\tilde\chi_\vth(\eta).$

\begin{lemma}\label{nlcontrol}
   Assume that $\la\ge R^\ga.$ Given $\C\in \Sigma,$ there exist a family of constants $\C'=\{C'_\beta\}_{\beta\in \NN^3}$ so that for any 
$\phi\in\F^{\kappa_1,\kappa_2,\gamma,\C}(\Gamma_R^\pm)$, with $R\geq R_0$, any $\xi_0\in\Gamma^\pm_R$ with associated orthonormal frame $E=E_{\xi_0},$ the following estimates hold true:
\begin{equation}\label{sss-est}
|\partial^\beta_\eta (\la\phi_{\rm nl}\circ E)(\eta)|\leq C'_\beta \rho_1^{-\beta_1}\rho_2^{-\beta_2}\qquad \text{for all }\ \eta\in E^{-1}(\vartheta(\xi_0)), \beta\in\NN^3.
\end{equation}

Moreover, the  constants in $\C'$ can be chosen to be controlled by the constants $C_\al\in \C$ in \eqref{deriv}.\end{lemma}
\begin{remark}\label{eilaphi}
Our proof will show that  estimates of type  \eqref{sss-est} will hold as well for $e^{i\la \phi_{\rm nl}\circ E}$ in place of $\la \phi_{\rm nl}\circ E.$
\end{remark}

\begin{proof} Given $\xi_0,$ let $\tilde\phi$ again express $\phi$ in the linear coordinates $\eta,$ i.e., 
$\tilde\phi(\eta):=\phi(E\eta),$ and similarly $\tilde\phi_{\rm nl}(\eta):=\phi_{\rm nl}(E\eta),$ etc., so that 
$$
\tilde\phi(\eta) = \tilde\phi_{\rm lin}(\eta) + \tilde\phi_{\rm nl}(\eta),
$$
where now $\tilde\phi_{\rm lin}(\eta)$ is the  linear part of $\tilde\phi$ in its Taylor expansion around $e_3.$
\smallskip

In order to prove \eqref{sss-est}, we then re-scale $\tilde \phi$ by putting
$$
\psi(\eta'):=\lambda \tilde\phi(\rho_1\eta'_1\rho_2\eta'_2,\eta'_3),
$$
and accordingly setting
 $$
 \psi_{\rm nl}(\eta'):=\psi(\eta')-\psi_{\rm lin}(\eta'),
 $$
where $\psi_{\rm lin}(\eta') =\lambda \tilde\phi_{\rm lin}(\rho_1\eta'_1\rho_2\eta'_2,\eta'_3)$ is the  linear part in the Taylor expansion of $\psi$ around $\eta'_0.$

Then \eqref{sss-est} is equivalent to showing that 
\begin{equation}\label{rhoest}
	|\partial^\beta_{\eta'} \psi_{\rm nl}(\eta')|\leq C'_\beta \qquad \text{for all }\ |\eta'_1|,|\eta'_2|\le 1, \eta'_3\sim 1,\  \text{and}\  \beta\in\NN^3.
\end{equation}

Since, as we have seen, $\phi_{\rm lin}$  and $\phi_{\rm nl}$ are homogeneous of degree 1,  the same is true of $\psi_{\rm lin}$ and $\psi_{\rm nl}.$ 

 Moreover,  $\psi_{\rm nl}(\eta'_0)=0, \nabla \psi_{\rm nl}(\eta'_0)=0.$  Writing $s_1:=\eta'_1/\eta'_3$,  $s_2:=\eta'_2/\eta'_3$, $s=(s_1,s_2)$, by homogeneity and Taylor expansion in $s$  around $s=0,$ we have 
\begin{eqnarray}\label{phinl}
	\psi_{\rm nl}(\eta') = \eta'_3 \psi_{\rm nl}(s_1,s_2,1) &=&  \eta'_3 \int_0^1(1-\tau) 
	\trans(s_1,s_2) (D^2_{\eta'_1,\eta'_2}\psi_{\rm nl})(\tau s_1, \tau s_2,1)(s_1,s_2) d\tau \nonumber\\
	  &&\hskip-2cm=  (\eta'_3)^{-1} \int_0^1(1-\tau) \trans(\eta'_1,\eta'_2) (D^2_{\eta'_1,\eta'_2}\psi)(\tau\eta'_1/\eta'_3,\tau\eta'_2/\eta'_3,1)(\eta'_1,\eta'_2) d\tau,
\end{eqnarray}
where $D^2_{\eta'_1,\eta'_2}\psi$ denotes the reduced Hessian of $\psi$ as a function of $\eta'_1$ and $\eta'_2$.
By \eqref{phinl}, it will suffice to check  that for all $\alpha=(\alpha_1,\alpha_2)\in\NN^2_{\geq 2}$
\begin{equation}
	|\partial_{\eta'_1}^{\alpha_1}\partial_{\eta'_2}^{\alpha_2} \psi(\eta'_1,\eta'_2,1)| \leq C'_{\alpha},  
\end{equation}
or, equivalently, that
\begin{equation}\label{rhoest'}
	|\lambda\partial_{\eta_1}^{\alpha_1}\partial_{\eta_2}^{\alpha_2} \tilde\phi(\eta_1,\eta_2,1)| \leq C'_{\alpha} \rho_1^{-\alpha_1}\rho_2^{-\alpha_2}, 
\end{equation}
where the constants $C'_\al$ are controlled by the constants in $\C.$

\smallskip

But, by \eqref{rho}, we have  $\rho_2\le R^{-1/2}$, so $\vartheta(\xi_0)\subset\theta_{R^{-1/2}}(\xi_0)$. Therefore Corollary \ref{stability3} implies
 that for $\phi\in\F^{\kappa_1,\kappa_2,\gamma,\C}(\Gamma_R^\pm)$ and $\alpha\in\NN_{\geq2}^2$
\begin{equation}
	|\partial_{\eta_1}^{\alpha_1}\partial_{\eta_2}^{\alpha_2} \tilde\phi(\eta)| =|\partial_{n(\xi_0)}^{\alpha_1}\partial_{t(\xi_0)}^{\alpha_2} \phi(\xi)| \leq C''_{\alpha} R^{\alpha_1\kappa_1+\alpha_2\kappa_2-\gamma}
\end{equation}
on $\theta_{R^{-1/2}}(\xi_0),$ where the constants $C''_\al$ are controlled by the constants in $\C.$
\smallskip

But if $\alpha_1+\alpha_2\geq 2,$  then \eqref{lambdaR} and \eqref{rho} imply
$$\lambda R^{\alpha_1\kappa_1+\alpha_2\kappa_2-\gamma} \lesssim (\lambda R^{-\gamma})^{\tfrac12(\alpha_1+\alpha_2)} R^{\alpha_1\kappa_1+\alpha_2\kappa_2}
= (\tilde\rho_1)^{-\alpha_1} (\tilde\rho_2)^{-\alpha_2} \leq  \rho_1^{-\alpha_1} \rho_2^{-\alpha_2},$$
since we are assuming that $\la\ge R^\ga.$ This  proves \eqref{rhoest'}.
\end{proof}

\bigskip

\subsection{Decomposition into L-boxes}\label{Lboxdecomp}

Assume  that $\phi\in\F^{\kappa_1,\kappa_2,\gamma,\C}(\mathbf\Gamma_R^\pm).$ The preceding discussion shows that we can decompose  each sector $\theta$ of angular width $R^{-1/2}$ in $\Gamma_R^\pm$ into at most  $N(\lambda,R)$ L-boxes $\vartheta,$ where 
\begin{equation}\label{Ndefine}
	N(\lambda,R)\sim \frac{R^{-3/2}}{\rho_1\rho_2}.
\end{equation}
 More precisely, this can be done by means of smooth bump functions of the form
 
\begin{equation}\label{decompvth}
\chi^{\rho_1,\rho_2}_{i,k}(\xi):=\chi_0(\pm\rho_1^{-1}F(\xi)-k) \Psi_{\om_i}((\xi_1,\xi_2)/|(\xi_1,\xi_2)|)\chi_1(\xi_3),
\end{equation}
where similar to  Example \ref{existphi} the  functions $\Psi_{\om_i}((\xi_1,\xi_2)/|(\xi_1,\xi_2)|)$ are  angular bump functions which localize to angular sectors, here of angular width $\sim \rho_2,$ and $k\in \NN$ is such that $\frac 1{2R}\le k\rho_1\le \frac 1R.$ By choosing the $\Psi_{\om_i}$ and $\chi_0$ appropriately, we may also assume that these functions form a partition of unity on $\Gamma_R.$

\smallskip

We observe that $\chi^{\rho_1,\rho_2}_{i,k}$ does indeed localize to an L-box $\vth=\vth_{i,k}.$ A priori, the support of  $\chi^{\rho_1,\rho_2}_{i,k}$ might look like a curved box, but this  not the case, due to the following inequality:
\begin{equation}\label{rho1rho2}
\rho_2^2\le \rho_1.
\end{equation}
This is an  immediate consequence of the following lemma, which shows that $\rho_2^2\le \rho_1 R^{1/2} \rho_2\le \rho_1.$

\begin{lemma}\label{rhofrak}
Assume that $\tilde \rho_2=\rho_2.$ Then, for any $i,j=1,2$, we have
$$
\frac{\rho_i}{\rho_j}\leq R^{1/2}.
$$
\end{lemma}
\begin{proof} 
 By definition \eqref{rho}, we have
$$
	\frac{\tilde\rho_i}{\tilde\rho_j} = R^{\kappa_j-\kappa_i} \leq R^{1/2},
$$
since  $|\kappa_1-\kappa_2|\le 1/2.$
Recall that $\rho_2=\tilde\rho_2\leq R^{-1/2}$ by \eqref{rho2tilde}, so we always have 
$$
\frac{\rho_1}{\rho_2}\leq\frac{\tilde\rho_1}{\tilde\rho_2} \leq  R^{1/2}.
$$
 If $\rho_1=\tilde\rho_1$, we also have that $\rho_2/\rho_1\le R^{1/2}.$  Otherwise, $\rho_1=R^{-1}$, and
$\rho_2\leq R^{-1/2} = R^{1/2}\rho_1,$ so that again $\rho_2/\rho_1\le R^{1/2}.$
\end{proof} 

By means of \eqref{rho1rho2} it is also easy to see that the derivatives of $\chi^{\rho_1,\rho_2}_{i,k}$  satisfies the same kind of estimates in the $\eta$-coordinates as the  bump function $\tilde \chi_\vth$ in \eqref{thetabump}.

\smallskip 

Of course, the preceding decomposition into L-boxes can be preformed in a very similar way also for phase functions 
$\phi\in\F^{\kappa_1,\kappa_2,\gamma,\C}(\mathbf\Gamma_R)$ defined on $\mathbf\Gamma_R,$ with $\Ga^\pm_R$ replaced by 
$\Ga_R;$  we skip the details.

\subsection{The small mixed derivative condition (SMD)}\label{SMDcond}

Let again either $\Ga^\circ_R=\Ga_R,$ or $\Ga^\circ_R=\Ga^\pm_R.$ In  the case where  $R\tilde \rho_2>1,$ i.e., where 
$$
\la<R^{2(1-\ka_2)+\gamma},
$$
 the conditions \eqref{deriv} on a phase  $\phi\in\F^{\kappa_1,\kappa_2,\gamma,\C}(\mathbf \Ga^\circ_R)$ may not be strong enough in general for guaranteeing  our FIO-cone multiplier estimates,  and we need an extra condition allowing to  control the variation of the gradient of $\la\phi$  in a sufficient way.
\smallskip

To this end, recall that if $\al_1=\al_2=1,$ then \eqref{deriv} requires in particular that 
 \begin{equation}\label{mixedder}
 |\partial_{n(\xi_0)}\partial_{t(\xi_0)} \la \phi(\xi_0)|\leq C\la R^{\kappa_1+\kappa_2-\gamma}\lesssim \rho_1^{-1} \rho_2^{-1}
	 \quad \text{for all}\  \xi_0\in\Gamma^\circ_R,
\end{equation}
since $\la R^{\kappa_1+\kappa_2-\gamma}=  \tilde\rho_1^{-1} \tilde\rho_2^{-1}\le \rho_1^{-1} \rho_2^{-1}.$

We shall therefore   impose the following strengthening of the above condition, 
the {\it small mixed-derivative condition (SMD)}:
\begin{equation}\label{smd}
 |\partial_{n(\xi_0)}\partial_{t(\xi_0)} \la \phi(\xi_0)|\leq C \frak a \rho_1^{-1} \rho_2^{-1}
	 \quad \text{for all}\  \xi_0\in\Gamma_R^\circ,
\end{equation}
where 
\begin{equation}\label{adefine}
\frak a:=\frac 1{R\tilde \rho_2}\wedge 1.
\end{equation}
Note that this condition is stronger than \eqref{mixedder} only when $R\tilde \rho_2>1.$

\subsection{Admissible amplitudes}

Let again denote $\Ga^\circ_R$ either $\Ga_R$, or $\Ga^\pm_R.$ A smooth function $a_R$ on $\Ga^\circ_R$ will be called an {\it admissible amplitude} on $\Ga^\circ_R,$ if for every $\xi\in\Ga^\circ_R,$
\begin{equation}\label{admisamp}
 |\partial_{n(\xi)}^{\alpha_1}\partial_{t(\xi)}^{\alpha_2} \partial_{\xi}^{\alpha_3} a_R(\xi)|\leq B_{\alpha} R^{\alpha_1+\alpha_2/2}
 \qquad \text{for all } \al\in \NN^3.
\end{equation}

\begin{remark}\label{amppert}
If $a_R$ is an admissible amplitude in $\Ga^\circ_R,$ then there are constants $\{B'_\al\}_\al$ which are controlled by the constant $\{B_\al\}_\al$  in \eqref{admisamp} so that for every $\xi\in\Ga_R$ and every  $\check\xi\in \theta_{R^{-1/2}}(\xi),$
$$
|\partial_{n(\check\xi)}^{\alpha_1}\partial_{t(\check\xi)}^{\alpha_2} \partial_{\check\xi}^{\alpha_3} a_R(\xi)|\leq B'_{\alpha} R^{\alpha_1+\alpha_2/2}
 \qquad \text{for all } \al\in \NN^3.
 $$
 Conversely, if for every $\xi\in\Ga^\circ_R$ there is a $\check\xi\in \theta_{R^{-1/2}}(\xi)$ such that the above estimates hold true, then the estimates \eqref{admisamp} hold true, with constants $\{B_\al\}_\al$ which are controlled by the constants $\{B'_\al\}_\al.$ 
\end{remark}
This follows immediately from Lemma \ref{EEcheck} in combination with Lemma \ref{stability} (choosing $\al:=R^{-1/2}$ and $\ka_1:=1, \ka_2:=1/2$ and $\ka_3:=0$).

\begin{lemma}\label{admissaex}
For $R\gg 1,$ let $I_R:=\{u\in\RR: |Ru|<1\}$  and $I^\pm_R:=\{u\in\RR: 1/2 <\pm Ru<1\},$ and denote by $I^\circ_R$ either $I_R$ (and let then $\Ga^\circ_R=\Ga_R$), or $I^\pm_R$ (and let then $\Ga^\circ_R=\Ga^\pm_R$). 

Assume that $h_R:I^\circ_R\to \bC$ is a smooth function satisfying estimates of the following form on $\Ga^\circ_R:$
\begin{equation}\label{hRest}
\big|\frac {d^k}{du^k}h_R(u)\big|\le C_k R^k, \qquad k\in \NN.
\end{equation}

Then $a_R(\xi):=h_R(F(\xi)), \,\xi\in\Ga^\circ_R,$ is an admissible amplitude, i.e., $a_R$ satisfies  estimates of the form \eqref{admisamp}, with constants $B_\al$ controlled by the family of constants $\{C_k\}_k.$ 
\end{lemma}
\begin{proof} In order to check \eqref{admisamp} at $\xi\in\Ga^\circ_R,$ after scaling and a rotation around the axis of the light cone, we may assume that $\xi_2=0,$   $\xi_3=1$ and $\xi_1=-1+\mathcal O(R^{-1}).$ Then, by changing to the tilted coordinates $\eta$ from \eqref{tiltedc} as in Example \ref{Tex},  in which $\eta_3+\eta_1\sim 1,$
the function $a_R$ assumes the form  $\tilde a_R(\eta)= h_R(\tilde F_v(\eta)),$
where  (cf. \eqref{tildeF}) 
$$
\tilde F(\eta):=\frac{\eta_2^2/2-\eta_1\eta_3}{(\eta_3+\eta_1)^2}.
$$
Moreover, the vector $\check \xi$ corresponding to $\check\eta:=e_3=\trans(0,0,1)$ in the tilted coordinates lies in $\theta_{R^{-1/2}}(\xi),$ and thus, by  Remark \ref{amppert},  as in Example  \ref{Tex} it  will suffice to prove that estimates of the form
$$
 |\partial_{\eta_1}^{\al_1}\partial_{\eta_2}^{\al_2}\partial_{\eta_3}^{\al_3} \tilde a_R(\eta)|\le B_\al R^{\al_1+\tfrac12 \al_2}
 $$
hold at $\eta=\trans (0,0,1),$ with constants $B_\al$ controlled by the constants $C_k.$ 
 \smallskip
 
The scaling transform given by the change of coordinates $\eta_1'=R\eta_1$, $\eta_2'=R^{1/2}\eta_2$ $\eta_3'=\eta_3$ finally allows to reduce these estimates to the equivalent estimates 
\begin{equation}\label{hRests}
 |\partial_{\eta'_1}^{\al_1}\partial_{\eta'_2}^{\al_2}\partial_{\eta'_3}^{\al_3} \tilde{\tilde a}_R(\eta')|\le B_\al,
\end{equation}
 at $\eta'=\trans (0,0,1),$ where  $\tilde{\tilde a}_R(\eta'):=\tilde a(R^{-1}\eta'_1,R^{-\frac 12}\eta'_2,\eta'_3)$ expresses $\tilde a_R$ in the coordinates $\eta'.$  But, one easily computes that 
 $$
 \tilde{\tilde a}_R(\eta')=\tilde h_R(\tilde {\tilde F}(\eta')),
 $$
 where $\tilde h_R(u'):=h_R(R^{-1} u')$, and where
 $$
 \tilde {\tilde F}(\eta'):=\frac{(\eta'_2)^2/2-\eta'_1\eta'_3}{(\eta'_3+R^{-1}\eta'_1)^2}.
 $$
 Since clearly 
$$
\left|\frac{d^k}{du^k} \tilde h_R(u')\right| = R^{-k} |h^{(k)}_R(u)|\leq C_k,
$$
the estimates \eqref{hRests}  are now  obvious.
\end{proof}

\subsection{The main result}\label{maintheorem}
We are now in a position to state our main theorem. 

\smallskip

Given $R\gg  1,$  we fix a smooth bump function $\chi_R$ which localizes to $\Ga^\circ_R$ as follows:

-- If $\Ga^\circ_R=\Ga_R,$ then we define $\chi_R$ by 
$$
\chi_R(\xi):= \chi_0(RF(\xi)) \chi_1(\xi_3),
$$
 with suitable smooth bump functions  $\chi_0$ and $\chi_1$ so that $\chi_R$ is supported in $\Ga_R.$ 
 
 -- If $\Ga^\circ_R=\Ga^\pm_R,$ then we define $\chi_R$ by 
$$
\chi_R(\xi):= \chi_1(\pm RF(\xi)) \chi_1(\xi_3),
$$
 with a  smooth bump function $\chi_1$ chosen so that $\chi_R$ is supported in $\Ga^\pm_R.$

\smallskip

Given a phase $\phi\in \F^{\kappa_1,\kappa_2,\gamma,\C}(\mathbf \Ga^\circ_R)$ and an admissible amplitude $a_R$ on 
$\Ga^\circ_R,$ we denote by  $T^\la_R:=T_{m^\lambda_R}$  the Fourier multiplier  operator given by the multiplier
\begin{equation}\label{mlaR}
m^\lambda_R(\xi) = m^{\lambda,\phi}_{R,a_R}(\xi):= e^{-i\lambda\phi(\xi)} a_R(\xi) \chi_R(\xi).
\end{equation} 

\smallskip

\begin{thm}\label{mainthm}
Assume that  a family of constants $\C\in\Sigma$ and $R_0\gg1$ are given, and denote for $R\ge R_0$ by $\mathbf\Ga^\circ_R$  either $\mathbf \Ga^\pm_R,$ or $\mathbf \Ga_R.$ 

Let $\phi\in \F^{\kappa_1,\kappa_2,\gamma,\C}(\mathbf \Ga^\circ_R),$ and assume that $\phi$ satisfies the small mixed-derivative condition (SMD). Moreover, let  $a_R$ be an admissible amplitude, and let $N(\la,R)$ be as defined in \eqref{Ndefine}.  

Then, for any $\epsilon >0,$   there exists a constant $C_\epsilon>0$ so that for all $R\geq R_0$ and  $\lambda\ge R^{\gamma}$ the following estimate holds true:
\begin{equation}\label{mainthmeq}
 \|T^\lambda_{R} f\|_{L^4(\RR^3)} \leq C_\epsilon\,R^\epsilon N(\lambda,R)^{\frac12} \|f\|_{L^4(\RR^3)}.
\end{equation}
The constant  $C_\epsilon$ can be chosen independently of   $R\ge R_0,$ $\phi$ and $a_R,$ and  can be controlled by  a finite number of the constants $C_\al$ in $\C$ and $B_\al$ in \eqref{admisamp}.
\end{thm}

\begin{remark}\label{ka2g2}
If $\la\ge R^{2(1-\ka_2)+\gamma},$ i.e., if $R\tilde\rho_2\le 1,$ then  condition (SMD) does not provide any  extra condition.
This applies in particular when $\ka_2\ge 1,$  since we  assume that $\la\ge R^\ga$  (which, in return,  is basically needed in the proof of Lemma \ref{nlcontrol}).\end{remark}

\begin{cor}\label{mainlp}
Assume that  $\C\in\Sigma$ and $R_0\gg1$ are given, and denote for $R\ge R_0$ by $\mathbf\Ga^\circ_R$  either $\mathbf \Ga^\pm_R,$ or $\mathbf \Ga_R.$ 

Let $\phi\in \F^{\kappa_1,\kappa_2,\gamma,\C}(\mathbf \Ga^\circ_R),$ and assume that $\phi$ satisfies the small mixed-derivative condition (SMD). Moreover, let  $a_R$ be an admissible amplitude, and let $N(\la,R)$ be as defined in \eqref{Ndefine}.  

Then, for any $\epsilon >0,$   there exists a constant $C_\epsilon>0$ so that for all $R\geq R_0$ and  $\lambda\ge R^{\gamma}$ the following hold true:
\smallskip

(i) If $4/3\le p\le 4,$ then
\begin{equation}\label{mainlpeq}
 \|T^\lambda_{R} f\|_{L^p(\RR^3)} \le C_\epsilon\,R^\epsilon N(\lambda,R)^{|1-\frac 2p|} \|f\|_{L^p(\RR^3)}.
\end{equation}
\smallskip

(ii) If $1\le p<4/3$ or $4<p\le \infty,$ then
\begin{equation}\label{mainlpeq2}
 \|T^\lambda_{R} f\|_{L^p(\RR^3)} \le C_\epsilon\,R^{\epsilon} R^{|1-\frac 2p|-\frac 12}N(\lambda,R)^{|1-\frac 2p|} \|f\|_{L^p(\RR^3)}.
\end{equation}
The constant  $C_\epsilon$ can here be chosen independently of   $R\ge R_0,$ $\phi$ and $a_R,$ and  can be controlled by  a finite number of the constants $C_\al$ in $\C$ and  $B_\al$ in \eqref{admisamp}.
\end{cor}

\section{Preparatory proof steps}\label{prep}

\color{black}

\subsection{Guth-Wang-Zhang semi-norms} \label{GWZ} We begin by recalling some of the fundamental insights from the proof of the square function estimate in \cite{GWZ}. 
\medskip

Following closely  the notation from  \cite{GWZ}, suppose that  $s$ is a dyadic parameter in the range $R^{-1/2}\leq s\leq 1$. 
We then group the caps $\theta,$ i.e., the sectors of angular width about $R^{-1/2}$ within $\Gamma_R$, into sectors $\tau$ of angular width  $s$ (also called the aperture $s=d(\tau)$ of $\tau$  in  \cite{GWZ}) in the $(\xi_1,\xi_2)$-plane. 

For instance, if $s=R^{-1/2}$, each sector $\tau$ simply is a single cap $\theta$. And, if $s=1$, $\tau$ consists of all caps\footnote{There is no loss of generality to first decompose the full circle into sectors of angle 1.} $\theta$.\\
Recall that each cap $\theta$ is essentially a box of dimensions $R^{-1}\times R^{-1/2}\times 1$ with respect to the coordinates 
$\eta$ associated to the  orthonormal frame 
$$E_1=n(\xi_\theta),\quad E_2=t(\xi_\theta),\quad E_3=\xi_\theta/|\xi_\theta|,$$
if $\xi_\theta$ denotes the center of the cap $\theta.$

The dual convex body $\theta^*$ to $\theta$ is then a rectangular box of dimensions $R\times R^{1/2}\times 1$  with respect to the same coordinates $\eta,$ which is centered at the origin.

Note that if  $\chi_\theta$ is  a smooth bump function adapted to the box $\theta$, then, by the uncertainty principle,
 $\widehat{\chi_\theta}$ is essentially supported in $\theta^*$.
More precisely: if, in the coordinates $\zeta$ given by  $\xi=\xi_\theta+E\zeta$, $\chi_\theta$ is of the form 
 $$
 \chi_\theta(\xi)=\chi_0(R\zeta_1)\chi_0(R^{1/2}\zeta_2)\chi_0(\zeta_3), 
 $$
 then by integration by parts, $\widehat{\chi_\theta}$ decays rapidly far away from $\theta^*$.
\smallskip

Of course, the orientation of the boxes $\theta^*$ differs for different $\theta$ in $\tau,$ but there is a ``bounding box''
$U_{\tau,R}$ containing all these dual boxes $\theta^*$,  which is essentially the convex hull of the union of all these $\theta^*,$ i.e., 
$$
U_{\tau,R} \simeq \text{Convex Hull} (\bigcup_{\theta\subset\tau} \theta^*).
$$
If $\xi_\tau\in \Gamma_R$ denotes the center of $\tau,$  then in the  coordinates $x$ corresponding to the orthonormal frame at $\xi_\tau$ this box is given by
$$
U_{\tau,R} :=\{x\in\RR^3: |x_1|\le R, |x_2|\le Rs \text{ and } |x_3|\le Rs^2\}.
$$Let us again look at the   two extremal cases: For  $s=R^{-1/2}$, $U_{\tau,R}$ consists of a single dual box $\theta^*,$ which is indeed a box of dimensions $R\times R^{1/2}\times 1$. And, for $s=1$, $U_{\tau,R}$ is a box of dimensions  $R\times R\times R$ (this size is of obviously  needed if  $U_{\tau,R}$ is to contain all  boxes $\theta^*$).

\medskip
Now by $\{U\}_{U\transl U_{\tau,R}}$, we denote a tiling of $\R^3$ by boxes $U$ which are translates of $U_{\tau,R}$ and only overlap on a set of measure zero. Define the following semi-norms:
\begin{eqnarray*}\label{GWZs}
	\|f\|_{GWZ,s}&:=& \left(\sum_{\angle(\tau)=s} \sum_{U\transl U_{\tau,R}} |U|^{-1}\big \|(\sum_{\theta\subset\tau}|f_\theta|^2)^\frac12\big \|_{L^2(U)}^4 \right)^\frac14 \\
	\|f\|_{GWZ}&:=& \sum_{R^{-1/2}\le s\le 1}\|f\|_{GWZ,s}.
\end{eqnarray*}
Here, $\sum_{\angle(\tau)=s}$ means summation  over a decomposition into sectors $\tau$ of angular width  $s$ with measure zero overlap, and   $\sum_{R^{-1/2}\le s\le 1}$  means summation over all dyadic values of $s$ in the range $R^{-1/2}\leq s\leq 1$.

\medskip

To gain a better understanding of these semi-norms, note that they might be seen as refinements of the $L^4$-norm: By Cauchy-Schwarz' inequality,
$$ 
|U|^{-1/4} \| F\|_{L^2(U)} \leq \| F\|_{L^4(U)},
$$
and therefore 
$$
\sum_{U\transl U_{\tau,R}} |U|^{-1} \|F \|_{L^2(U)}^4 \leq \sum_{U\transl U_{\tau,R}}\| F\|^4_{L^4(U)} = \| F\|^4_{L^4(\R^3)}.
$$
Consider again  the extremal cases $s=1$ and $s=R^{-1/2}$:

\smallskip
For $s=1$, recall that there is only one cap $\tau$, and $U$ is a cube of side-length  $R$. If we denote by $\{B_R\}$ a tiling of $\R^3$ into cubes of side-length $R$, then we have
\begin{equation}\label{GWZ1}
	\|f\|_{GWZ,1}= \left(\sum_{B_R} |B_R|^{-1} \big\|(\sum_{\theta}|f_\theta|^2)^\frac12 \big\|_{L^2(B_R)}^4 \right)^\frac14,
\end{equation}
so that the previous estimates imply that 
$$
	\|f\|_{GWZ,1}\leq \big\|(\sum_{\theta}|f_\theta|^2)^\frac12 \big\|_{L^4(\R^3)}.
$$
This is just the right-hand-side for the square function estimate \eqref{squarest}.
\smallskip

For $s=R^{-1/2}$, recall that any $\tau$ consists of a single cap $\theta$, so that
\begin{equation}\label{GWZ2}
	\|f\|_{GWZ,R^{-1/2}}= \left(\sum_{\theta} \sum_{U\transl \theta^*} |U|^{-1} \|f_\theta \|_{L^2(U)}^4 \right)^\frac14,
\end{equation}
which implies that 
$$
	\|f\|_{GWZ,R^{-1/2}}\leq \left(\sum_{\theta} \|f_\theta \|_{L^4(\R^3)}^4 \right)^\frac14;
$$
this is like the right-hand-side in the $\ell^4$-decoupling  estimate \eqref{4decouple}, but without an additional power of $R$. 
Note that
$$
	\left(\sum_{\theta} \|f_\theta \|_{L^4(\R^3)}^4 \right)^\frac14 = \big\|(\sum_{\theta}|f_\theta|^4)^\frac14 \big\|_{L^4(\R^3)} \leq  
	\big\|(\sum_{\theta}|f_\theta|^2)^\frac12 \big\|_{L^4(\R^3)}.
$$
As similar arguments also work for $s$ in the intermediate range $R^{-1/2}<s<1$ (cf. \cite{GWZ}), we actually have 
\begin{equation}\label{gwzlessl4}
	\|f\|_{GWZ} \lesssim \|(\sum_{\theta}|f_\theta|^2)^\frac12 \|_{L^4(\R^3)}.
\end{equation}
The key step in \cite{GWZ} (see \cite[Theorem 1.3]{GWZ}) is to prove for all functions $f$ with Fourier support contained in the $R^{-1}$ neighborhood  $\Gamma_R$ of the truncated cone and all $\epsilon >0$  the estimate
\begin{equation}\label{gwzest}
	\| f\|_4\le C_\epsilon R^\epsilon \|f\|_{GWZ}.
\end{equation}
By \eqref{gwzlessl4}, this implies the  square function estimate \eqref{squarest}. Actually,  as shown by Mockenhaupt \cite{Mo}, we also have that for such functions 
\begin{equation}\label{gwzestmo}
	\|(\sum_{\theta}|f_\theta|^2)^\frac12 \|_{L^4(\R^3)}\le  C (\log R)^\alpha \|f\|_{L^4(\R^3)},
\end{equation}
so that, up to factors $R^\epsilon$ with arbitrarily small $\epsilon >0,$ all the norms $\|f\|_4, \, \|f\|_{GWZ} $ and $\|(\sum_{\theta}|f_\theta|^2)^\frac12 \|_{L^4(\R^3)}$ are equivalent on functions $f$ with Fourier support contained in $\Gamma_R.$

\medskip

However, estimate \eqref{gwzest} is much better suited for an induction on scales approach (i.e., induction on the size of $R$) than  \eqref{squarest}.
\smallskip

\eqref{gwzest} will also be the starting point for our proof. Note that the advantage of the Guth-Wang-Zhang semi-norm \eqref{GWZ} is the appearance of $L^2$-norms in place of $L^4$-norms, which opens the possibility of using Plancherel's theorem, even though  we have to take into consideration the localization to the sets $U$.

\begin{remark}\label{tothetapm}
In our later applications, we shall usually work only within  $\Gamma_R^\pm$ in place of $\Gamma_R.$ Then one has to replace the caps $\theta\subset \Gamma_R$ by the corresponding caps $\theta^\pm:=\theta\cap \Gamma_R^\pm$ in the arguments to follow. For simplicity of the notation,  we shall then denote those caps $\theta^\pm$ again by $\theta.$ 
\end{remark}

\color{blue} 
\subsection{Dual L-boxes} 
\color{black}
Assume again that $\phi\in \F^{\kappa_1,\kappa_2,\gamma,\C}(\mathbf \Ga^\pm_R).$
Besides the boxes $U \slash\hskip-.1cm \slash U_{\tau,R}$ appearing in  the Guth-Wang-Zhang semi-norms, we also have to take into account the size of the  boxes $\vartheta^*$ dual to the small boxes $\vartheta$ of dimensions $\rho_1\times\rho_2\times 1.$

Recall that an L-box $\vth$ is essentially of the form 
$$
\vartheta=\{\xi=\xi_\vth+E^\vth\zeta:\ |\zeta_j|\leq\rho_j,\ j=1,2,3 \},
$$
if $E^\vth:=E_{\xi_\vth}$ denotes  the orthonormal frame associated to the center $\xi_\vartheta$ of $\vartheta.$ Here,  we have put $\rho_3:=1.$

The dual box $\vartheta^*$ then is
$$
\vartheta^{\ast}:=\{y\trans E^\vth:\ |y_j|\leq\rho_j^{-1},\ j=1,2,3 \}.
$$

Actually, in order to be able to  control certain Schwartz tails in later arguments,  we shall work with slightly larger dual boxes  
$\vartheta^{\ast,\e}$ than $\vartheta^*,$  namely 
$$
\vartheta^{\ast,\e}:=\{y\trans E^\vth:\ |y_j|\leq\rho_j^{-1}\lambda^\e,\ j=1,2,3 \},
$$
where $\e>0$ will be chosen sufficiently small.

\medskip
In the sequel, we shall use the following notation: If $B$ is any symmetric convex body in $\RR^3$ centered at the point $a,$ then $r\cdot B$ will denote the scaling by the factor $r>0$ of $B$ which keeps the center fixed, i.e.,
$$
r\cdot B:= a+r(-a+B).
$$
For instance, if $c\ge 2$ is any fixed positive constant, then $c\cdot B$ is often called the {\it doubling} of $B.$

\medskip

Since the phase $\phi$ is essentially  linear on $\vartheta,$ the Fourier support of the multiplier $m^\lambda_R$ restricted to $\vartheta$ will essentially be a translation of $\vartheta^{\ast,\e}$, as the modulation term $e^{i\lambda\phi(\xi)}$ effects a translation on the Fourier transform side.
\smallskip

More precisely, let $\hat\vartheta:=\lambda\nabla\phi(\xi_\vth)+\vartheta^*,$ and 
\begin{equation}\label{vthhat}
	\hat\vartheta^\e:=\lambda\nabla\phi(\xi_\vth)+\vartheta^{\ast,\e}.
\end{equation}
Note that then
$$
\hat\vartheta^\e=\la^\ve \cdot\hat\vth.
$$

We localize our multiplier 
$$m_R^\lambda=e^{-i\lambda\phi(\xi)}  a_R(\xi) \chi_1(\pm RF(\xi)) \chi_1(\xi_3)$$ 
by means of a smooth bump function $\chi_\vth$ adapted to $\vartheta$ of the form
\begin{equation}
\chi_\vth(\xi):=\chi_0(\rho_1^{-1}\eta_1)\chi_0(\rho_2^{-1}\eta_2)\chi_0(\eta_3-1)
\end{equation}
in the associated coordinates  $\eta=\trans E^\vth \xi,$ 
 by setting
$$m_\vartheta:=m_{R,\vartheta}^\lambda :=m_R^\lambda \chi_\vth.$$

Define the measure $\mu_\vth$ by $\widehat{\mu_\vth}:=m_\vartheta.$ Then
\begin{equation}\label{Tdecomp}
T^\la_R f=\sum\limits_{\vth} f*\mu_\vth.
\end{equation}

By adapting  a key idea of the method of Seeger, Sogge and Stein from \cite{SSS}, we obtain 
\begin{lemma}\label{4iersupp}
For all $N\in\NN$, there is a constant $C_N$ such that for all $R\geq R_0$, all L-boxes $\vth$ and all $x\in\R^3$, we have
\begin{equation}\label{outside}
	|\mu_\vth(x)|\leq C_N \rho_1\rho_2\rho_3(1+\sum_{j=1}^3 \rho_j 
	|\langle E^\vth_j,x-\lambda\nabla\phi(\xi_\vth)\rangle| )^{-N}.
\end{equation}

The constants $C_N$ can be chosen to depend only on $N$ and   a finite number of the constants $C_\al$ in $\C.$

In particular, there is a constant $C$ depending only on a finite number of the constants $C_\al$ in $\C$ so that 
\begin{equation}\label{mutheta1}
\|\mu_\vth\|_{L^1(\RR^3)}\le C
\end{equation}
for all phases $\phi\in\F^{\kappa_1,\kappa_2,\gamma,\C}(\Gamma_R^\pm)$ and all L-boxes $\vth.$

\end{lemma}

\begin{proof} We denote by $E$ the orthonormal frame $E^\vartheta.$ Then, up to a multiplicative constant,
\begin{eqnarray*}
\mu_\vth(x)&=&\int e^{i(x\xi-\lambda\phi(\xi))} \chi_1(\pm RF(\xi)) a_R(\xi) \chi_1(\xi_3)\chi_\vth(\xi) d\xi  \\
		&=&e^{-i\lambda\phi(\xi_\vth)}\int e^{i([x-\lambda\nabla\phi(\xi_\vth)]\xi-\lambda\phi_{\rm nl}(\xi))} \chi_1(\pm RF(\xi)) a_R(\xi)\chi_1(\xi_3)\chi_\vth(\xi) d\xi	  \\
		&=&e^{-i\lambda\phi(\xi_\vth)}\int e^{i[x-\lambda\nabla\phi(\xi_\vth)]E\eta}e^{-i\lambda\phi_{\rm nl}(E\eta)} \chi_0(\rho_1^{-1}\eta_1)\chi_0(\rho_2^{-1}\eta_2)\chi_0(\eta_3-1) \tilde a_R(\eta)\,  d\eta,
\end{eqnarray*}
where we have put 
$$
\tilde a_R(\eta):=a_R(E\eta) \chi_1(\pm RF(E\eta))\chi_1((E\eta)_3).
$$

Note that by \eqref{admisamp} and Remark \ref{amppert} the following estimates hold true:
$$
 |\partial_{\eta_1}^{\al_1}\partial_{\eta_2}^{\al_2}\partial_{\eta_3}^{\al_3} \tilde a_R(\eta)|\le B_\al R^{\al_1+\tfrac12 \al_2}
 \lesssim \rho_1^{-\al_1}\rho_2^{-\al_2}, \qquad \al\in \NN^3.
 $$
 since $\rho_1\le R^{-1}$ and $\rho_2\le R^{-1/2}.$

 Rescaling $\eta_j$ by $\rho_j$ and integrating by parts now gives the desired bounds \eqref{outside}, since,  by Lemma \ref{nlcontrol},  the  derivatives of the factor $e^{-i\lambda\phi_{\rm nl}(E\eta)}$ behave no worse than those of the amplitude.

The estimate \eqref{mutheta1} is an immediate consequence. 
\end{proof}

 The  lemma shows that the measure $\mu_\vth$ is essentially supported in $\hat\vartheta^\e:$

\begin{cor}\label{4iersupp2} For all phases $\phi\in\F^{\kappa_1,\kappa_2,\gamma,\C}(\Gamma_R^\pm),$  all L-boxes $\vth$ and all $N\in\NN$, there is a constant $C_N$ depending only on $N$ and a finite number of the constants $C_\al$ in $\C$  such that for all $R\geq R_0$ the following hold:
 \smallskip
 
 a)   If $\vth$ is any  L-box, then
\begin{equation}\label{rapiddecay}
	\|\mu_\vth\|_{L^1(\RR^3\setminus \hat\vartheta^\e)} \leq C_N \lambda^{-N}.
\end{equation}

b) Suppose $\xi_0\in\Gamma_R$ and  $T\subset \RR^3$ is a cuboid with axes parallel to the directions given by the orthonormal frame $E_{\xi_0}$ at $\xi_0.$ Furthermore, suppose  $\L$ is a finite collection of boundedly overlapping L-boxes $\vth$ contained in $\Gamma_R$ such that 
$$
\hat\vth= \lambda\nabla\phi(\xi_\vth) +\vartheta^*\subset T \qquad\text{for all}\ \vth\in \L.
$$
If we put  $\mu_\L:=\sum\limits_{\vth\in \L} \mu_\vth,$ 
then
$$
\|\mu_\L\|_{L^1(\RR^3\setminus \la^\ve\cdot T)} \leq C_N \lambda^{-N}.
$$
\end{cor}

\begin{proof}
a) Changing variables to $z_j:=\rho_j 	\langle E^\vth_j,x-\lambda\nabla\phi(\xi_\vth)\rangle$, by \eqref{outside} we have
$$
\int_{\RR^3\setminus \hat\vartheta^\e} |\mu_\vth(x)| dx \lesssim  
 C_N\int_{\{|z|>\lambda^\e\}} (1+|z|)^{-N} dz \lesssim C_N \lambda^{(3-N)\ve}.
$$
which gives  \eqref{rapiddecay}.

\medskip

b) Let $\vth\in\L.$ We claim that
$$
\hat\vartheta^\e\subset \la^\ve\cdot T.
$$
After translation, we may assume that $T$ is centered at the origin. Put $ r:=\la^\ve \ge 1.$ What we then have to show is that if $a+\vth^*\subset T,$ then $a+r\vth^*\subset rT,$ i.e.,
$$
r^{-1} a+\vth^*\subset T.
$$
In order to show this, note that there is a  cuboid $(a+T_\vth)\subset T$ also centered at $a$ with the same axes as $T$   which contains $a+\vth^*,$ i.e., $\vth^*\subset T_\vth.$ But then it is easily seen that  also $r^{-1} a+T_\vth\subset T$ for all $r\ge 1,$ which implies our claim.

Thus, by a), for every $N\in \NN,$
$$
\|\mu_\L\|_{L^1(\RR^3\setminus \la^\ve\cdot T)} \leq \sum\limits_{\vth\in \L} \|\mu_\vth\|_{L^1(\RR^3\setminus\hat\vartheta^\e)}
\le C_N\la^{-N}\sharp \L.
$$
But, since $\lambda\geq R^\gamma,$ with $\gamma>0,$  we have that $\rho_1,\rho_2\geq \lambda^{-M_\gamma}$ for some suitable constant $M_\gamma,$ so that the number $\sharp \L$ of L-boxes in $\Ga_R$ is $\mathcal O(\la^{N_\gamma})$   for some exponent $N_\ga$ depending only on $\gamma.$ This proves also b).
\end{proof}

\smallskip

\section{The proof of Theorem \ref{mainthm}} 

Let $\phi\in\F^{\kappa_1,\kappa_2,\gamma,\C}(\mathbf\Gamma_R^\pm),$ and recall from Subsection \ref{Lboxdecomp} that we can decompose  each sector 
$\theta$ of angular width $R^{-1/2}$ in $\Gamma_R^\pm$ into at most  $N(\lambda,R)$ L-boxes $\vartheta,$
i.e., 
$$
\sharp\{\vth\subset \theta\}\lesssim N(\lambda,R)\sim  \frac{R^{-3/2}}{\rho_1\rho_2}.
$$
Since there are only $\mathcal O(\log R)$ dyadic parameters $s$  in the range $R^{-1/2}\leq s\leq 1,$ in order to prove 
Theorem \ref{mainthm}, in view of estimate \eqref{gwzest} it will thus suffice to show that for every $\ve>0$ we can bound  
 
\begin{equation}\label{mainfinest}
\|T^\la_R f\|_{GWZ,s}\le C_\ve R^\ve \max_\theta \, \sharp \{\vth\subset \theta\}^{\frac 12} \|f\|_{L^4(\RR^3)},
\end{equation}
uniformly in $s\in[R^{-1/2},1]$. Given  our previous discussions, our arguments will also show that the constants $C_\epsilon$ can be chosen to depend only on $\epsilon$ and a finite number of the constants $C_\al$ from $\C.$
\medskip

Let us thus fix $s\in[R^{-1/2},1]$, and choose any sector $\tau\subset \Ga^\pm_R$ of angular width $s.$ Moreover, suppose   $\vth\subset \tau$ is  any L-box contained in  $\tau.$
\smallskip

Ignoring for the moment  the translation by $\lambda\nabla\phi(\xi_\vth)$ in \eqref{vthhat}, by replacing the multiplier 
$m_\vth$ by $\chi_\vth,$ we observe that $\widehat{\chi_\vth}$ is ``essentially'' supported in  $\vartheta^*,$ und thus we essentially have 

$$
(\widehat{\chi_\vth}\ast f)\chi_{U_{\tau,R}}=(\widehat{\chi_\vth}\ast (\chi_ {U_{\tau,R}-\vartheta^{\ast}}f))\chi_{U_{\tau,R}}.
$$ 
The sets  $U_{\tau,R}$ and $\vartheta^{\ast}=-\vartheta^{\ast}$ may have quite different dimensions, but their Minkowski difference 
$$U_{\tau,R}-\vartheta^{\ast}=U_{\tau,R}+\vartheta^{\ast}$$ 
is essentially a rectangular box of dimensions
$\q_1\times \q_2\times \q_3$ with respect to the coordinates given by the orthonormal frame $E^\vth:=E_{\xi_\vth},$ where 
\begin{eqnarray}\label{qjs0}
\begin{split}
\q_1&:=R\vee\rho_1^{-1} = \rho_1^{-1},\\
\q_2&:=(sR)\vee\rho_2^{-1},\\
\q_3&:=(s^2R)\vee 1) =s^2R,
\end{split}
\end{eqnarray}
 since $\rho_1=\tilde\rho_1\wedge R^{-1}\leq R^{-1}$ and $s^2R\ge 1.$ 
 
 In view of the definition of the ``horizontal'' size $\q_2,$ we shall distinguish the following two cases:

{\bf Case I:} $sR \tilde \rho_2> 1,$ i.e., $\la<  s^2 R^{2+\gamma-2\ka_2},$ 
in which $\q_2=sR,$ 

 and 
\smallskip

 {\bf Case II:} $sR \tilde \rho_2\le 1,$ i.e., $\la\ge s^2 R^{2+\gamma-2\ka_2},$ in which $\q_2=\rho_2^{-1}.$ 
 \smallskip

 \subsection{Case I: $sR \tilde \rho_2>  1$ }
Note that in this case, we have 
  \begin{eqnarray}\label{qjs0I}
\begin{split}
\q_1&:=R\vee\rho_1^{-1} = \rho_1^{-1},\\
\q_2&:=sR,\\
\q_3&:=(s^2R)\vee 1) =s^2R.
\end{split}
\end{eqnarray} 
Since here $\frak q_2=sR\ge \rho_2^{-1},$ (at least in  the ``horizontal'' direction $x_2$) about  $N_0:=
sR \rho_2\ge 1$ translates of the dual boxes $\vth^*$ may fit into the aforementioned box of dimensions
$\q_1\times \q_2\times \q_3$ containing $U_{\tau,R}+\vartheta^{\ast}.$ This suggest to group the $\vth$s into larger sectors of angular width 
$$
\al:=N_0\rho_2=sR\rho_2^2.
$$

Since each cap $\theta$ has angular width $R^{-1/2},$ we then consider  the two sub-cases I.a, where $\al\le R^{-1/2},$ and  I.b, where $\al> R^{-1/2}.$

\smallskip
 \subsection{Sub-case I.a: $\al\le R^{-1/2}$ } This is the sub-case where $sR^{3/2}\rho_2^2\le 1.$ Here, we decompose each given cap $\theta$ into  {\it caps  $\si$} of  dimensions $\Delta\times \al\times 1,$ where 
 $$
 \Delta:=\al^2\vee \rho_1 \le R^{-1},
 $$
 in a similar way as we had decomposed $\theta$ into the $\vth$s  in a smooth way in \eqref{decompvth}, i.e., 
 by first decomposing into angular sectors  $\pi$ of width $\al,$ and then decomposing each such sector $\pi$ ``radially''  into shells  of thickness  $\Delta$ by means of level intervals of $F.$
 
Note that since $\Delta\ge \al^2,$ the caps $\si$ are effectively rectangular boxes with respect to the orthonormal frame $E^\si:=E_{\xi_\si,}$ if $\xi_\si$ denotes the center of $\si.$
 \medskip

Note next that each sector $\pi$ will decompose  into about 
$$N_1:=\frac{\Delta}{\rho_1}=\frac {(sR\rho_2^2)^2}{\rho_1}\vee 1$$ 
boxes $\vth.$ The corresponding translated dual boxes $\hat\vth$ will then be contained in a region of $x_1$-dimension at most $N_1 \rho_1^{-1}=\rho_1^{-2}\Delta.$  
\color{blue} For this reason, we  shall  here work with boxes of larger size $q_1:=N_1 \rho_1^{-1}$ in place of $\frak q_1=\rho_1^{-1}.$ I.e., given an L-box $\vth$ with center $\xi_\vth,$ in place of the ``box'' $U_{\tau,R}+\vartheta^{\ast},$ we shall here define  $Q^\vartheta_{\tau,R}=Q^{\vartheta,\ve}_{\tau,R}$  to be the box of dimensions $q_1\times q_2\times q_3$ with respect to the coordinates given by the orthonormal frame $E^\vth:=E_{\xi_\vth},$ \color{black}
   where 
  \begin{eqnarray}\label{qjs2}
\begin{split}
q_1&:= c\lambda^{\e}\rho_1^{-2}\Delta,\\
q_2&:=c\lambda^{\e}sR,\\
q_3&:=c\lambda^{\e}  s^2 R.
\end{split}
\end{eqnarray}
Here, $c\gg 1$ is a constant chosen so large that $U_{\tau,R}+\vartheta^{\ast}\subset Q^{\vartheta,0}_{\tau,R}$ for any 
$\vth\subset \theta\subset \tau$
(note that $Q^\vartheta_{\tau,R}=\la^{\ve}\cdot Q^{\vartheta,0}_{\tau,R}$).

The formal definiton of the boxes $Q^\vartheta_{\tau,R}$ is
\begin{equation}
	  Q^\vartheta_{\tau,R} :=\{x=y\trans E^\vartheta : |y_j|\leq q_j,\ j=1,2,3 \}
\end{equation}
(note that we are keeping here our  convention that $x,y$ are row vectors, and that $\trans (E^\vartheta \trans y)=y\trans E^\vartheta$).

 \begin{lemma}\label{directionQIa}
 Let $\vartheta_1, \vartheta_2\subset \theta$ be  any L-boxes within a cap  $\theta\subset \tau.$   Then 
 $Q^{\vartheta_1}_{\tau,R}\subset C\cdot Q^{\vartheta_2}_{\tau,R},$ where $C$ is a large constant independent of $R$ and $\lambda$. 
 \end{lemma}
 
  \begin{proof} 
Assume that $x\in Q^{\vartheta_1}_{\tau,R}$, say $x=y^{\vartheta_1}\trans E^{\vartheta_1},$ with $|y^{\vartheta_1}_j|\leq q_j$, $j=1,2,3$. Let us then also write $x=y^{\vartheta_2}\trans E^{\vartheta_2}.$ Then, by definition, 
$$
y^{\vartheta_2}=y^{\vartheta_1}\trans E^{\vartheta_1}E^{\vartheta_2},
$$
and what we  have to show is that there is a suitable constant $C\ge 1$ such that   $|y^{\vartheta_2}_j|\leq C q_j$ for all $j=1,2,3.$ 

To this end, note that, by  Lemma \ref{EEcheck},
\begin{equation}\label{sumqj}
	|y^{\vartheta_2}_j|= \left| \sum_{i=1}^3 y_i^{\vartheta_1}\langle E_i^{\vartheta_1}, E_j^{\vartheta_2}\rangle  \right| 
				\lesssim  \sum_{i=1}^3  q_iR^{-\frac12|i-j|}.
\end{equation}
Let us examine which conditions will be necessary and sufficient for   the right-hand side to be bounded  by $Cq_j$, for $j=1,2,3.$ 

For the summand with $i=j$, $q_j$ is already of the right order. For the summands with $|i-j|=1$, the following four conditions are necessary and sufficient (up to multiplicative constants):
\begin{equation}\label{c1}
R^{-1/2}\le q_2/q_1,
\end{equation}
\begin{equation}\label{c2}
R^{-1/2}\le q_1/q_2,
\end{equation}
\begin{equation}\label{c3}
R^{-1/2}\le q_2/q_3,
\end{equation}
and 
\begin{equation}\label{c6}
R^{-1/2}\le q_3/q_2.
\end{equation}
For the remaining summands with $|i-j|=2$ in \eqref{sumqj}, the required inequalities are 
\begin{equation}\label{c4}
R^{-1}\le q_1/q_3.
\end{equation}
and 
\begin{equation}\label{c5}
R^{-1}\le q_3/q_1.
\end{equation}
Note, however, that these are  immediate consequences of \eqref{c1} -- \eqref{c6}.  
\smallskip

Thus, to prove the lemma, it will suffice to verify the inequalities \eqref{c1} -- \eqref{c6}.
\smallskip
 The inequalities \eqref{c1} and  \eqref{c2} require that 
 $$
 R^{-1/2}\le \frac {q_1}{q_2}=\frac {s R\rho_2^4}{\rho_1^2}\vee \frac 1{sR\rho_1}\le R^{1/2}.
 $$ 
 Since $sR\rho_1\le 1,$ the first inequality is obvious. Next, 
 by Lemma \ref{rhofrak}, 
 $$sR^{3/2} \rho_1\ge sR \rho_2>1,$$
 since we are in Case I. Moreover,  since we are in Case I.a, using again also Lemma \ref{rhofrak}, 
 $$
s R\rho_2^4=(sR^{3/2}\rho_2^2) R^{-1/2} \rho_2^2\le R^{-1/2} \rho_2^2\le R^{1/2} \rho_1^2,
 $$
so that also the second inequality follows.
\smallskip

 Inequalities \eqref{c3} and \eqref{c6} are also clear, since 
  $$
 R^{-1/2} \le \frac {q_3}{q_2} =s \le 1\le R^{1/2}.
 $$
 \end{proof} 

A priori, the orientation of the boxes $Q^{\vartheta,\ve}_{\tau,R}$ depends on $\vartheta\in\theta$. However, if we fix any
$\vth_0\subset \theta,$ and define the rectangular box  $Q^{ \theta}_{\tau,R}=Q^{\theta,\ve}_{\tau,R}$ to be $Q^{\vartheta_0,\ve}_{\tau,R},$ then Lemma \ref{directionQIa} shows that $Q^{ \theta}_{\tau,R}$ is a box of the same dimensions $q_1\times q_2\times q_3,$ such that
 \begin{equation}\label{QvtinQ}
 Q^{\vartheta}_{\tau,R}\subset Q^{ \theta}_{\tau,R} \quad \text{for all}\quad \vth\subset \theta.
\end{equation}
 Note, however, that the orientation in space of $Q^{\theta}_{\tau,R}$ may change with $\theta$.

\medskip
 Consider next
 $$H(\xi):=\la\nabla\phi(\xi).$$
  The following estimates will be shown in Section \ref{variation}: 
   Given any cap $\si\subset\theta,$  consider again the coordinates $\eta$ defined by 
   $$\xi=E^\si\eta.$$
Then  $\tilde \phi(\eta):=\phi(E^\si\eta)$ expresses $\phi$ in the coordinates $\eta.$ We denote  by $\tilde \si$ the box corresponding to $\si$ in the coordinates $\eta$, and express by 
$$\tilde H(\eta):=\la\nabla\tilde \phi(\eta)=H(\xi)E^\si=\sum_{j=1}^3 (H(\xi)E^\si_j) E^\si_j=\sum_{j=1}^3 \tilde H_j(\eta) E^\si_j$$
  the mapping $H$ in the coordinates $\eta.$ Then, Lemma \ref{Hdiff} , which makes use of the small mixed derivative 
  condition (SMD), implies that for any $\eta,\eta'\in \tilde \si,$
\begin{equation}\label{vargrad1}
|\tilde H_j(\eta)-\tilde H_j(\eta')|\ll q_j, \qquad  j=1,2,3.
\end{equation}

\smallskip

 Next, consider translates  $Q=c_Q+Q^{\theta}_{\tau,R}$ of $Q^{\theta}_{\tau,R}$.
For any such box $Q$, we have 
\begin{equation}\label{QoverU2}
	\frac{|Q|}{|U|} =\frac{|Q^{ \theta}_{\tau,R}|}{|U_{\tau,R}|} \simeq  \la^{3\ve}\frac{q_1}{R}=:\om.
\end{equation}
 Note that $\om=\om(\la, q_1,q_2, s,R)$ does not depend on $\tau$ or $\theta.$
\smallskip

Recall next  from \eqref{vthhat} that
$$
\hat\vartheta^\e:=\lambda\nabla\phi(\xi_\vth)+\vartheta^{\ast,\e}=H(\xi_\vth)+\vartheta^{\ast,\e}.
$$
This identity in combination with   \eqref{vargrad1}  and \eqref{QvtinQ} allows us
 to choose a {\it boundedly overlapping covering} $\{Q\}$ of $\R^3$ by translates $Q$ of  $Q^{\theta}_{\tau,R}$ such that for any cap $\si\in \theta$ there exists a 
$Q_\si\translb Q^{\theta}_{\tau,R}$ (where this is the short hand notation for $Q_\si$ being in this covering)  such that
\begin{equation}\label{Qsia}
 \hat\vth^\e\subset Q_\si \quad \text{for every}\quad \vth\subset\si.
\end{equation}
In particular, if we put $Q_\vth:=Q_\si$ for all $\vth\in\si,$ then $\hat\vth^\e\subset Q_\si$ for every $\vth\subset\si.$
  
\smallskip

In a first step, for any Schwartz function $f\in \S(\RR^3),$ we now estimate (for any dyadic $s$ between $R^{-1/2}$ and 1) 
\begin{eqnarray*}
\|f\|_{GWZ,s}^4
&=& \sum_{\angle(\tau)=s} \sum_{U\transl U_{\tau,R}} |U|^{-1} \left[\sum_{\theta\subset\tau}
			\| f_\theta\|_{L^2(U)}^2\right]^2  \\
&\leq &  \sum_{\angle(\tau)=s} \sharp\{\theta\subset\tau\}\sum_{\theta\subset\tau}\sum_{U\transl U_{\tau,R}} |U|^{-1} \| f_\theta \|_{L^2(U)}^4  \\
&\le  &  \sum_{\angle(\tau)=s} \sharp\{\theta\subset\tau\}\sum_{\theta\subset\tau}\sum_{Q\transl Q^{\theta}_{\tau,R}} \frac{|Q^{\theta}_{\tau,R}|}{|U_{\tau,R}|} |Q|^{-1}\sum_{U\subset Q} \| f_\theta\|_{L^2(U)}^4  \\
&\le& \om \sum_{\angle(\tau)=s} \sharp\{\theta\subset\tau\}\sum_{\theta\subset\tau}\sum_{Q\transl Q^{\theta}_{\tau,R}} |Q|^{-1}\big(\sum_{U\subset Q} \| f_\theta\|_{L^2(U)}^2\big)^2,
\end{eqnarray*}
where we have used Cauchy-Schwarz in the first inequality . Thus, for any $f\in \S(\RR^3),$ 
\begin{equation}\label{step1e}
\|f\|_{GWZ,s}^4\le \om\max_{\angle(\tau)=s} \sharp\{\theta\subset\tau\} \sum_{\angle(\tau)=s}\sum_{\theta\subset \tau}\sum_{Q\transl Q^{\theta}_{\tau,R}} |Q|^{-1} \| f_\theta\|_{L^2(Q)}^4.
\end{equation}
Note that since the multiplier $m^\la_R$ lies in $C_0^\infty,$ this estimate applies as well to $T^\la_R f$ in place of $f,$ where 
$$
(T^\la_R f)_\theta=T^\la_R f_\theta.
$$

Recall next from \eqref{Tdecomp} that
$$
T^\la_R f=\sum\limits_{\vth} f*\mu_\vth,
$$
where $\widehat{\mu_\vth}=m_\vth.$ 
Since all $\vth\subset\si$ have the same $Q_\vth=Q_\si,$ we pass to the coarser decomposition
$$
T^\la_R f=\sum\limits_{\si} f*\mu_\si,
$$
where $\widehat{\mu_\si}:=m^\la_R \chi_\si,$ with $\chi_\si$ denoting  a suitable smooth cut-off function which essentially localizes to  $\si.$ Correspondingly, we define $f_\si$ by $\widehat{f_\si}:=\chi_\si \hat f.$

\begin{lemma}[Key Lemma]\label{key}
Let $Q\translb Q^{\theta}_{\tau,R},$  and decompose 
$$
T^\la_R f_\theta=\sum\limits_{\si\subset \theta} (f_\si\charac_{(Q-Q_{\si})^c})*\mu_\si+\sum\limits_{\si\subset\theta} 
(f_\si\charac_{Q-Q_{\si}})*\mu_\si.
$$
Then the following estimates hold true:
\begin{equation}\label{sclaima}
\sum_{Q\transl Q^{\si}_{\tau,R}}|Q|^{-1}\big \|\sum_{\si\subset\theta}(f_\si\charac_{(Q-Q_{\si})^c})\ast\mu_\si\big\|^4_{L^2(Q)} 
\lesssim \lambda^{-N}\sum_{\si\subset\theta} \|f_\si\|_4^4
\end{equation}
for every $N\in\NN,$ 
and 
\begin{equation}\label{bigclaima}
\sum_{Q\transl Q^{\theta}_{\tau,R}}|Q|^{-1}\big\|\sum\limits_{\si\in\theta} 
(f_\si\charac_{Q-Q_{\si}})*\mu_\si\big\|_{L^2(Q)}^4\lesssim \sharp \{\si\subset\theta\} \sum_{\si\subset\theta} \|f_\si\|_4^4.
\end{equation}
\end{lemma}

 \begin{proof} 
 Expanding 
 $$
 T^\lambda_R f_{\theta} = \sum_{\si\subset\theta} f_\si\ast\mu_\si, 
 $$
we observe that for any $x\in Q,$
\begin{eqnarray*}
	|(f_\si\charac_{(Q-Q_{\si})^c})\ast\mu_\si (x)| 
	&=& \left|\int f_\si(x-y)\charac_{(Q-Q_\si)^c}(x-y)\mu_\si (y) dy\right| \\
	&=& \left|\int_{Q_\si^c} f_\si(x-y)\charac_{(Q-Q_\si)^c}(x-y)\mu_\si (y) dy\right| 	 \\
	&\leq& \int_{Q_\si^c} |f_\si(x-y)|\cdot|\mu_\si (y)| dy \\
	&=&|f_\si|\ast|\mu_\si\charac_{Q_\si^c}|(x),
\end{eqnarray*}
because the $y$-integral over $Q_\si$ vanishes  for $x\in Q$. Therefore, using  Cauchy-Schwarz, 
 
\begin{eqnarray*}	
	\sum_{Q\transl Q^{\theta}_{\tau,R}}|Q|^{-1}\big \|\sum_{\si\subset\theta}(f_\si\charac_{(Q-Q_{\si})^c})\ast\mu_\si\big\|^4_{L^2(Q)}\
	&\leq& \sum_{Q\transl Q^{\theta}_{\tau,R}}\big\|\sum_{\si\subset\theta}|f_\si|\ast|\mu_\si\charac_{Q_\si^c}|\big\|^4_{L^4(Q)}  \\
	&\lesssim& \|\sum_{\si\subset\theta} |f_\si|\ast|\mu_\si\charac_{Q_\si^c}|\|^4_{L^4(\R^3)}  \\
	&\leq& \sum_{\si\subset\theta} \|f_\si\|^4_4 \|\mu_\si\charac_{Q_\si^c}\|^4_1. \\
	\end{eqnarray*}
But, by \eqref{Qsia} and Lemma \ref{4iersupp2} b), we have 
$$
\|\mu_\si\charac_{Q_\si^c}\|_1\lesssim \la^{-N}\quad  \text{for any }  N\in\NN,
$$ 
so that we obtain \eqref{sclaima}.
\smallskip

To prove the second estimate, let us write $f_\si^{Q-Q_\si}:=f_\si\charac_{Q-Q_{\si}}.$ Then, by Plancherel, 
\begin{eqnarray*}
 \|\sum_{\si\subset\theta} f_\si^{Q-Q_\si} \ast\mu_\si \|_{L^2(Q)}^2 
	&\lesssim& \|\sum_{\si\subset\theta} \widehat {f_\si^{Q-Q_\si}}m_\si \|_{L^2(\R^3)}^2 \\
	&\simeq& \sum\limits_{\si\subset \theta} \|\widehat{ f_\si^{Q-Q_\si}}m_\si \|_{L^2(\R^3)}^2 \\
	&\lesssim& \sum_{\si\subset\theta} \|f_\si\|_{L^2(Q-Q_\si)}^2.
\end{eqnarray*}
Thus,
\begin{eqnarray*}
\sum_{Q\transl Q^{\theta}_{\tau,R}}|Q|^{-1} \|\sum_{\si\subset\theta}(f_\si^{Q-Q_\si})\ast\mu_\si\|^4_{L^2(Q)}
&\le& \sum_{Q\transl Q^{\theta}_{\tau,R}}|Q|^{-1} \Big(\sum_{\si\subset\theta} \|f_\si\|_{L^2(Q-Q_\si)}^2\Big)^2\\
&\le& \sharp\{\si\subset \theta\}\sum\limits_{\si\subset\theta} \sum_{Q\transl Q^{\theta}_{\tau,R}}|Q|^{-1} 
 \|f_\si\|_{L^2(Q-Q_\si)}^4\\
&\le&\sharp\{\si\subset \theta\}\sum\limits_{\si\subset\theta} \sum_{Q\transl Q^{\theta}_{\tau,R}}  \|f_\si\|_{L^4(Q-Q_\si)}^4\\
&\le&\sharp\{\si\subset \theta\}\sum\limits_{\si\subset\theta} \|f_\si\|_{L^4(\RR^3)}^4.
\end{eqnarray*}
In the last inequality, we have used  that if $Q_1$  and $Q_2$ are  two boxes  centered at $z_1$ and $z_2$, respectively, then  their Minkowski-difference $Q-Q'$ is the  doubling of the box centered at $z_1-z_2.$
This implies \eqref{bigclaima}.
\end{proof} 

In combination with \eqref{step1e} (applied to  $T^\la_R f$), we now see that 
\begin{eqnarray*}
\|T^\lambda_Rf\|_{GWZ,s}^4&\lesssim& \om\max\sharp\{\theta\subset\tau\} \max\sharp\{\si\subset \theta\}\sum_{\angle(\tau)=s}\sum_{\theta\subset \tau}\sum\limits_{\si\subset\theta}\|f_\si\|_{L^4(\RR^3)}^4\\
&\lesssim& \om\max\sharp\{\theta\subset\tau\} \max\sharp\{\si\subset \theta\}
\sum\limits_{\si}\|f_\si\|_{L^4(\RR^3)}^4.
\end{eqnarray*}

But, by interpolation between the  trivial estimates $
\sum\limits_{\si} \|f_\si\|_{L^2(\RR^3)}^2\lesssim \|f\|_{L^2(\RR^3)}^2
$
and 
$
\max\limits_{\si} \|f_\si\|_{L^\infty(\RR^3)}\lesssim \|f\|_{L^\infty(\RR^3)}
$
(based on Plancherel's theorem, and the fact that the $\si$s are rectangular boxes),
we obtain
$$
\sum\limits_{\si} \|f_\si\|_{L^4(\RR^3)}^4\lesssim \|f\|^4_{L^4(\RR^3)}.
$$

Thus, we have established the crucial  estimate
\begin{equation}\label{crest1}
\|T^\lambda_Rf\|_{GWZ,s}^4\lesssim \la^\ve  \om\max\sharp\{\theta\subset\tau\} \max\sharp\{\si\subset \theta\}
 \|f\|^4_{L^4(\RR^3)}
 \end{equation}
for every $\ve>0.$

Finally, note that
\begin{eqnarray*}
\om&=&\la^{3\ve} \frac {q_1}{R} = c\la^{4\ve} \frac {\Delta}{R\rho_1^2},\\
\max\sharp\{\theta\subset\tau\}&\lesssim&\frac {s}{R^{-1/2}}=sR^{1/2},\\
\max\sharp\{\si\subset \theta\}&\lesssim&\frac{R^{-3/2}}{\al\Delta}=\frac{R^{-3/2}}  {sR\rho_2^2\Delta}.
\end{eqnarray*}
One then computes that

$$
 \om\max\sharp\{\theta\subset\tau\} \max\sharp\{\si\subset \theta\}\lesssim
 c\la^{4\ve} \frac {R^{-3}}{\rho_1^2\rho_2^2}\lesssim \la^{4\ve} (\sharp\{\vth\subset \theta\})^2,
$$

and by \eqref{crest1} we conclude that, for any $\ve>0,$ 
\begin{equation}\label{finest1}
\|T^\lambda_Rf\|_{GWZ,s}\lesssim \la^\ve (\max_\theta\sharp\{\vth\subset \theta\})^{1/2} \|f\|_{L^4(\RR^3)}.
 \end{equation}
Since $\la\le R^{2+\gamma-2\ka_2}$ in Case I, this implies \eqref{mainfinest}.

\smallskip
 \subsection{Sub-case I.b: $\al> R^{-1/2}$ } This is the sub-case where $sR^{3/2}\rho_2^2>1.$
 We shall here decompose $\Gamma_R$ into  even larger sectors $\Theta$ of angular width $\al=sR\rho_2^2>R^{-1/2}$ than the sectors $\theta$ (note, however, that such sectors $\Theta$ are in fact ``curved boxes'').  Choosing  $N_2\in\NN_{\ge 1}$ such that
 $$
 N_2\sim sR^{3/2}\rho_2^2,
 $$ 
we may  assume without loss of generality that each  sector $\Theta$ decomposes into $\sim N_2$ caps $\theta.$
Moreover, since $\rho_2\leq R^{-1/2},$ we have  $\al\leq s,$ and thus we  assume also that each sector $\tau$ can be decomposed into sectors  $\Theta$.
\smallskip

 Let us denote by
$$
\sharp\{\tth\subset\tau\}
$$
the minimal number of   sectors $ \tth$ needed to cover  the sector $\tau$.  This number  is essentially 
 independent  of $\tau$. 
\smallskip

Note next that  each sector of angular width $\rho_2$  decomposes  into $N_1$ L-boxes $\vth,$ where
$$
N_1\sim\frac{R^{-1}}{\rho_1}=\frac 1{R\rho_1}\ge 1.
$$

\begin{remark}\label{siTheta}
Observe that the angular sectors $\Theta$ have angular width $\al>R^{-1/2}$  (so that $\al^2>R^{-1}$), but thickness $R^{-1}.$  In analogy with the definition of $\Delta$ in Sub-case I.a, it is thus  natural to define here $\Delta:=R^{-1},$ so that in both  sub-cases
$$
 \Delta:=(\al^2\vee \rho_1)\wedge \
 R^{-1}.
 $$

This  shows that the sectors $\Theta$ can be seen as the natural analogues of the caps $\si$ from Sub-Case I.a.
\end{remark}
In particular with this definition of $\Delta:=R^{-1},$ we again have that $N_1\sim\frac{\Delta}{\rho_1}$.

\smallskip
We shall   then again enlarge $\q_1$ by the factor $N_1.$  I.e., given an L-box $\vth\in\tth,$ we shall accordingly here define  the boxes $Q^\vartheta_{\tau,R}=Q^{\vartheta,\ve}_{\tau,R}$  of dimensions $q_1\times q_2\times q_3$ with respect to the coordinates given by the orthonormal frame $E^\vth,$
  where as in \eqref{qjs2}
\begin{eqnarray}\label{qjsb}
\begin{split}
q_1&:= c\lambda^{\e} \rho_1^{-2}\Delta,\\
q_2&:=c\lambda^{\e}sR,\\
q_3&:=c\lambda^{\e}  s^2 R.
\end{split}
\end{eqnarray}

In analogy to Lemma \ref{directionQIa}, we now have

\begin{lemma}\label{directionQIb}
 Let $\vartheta_1, \vartheta_2\subset \tth$ be any L-boxes within a sector  $ \tth\subset \tau.$   Then 
 $Q^{\vartheta_1}_{\tau,R}\subset C\cdot Q^{\vartheta_2}_{\tau,R},$ where $C$ is a large constant independent of $R$ and $\lambda$.
 \end{lemma}
 
  \begin{proof} 
We can argue in a similar way as in the proof of Lemma \ref{directionQIa}, but must allow now for rotations of angles $\be\le \al $ in place of  $\be\le R^{-1/2}.$ This leads to the following conditions that  need to be verified:

\begin{equation}\label{c1b}
\al\le q_2/q_1,
\end{equation}
\begin{equation}\label{c2b}
\al\le q_1/q_2,
\end{equation}

\begin{equation}\label{c3b}
\al\le q_2/q_3,
\end{equation}
and 
\begin{equation}\label{c6b}
\al\le q_3/q_2,
\end{equation}
as well as 
\begin{equation}\label{c4b}
\al^2\le q_1/q_3.
\end{equation}
and 
\begin{equation}\label{c5b}
\al^2\le q_3/q_1.
\end{equation}
Again, the last two  inequalities are consequences of the first four inequalities.

\smallskip

As for \eqref{c1b} and \eqref{c2b}, note that, again by Lemma \eqref{rhofrak}, and since $s^2R^3\rho_1^2\rho_2^2\le 1,$ 
$$
\al=sR\rho_2^2\le sR^2\rho_1^2=\frac{q_2}{q_1}\le \frac 1{sR\rho_2^2}=\al^{-1}.
$$
Next, \eqref{c3b} and \eqref{c6b} hold true, since 
$$
\al=sR\rho_2^2\le s=\frac {q_3}{q_2}\le \frac 1{sR\rho_2^2} = \al^{-1}.
$$ 
 \end{proof}

As before, a  priori, the orientation of the boxes $Q^{\vartheta,\ve}_{\tau,R}$ depends on $\vartheta\in\tth$. However, if we fix any $\vth_0\subset \tth,$ and define the rectangular box $Q^{ \tth}_{\tau,R}=Q^{\tth,\ve}_{\tau,R}$ to be $Q^{\vartheta_0,\ve}_{\tau,R},$ then Lemma \ref{directionQIb} shows that $Q^{ \tth}_{\tau,R}$ is a box of the same dimensions $q_1\times q_2\times q_3,$ such that
 \begin{equation}\label{QvtinQb}
 Q^{\vartheta}_{\tau,R}\subset Q^{ \tth}_{\tau,R} \quad \text{for all}\quad \vth\subset \tth.
\end{equation}
 Note, however, that the orientation in space of $Q^{\tth}_{\tau,R}$ may change with $\tth$.

\smallskip
 Next, consider translates  $Q=c_Q+Q^{\tth}_{\tau,R}$ of $Q^{\tth}_{\tau,R}$.
For any such box $Q$, we have 
\begin{equation}\label{QoverUb}
	\frac{|Q|}{|U|} =\frac{|Q^{ \tth}_{\tau,R}|}{|U_{\tau,R}|} \simeq  \la^{3\ve}\frac{q_1}{R}=\la^{3\ve}\frac{1}{R^2\rho_1^2}=:\om.
\end{equation}
 Note that $\om=\om(\la, q_1,q_2, s,R)$ does not depend on $\tau$ or  $\Theta.$
\smallskip

Next,  in a similar way is in the previous sub-case, in the coordinates $\eta$ given by $\xi=E^\Theta \eta,$ where $E^\Theta:=E_{\xi_\Theta}$ denotes the orthonormal frame at the center $\xi_\Theta$ of $\Theta,$  the vector field $H=\la\nabla\phi$ is given  by 
$$\tilde H(\eta)=H(\xi)E^\Theta,$$
and again by  Lemma \ref{Hdiff}, we see that for any $\eta,\eta'\in \tilde \Theta,$
\begin{equation}\label{vargradb}
|\tilde H_j(\eta)-\tilde H_j(\eta')|\ll q_j, \qquad  j=1,2,3.
\end{equation}
Here, $\tilde \Theta$ represents $\Theta$ in the coordinates $\eta.$

The identity \eqref{vthhat}, \eqref{vargradb}  and \eqref{QvtinQb} allow us to choose a {\it boundedly overlapping covering} $\{Q\}$ of $\R^3$ by translates $Q$ of  $Q^{\tth}_{\tau,R}$ such 
\begin{equation}\label{Qsib}
 \hat\vth^\e\subset Q_\tth \quad \text{for every}\quad \vth\in\tth.
\end{equation}
In particular, if we put $Q_\vth:=Q_\tth$ for all $\vth\in\tth,$ then $\hat\vth^\e\subset Q_\tth$ for every $\vth\in \tth.$

\medskip

Our next goal will be to prove an analogue of \eqref{step1e}, based on randomization: 

Let us fix a boundedly overlapping covering $D_R=\{\theta\}_\theta$ of $\Ga_R$ by caps $\theta.$ 
If 
$\epsilon=\{\epsilon_\theta\}_{\theta\in D_R}$  is any mapping $\epsilon:D_R\to \{-1,1\}$  determining a choice  of signs $\epsilon_\theta=\pm 1,$ then for any function $g(\epsilon)$ depending on $\epsilon,$ we denote by ${\rm Av}_\epsilon g$ the average 
$${\rm Av}_\epsilon g \ := \frac1{2^M} \sum_{\epsilon} g(\epsilon)$$
over all sign sequences $\epsilon:D_R\to \{-1,1\}$, where  $M:=\sharp D_R.$ 
For $f\in \S,$ we put 
$$
f^\epsilon :=\sum_{\theta\in D_R} \epsilon_\theta f_\theta.
$$
 Note  that $(T^\la_R f)^\epsilon=T^\la_R f^\epsilon.$ Then, by Khintchine,
$$
\big\|(\sum_{\theta\subset\Theta}| f_\theta|^2)^\frac12\big\|^2_{L^2(U)}={\rm Av}_\epsilon\|f^\epsilon_{\Theta} \|^2_{L^2(U)}.
$$

\begin{lemma}\label{step1b}
For any dyadic $s$ in $[R^{-\frac 12},1]$, we have for any $f\in\S$
$$
\|f\|_{GWZ,s}^4\le \om\, {\rm Av}_\epsilon \max_{\angle(\tau)=s} \sharp\{\Theta\subset\tau\} \sum_{\angle(\tau)=s}\sum_{\Theta\subset \tau}\sum_{Q\transl Q^{\Theta}_{\tau,R}} |Q|^{-1} \| f^\epsilon_\Theta\|_{L^2(Q)}^4.
$$
\end{lemma}

\begin{proof} 
We have 
\begin{eqnarray*}
\|f\|^4_{GWZ,s}
&=& \sum_{\angle(\tau)=s} \sum_{U\transl U_{\tau,R}} |U|^{-1} \left[\sum_{\Theta\subset \tau}\sum_{\theta\subset\Theta}\|f_{\theta} \|^2_{L^2(U)}\right]^2  \\
&=&\sum_{\angle(\tau)=s} \sum_{U\transl U_{\tau,R}} |U|^{-1} \left[\sum_{\Theta\subset \tau}
{\rm Av}_\epsilon\|f^\epsilon_{\Theta} \|^2_{L^2(U)}\right]^2   \\
&\le& {\rm Av}_\epsilon\sum_{\angle(\tau)=s} \sum_{U\transl U_{\tau,R}} |U|^{-1} \left[\sum_{\Theta\subset \tau}
\|f^\epsilon_{\Theta} \|^2_{L^2(U)}\right]^2     \\
&\leq &  {\rm Av}_\epsilon \ \sum_{\angle(\tau)=s} \sharp\{\Theta\subset\tau\}\sum_{\Theta\subset\tau}\sum_{U\transl U_{\tau,R}} |U|^{-1}\|f^\epsilon_{\Theta} \|^4_{L^2(U)},  
\end{eqnarray*}
where we have used Jensen's inequality  in the first inequality, and Cauchy-Schwarz in the second, 
hence
\begin{eqnarray*}
\|f\|^4_{GWZ,s}
&\le  &  {\rm Av}_\epsilon \sum_{\angle(\tau)=s} \sharp\{\Theta\subset\tau\}\sum_{\Theta\subset\tau}\sum_{Q\transl Q^{\Theta}_{\tau,R}} \frac{|Q^{\Theta}_{\tau,R}|}{|U_{\tau,R}|} |Q|^{-1}\sum_{U\subset Q} \|f^\epsilon_{\Theta} \|^4_{L^2(U)}    \\
&\le  & \om {\rm Av}_\epsilon \sum_{\angle(\tau)=s} \sharp\{\Theta\subset\tau\}\sum_{\Theta\subset\tau}\sum_{Q\transl Q^{\Theta}_{\tau,R}}  |Q|^{-1}\big(\sum_{U\subset Q} \|f^\epsilon_{\Theta} \|^2_{L^2(U)} \big)^2     \\
&\le& \om  {\rm Av}_\epsilon  \sum_{\angle(\tau)=s} \sharp\{\Theta\subset\tau\}\sum_{\Theta\subset\tau}\sum_{Q\transl Q^{\Theta}_{\tau,R}} |Q|^{-1} \|f^\epsilon_{\Theta} \|^4_{L^2(Q)} . 
\end{eqnarray*}
\end{proof} 

Applying Lemma \ref{step1b} to $T^\la_R f$ in place of $f,$ we see  that it will suffice to prove that the subsequent estimates will hold uniformly for functions of the form $f^\epsilon :=\sum_{\theta\subset\Gamma_R} \epsilon_\theta f_\theta,$ for any choice of signs $\epsilon_\theta=\pm 1.$

\smallskip

\begin{lemma}\label{keyb}
Let $Q\translb Q^{\Theta}_{\tau,R},$  and decompose 
$$
T^\la_R f^\epsilon_\Theta= \sum_{\theta \subset\Theta}(f^\epsilon_\theta\charac_{(Q-Q_{\Theta})^c})*\mu_\theta+
\sum_{\theta \subset\Theta}(f^\epsilon_\theta\charac_{Q-Q_{\Theta}})*\mu_\theta.
$$

Then the following estimates hold true:
\begin{equation}\label{sclaimb}
\sum_{Q\transl Q^{\Theta}_{\tau,R}}|Q|^{-1} \|(\sum_{\theta \subset\Theta}f^\epsilon_\theta\charac_{Q-Q_{\Theta}})*\mu_\theta)\|^4_{L^2(Q)} 
\lesssim \lambda^{-N}\sum_{\theta\subset\Theta} \|f^\epsilon_\theta\|^4_4=\lambda^{-N}\sum_{\theta\subset\Theta} \|f_\theta\|^4_4
\end{equation}
for every $n\in\NN,$ 
and 
\begin{equation}\label{bigclaimb}
\sum_{Q\transl Q^{\Theta}_{\tau,R}}
|Q|^{-1}\|(\sum_{\theta \subset\Theta}f^\epsilon_\theta \charac_{Q-Q_{\Theta}})*\mu_\theta\|_{L^2(Q)}^4\lesssim  
\sum_{\theta\subset\Theta} \|f^\epsilon_\theta\|^4_4=\sum_{\theta\subset\Theta} \|f_\theta\|^4_4.
\end{equation}
\end{lemma}
\begin{proof} 
The proof  is based on a modification of the proof of  Lemma \ref{key}, where we have to replace  $\theta$ by $\Theta$ and $\si$ by  $\theta$.
\smallskip

 Expanding 
 $$
 T^\lambda_R f^\epsilon_{\Theta} = \sum_{ \theta\subset\Theta} f^\epsilon_ \theta\ast\mu_ \theta, 
 $$
 and applying the same arguments that we had used for the proof of \eqref{sclaima}, we obtain
 $$
 \sum_{Q\transl Q^{\Theta}_{\tau,R}}|Q|^{-1} \|(\sum_{\theta \subset\Theta}f^\epsilon_\theta\charac_{Q-Q_{\Theta}})*\mu_\theta)\|^4_{L^2(Q)} 
\lesssim \lambda^{-N}\sum_{\theta\subset\Theta} \|f^\epsilon_\theta\|_4^4=\lambda^{-N}\sum_{\theta\subset\Theta} \|f_\theta\|_4^4.
 $$
 The only point to note here is the following: Since $\rho_1,\rho_2\geq \lambda^{-M_\gamma}$ and $\lambda\geq R^\gamma,$ with $\gamma>0,$ the number of L-boxes $\vth$ in $\theta$ is $O(\la^{M_\gamma})$   for some exponent $M_\ga$ depending only on $\gamma,$ and thus, for every $\theta\in\Theta,$
$$
\|\mu_ \theta\charac_{Q_\Theta^c}\|_1\lesssim \la^{-N}\quad  \text{for any }  N\in\NN,
$$ 
 which completes the proof of \eqref{sclaimb}.
 
 The proof of \eqref{bigclaimb} proceeds in exactly the same way as that of \eqref{bigclaima}.
 \end{proof} 

In combination with Lemma \ref{step1b}, we now see that 
\begin{eqnarray*}
\|T^\lambda_R f\|_{GWZ,s}^4 
&\lesssim& \om\,   \sharp\{\Theta\subset\tau\} \sum_{\angle(\tau)=s}\sum_{\Theta\subset \tau}\sum_{\theta\subset\Theta} \|f_\theta\|_4^4\\
&\lesssim& \om\,   \sharp\{\Theta\subset\tau\}\sum_{\theta\subset \Gamma_R} \|f_\theta\|_4^4.\\
\end{eqnarray*}
Applying the same interpolation argument that was used at the end of Sub-case I.a, we thus obtain the following analogue of \eqref{crest1}:

\begin{equation}\label{crest1b}
\|T^\lambda_Rf\|_{GWZ,s}^4\lesssim \la^\ve  \om\,   \sharp\{\Theta\subset\tau\}\|f\|_{L^4(\RR^3)}^4
\end{equation}
for every $\ve>0.$

\smallskip

Finally, recall that $\om=\la^{3\ve}/(R^2\rho_1^2),$ and note that 
$$
\max\sharp\{\Theta\subset\tau\}\lesssim \frac{s}{\al}=\frac{s}{sR\rho_2^2},
$$
so that $\om\,   \sharp\{\Theta\subset\tau\}\lesssim \la^{3\ve}R^{-3}/(\rho_1^2\rho_2^2).$ Thus,  by  \eqref{crest1b},  for any $\ve>0,$ 

\begin{equation}\label{finest1b}
\|T^\lambda_Rf\|_{GWZ,s}\lesssim \la^\ve (\max_\theta\sharp\{\vth\subset \theta\})^{1/2} \|f\|_{L^4(\RR^3)}.
 \end{equation}
Again, since $\la\le R^{2+\gamma-2\ka_2}$ in Case I, this implies \eqref{mainfinest}.

\smallskip
 \subsection{Case II: $sR \tilde \rho_2\le 1$ }
In this case, our previous strategy as to how to define  the boxes $Q$ does not seem to work anymore. According to this strategy,  one  would then  like to  enlarge the second dimension to $\rho_2^{-1}$ (which is bigger than $sR$ in Case II) as in Sub-case I.a.  Regretfully, this turns out to not  be the right choice of the boxes $Q$ in Case II, since the analogue of Lemma \ref{directionQIa} would fail to be true. Indeed, the condition \eqref{c6} would here require that 
$$
R^{-1/2}\le \frac {q_3}{q_2} =s^2R\rho_2,
$$
i.e., that $\rho_2\ge s^{-2}R^{-3/2},$ which will not always be true.

However, there is a different strategy adapted especially to Case II which will not even require a control of the variation of $\la\nabla\phi$,  and thus also not (SMD).   We shall no longer want the boxes $Q$ to be so large that they also contain entire dual boxes $\vth^*,$ and shall enlarge the boxes $U$ of dimensions $R\times sR\times s^2R$ in a different way. 
We shall keep the first dimension $R$ and the last dimension $s^2R$ of $U,$ but shall enlarge its second dimension $sR$ in the following way: 

\smallskip
Since in Case II, $\rho_2\le s^{-1}R^{-1},$ and  since also $s^{-2}R^{-3/2}\le s^{-1}R^{-1},$ we may choose an angle $\al_{II}$ here such that 
\begin{equation}\label{alII}
\rho_2\vee s^{-2}R^{-3/2}\le \al_{II} \le s^{-1}R^{-1}.
 \end{equation}
 Note also that $s^{-1}R^{-1}\le R^{-1/2}.$ 
We may thus decompose each cap $\theta$ into at most $N_0$  ``small'' sectors $\pi$  of angular width $\al_{II},$ 
 where 
 $$
 N_0:= \max_\theta\sharp\{\pi\subset \theta\}\sim \frac {R^{-1/2}}{\al_{II}}.
 $$
 Since $\al_{II}^2\le R^{-1},$ these sectors $\pi$ are then indeed essentially rectangular boxes
 in $\theta,$ of dimensions $R^{-1}\times \al_{II}\times 1.$

We enlarge the boxes $U_{\tau,R}$ accordingly: if $\xi_0\in \theta\subset \tau,$ 
we denote by $Q^{\xi_0}_{\tau,R}$ the box of dimensions 
$$
q_1\times q_2\times q_3:=R\times \frac 1{\al_{II}}\times s^2R
$$ 
centered at the origin whose axes are parallel to the basis vectors of the orthonormal frame $E_{\xi_0}.$ Note that 
$\al_{II}^{-1}\ge sR,$ so that the box $Q^{\xi_0}_{\tau,R}$ is indeed larger than $U_{\tau,R}.$

\smallskip

For these boxes, the following analogue of Lemma \ref{directionQIa} now holds true:

 \begin{lemma}\label{directionQII}
 If $\xi_1, \xi_2\in  \theta,$ then 
 $Q^{\xi_2}_{\tau,R}\subset 100\cdot Q^{\xi_1}_{\tau,R}.$ 
 \end{lemma}
 
  \begin{proof} 
   Analogously to Lemma \ref{directionQIa}, it will suffice to prove the estimates \eqref{c1} --  \eqref{c6}.
   
   But, by \eqref{alII},
   $$
   R^{-\frac 12}\le s^{-2}R^{-\frac 12}\le  \frac {q_1}{q_2}= R\al_{II}\le s^{-1}\le R^{\frac 12},
   $$
   which proves \eqref{c1} and \eqref{c2}.
   
Similarly, 
     $$
  R^{-\frac 12}\le  \frac {q_3}{q_2}= s^2R\al_{II}\le s\le R^{\frac 12},
   $$
   which proves  \eqref{c3} and  \eqref{c6}.
   
  The estimates  \eqref{c4} and  \eqref{c5} are again immediate consequences of the previous estimates.
\end{proof} 

A priori, the orientation of the boxes $Q^{\xi_0}_{\tau,R}$ depends on $\xi_0\in\theta$. However, if we fix any
$\xi_\theta\in \theta,$ and define the rectangular box $Q^{ \theta}_{\tau,R}$ to be $c Q^{\xi_\theta}_{\tau,R},$ where $c:=100,$ then Lemma \ref{directionQII} shows that $Q^{ \theta}_{\tau,R}$ is a box of dimensions $cq_1\times cq_2\times cq_3,$ such that
 \begin{equation}\label{QvtinQ2}
 Q^{\xi_0}_{\tau,R}\subset Q^{ \theta}_{\tau,R} \quad \text{for all}\quad \xi_0\in  \theta.
\end{equation}
 Note, however, that the orientation in space of $Q^{\theta}_{\tau,R}$ may change with $\theta$.

\smallskip
 Next, consider translates  $Q=c_Q+Q^{\theta}_{\tau,R}$ of $Q^{\theta}_{\tau,R}$.
For any such box $Q$, we have 
\begin{equation}\label{QoverU}
	\frac{|Q|}{|U|} =\frac{|Q^{ \theta}_{\tau,R}|}{|U_{\tau,R}|} \simeq  c^3\frac{q_2}{sR}=c^3\frac 1{\al_{II} sR}=:\om.
\end{equation}
\smallskip

Arguing in exactly the same way as in Sub-case I.a., estimate \eqref{step1e} is valid also here:
For any $f\in \S(\RR^3),$ 
\begin{eqnarray}\label{step1eII}
\|f\|_{GWZ,s}^4&\le& \om\max_{\angle(\tau)=s} \sharp\{\theta\subset\tau\} \sum_{\angle(\tau)=s}\sum_{\theta\subset \tau}\sum_{Q\transl Q^{\theta}_{\tau,R}} |Q|^{-1} \| f_\theta\|_{L^2(Q)}^4 \nonumber\\
&\lesssim&N_0 \sum_{\angle(\tau)=s}\sum_{\theta\subset \tau}\sum_{Q\transl Q^{\theta}_{\tau,R}} |Q|^{-1} \| f_\theta\|_{L^2(Q)}^4,
\end{eqnarray}
since  
$$
 \om\max_{\angle(\tau)=s} \sharp\{\theta\subset\tau\}\lesssim \frac 1{\al_{II} sR}\cdot\frac s{R^{-1/2}}=\frac 1{\al_{II} R^{1/2}}=N_0.
$$

Next, we choose $\chi\in\D(\RR)$ symmetric such that $\chi\ge 0, $ $\hat \chi\ge 0$ and $\hat\chi(x)\ge 1$ for $|x|\le 1,$ and define, in the coordinates $y$ defined by $y=xE_{\xi_\theta},$ the following ``smooth version'' of $\charac _{Q^{\theta}_{\tau,R}}$:
$$
\chi_{Q^{\theta}_{\tau,R}}(x):=\hat \chi\big (\frac {y_1}{cq_1}\big) \hat \chi\big (\frac {y_2}{cq_2}\big)\hat \chi\big (\frac {y_3}{cq_3}\big).
$$
Then $\chi_{Q^{\theta}_{\tau,R}}\in \S,$  $\chi_{Q^{\theta}_{\tau,R}}(x)\ge 1$ on $Q^{\theta}_{\tau,R},$ and 
$\|\chi_{Q^{\theta}_{\tau,R}}\|_2\lesssim |Q^{\theta}_{\tau,R}|^{\frac12}.$ We denote by  $\psi$ the inverse Fourier transform of 
$\chi_{Q^{\theta}_{\tau,R}}.$ Then, in the coordinates $\eta $ given by $\xi=E_{\xi_\theta}\eta,$ $\psi$ is given by 
$$
\psi(\xi)=\chi(cq_1\eta_1) \chi(cq_2\eta_2)\chi(cq_3\eta_3),
$$
so $\psi$ is supported in $c_1(Q^{\theta}_{\tau,R})^*,$ where $c_1>1$ is a fixed constant and $(Q^{\theta}_{\tau,R})^*$ is the dual box to $Q^{\theta}_{\tau,R}.$

More generally, if $Q/\hskip-0.1cm/Q^{\theta}_{\tau,R} ,$ i.e., if $Q=c_Q+Q^{\theta}_{\tau,R},$ we define 
$\chi_Q(x):=\chi_{Q^{\theta}_{\tau,R}}(x-c_Q),$ and  $\psi_Q(\xi):=e^{ic_Q\xi} \psi(\xi),$ so that $\psi_Q$ is the inverse Fourier transform of $\chi_Q.$
\smallskip

Note also that the $\psi_Q$ are all supported in a box of dimensions  $ \sim R^{-1}\times \al_{II} \times (s^2R)^{-1},$ where $ (s^2R)^{-1}\le 1,$
whereas the sectors $\pi$ have dimensions $R^{-1}\times \al_{II}\times 1.$

\begin{lemma}\label{passtopi}
For any $f\in\S(\RR^3),$ we have 
$$
\| f_\theta\|_{L^2(Q)}^2\lesssim \sum_{\pi\subset \theta} \| f_\pi\|_{L^2(\chi_Q dx)}^2,
$$
where $\widehat {f_\pi}:=\hat f \charac_\pi.$
\end{lemma}

\begin{proof} 
Since $f_\theta=\sum_{\pi\subset \theta} f_\pi,$ we have 
$$
\| f_\theta\|_{L^2(Q)}^2\le \sum_{\pi_1,\pi_2\subset \theta} \int f_{\pi_1}\overline {f_{\pi_2}} (\chi_Q)^2 dx
= c\sum_{\pi_1,\pi_2\subset \theta} \int \widehat {f_{\pi_1}} *\psi_Q\, \overline {\widehat {f_{\pi_2}} *\psi_Q} \,dx.
$$
Now, since $\widehat {f_{\pi}}$ is supported in $\pi,$ by the support property of the  $\psi_Q$ (note here that, by Lemma \ref{directionQII}, the dual boxes to the $Q$s do also not vary with $\pi$),  we see that the summand 
$\int \widehat {f_{\pi_1}} *\psi_Q\, \overline {\widehat {f_{\pi_2}}*\psi_Q} \,dx$ vanishes unless both $\pi_1$ and $\pi_2$ are contained in a sector of angular width $K\al_{II},$ where $K\ge 1$ is a fixed constant. I.e., for any $\pi_1,$ there is subset $A(\pi_1)$ of sectors $\pi_2$ in $\theta$ of cardinality  $\sharp A(\pi_1)\le K$ such that
\begin{eqnarray*}
\| f_\theta\|_{L^2(Q)}^2&\lesssim&
\sum_{\pi_1\subset \theta}\sum_{\pi_2\in A(\pi_1)} \int \widehat {f_{\pi_1}} *\psi_Q\, \overline {\widehat {f_{\pi_2}}*\psi_Q } \,dx\\
&\lesssim&\sum_{\pi_1\subset \theta}\sum_{\pi_2\in A(\pi_1)} \| \widehat {f_{\pi_1}} *\psi_Q\|_2 \|{\widehat {f_{\pi_2}}*\psi_Q }\|_2\\
&=&\sum_{\pi_1,\pi_2\subset \theta}c_{\pi_1,\pi_2} \| \widehat {f_{\pi_1}} *\psi_Q\|_2 \|{\widehat {f_{\pi_2}}*\psi_Q }\|_2,
\end{eqnarray*}
where 
$$
c_{\pi_1,\pi_2}=c_{\pi_2,\pi_1}:= 
\begin{cases}   1,&    \text{if}\  \pi_2\in A(\pi_1),\\
 0 ,&    \text{otherwise}.
\end{cases}
$$
Thus, by Schur's lemma (or Cauchy-Schwarz), 
$$
\| f_\theta\|_{L^2(Q)}^2\lesssim K\sum_{\pi\subset \theta} \| \widehat {f_{\pi}} *\psi_Q\|^2_2
= K\sum_{\pi\subset \theta} \| f_{\pi}\|^2_{L^2(\chi_Q dx)}.
$$
\end{proof} 

Applying \eqref{step1eII} and this lemma to $T^\la_R f$ in place of $f,$ we obtain
\begin{eqnarray*}
\|T^\la_R f\|_{GWZ,s}^4&\le& N_0 \sum_{\angle(\tau)=s}\sum_{\theta\subset \tau}\sum_{Q\transl Q^{\theta}_{\tau,R}} |Q|^{-1} \big( \sum_{\pi\subset \theta} \| T^\la_R f_\pi\|_{L^2(\chi_Q dx)}^2\big)^2\\ 
&\le& N_0^2 \sum_{\angle(\tau)=s}\sum_{\theta\subset \tau}\sum_{Q\transl Q^{\theta}_{\tau,R}} |Q|^{-1}  \sum_{\pi\subset \theta} \| T^\la_R f_\pi\|_{L^2(\chi_Q dx)}^4\\  
&\lesssim& N_0^2 \sum_{\angle(\tau)=s}\sum_{\theta\subset \tau}\sum_{Q\transl Q^{\theta}_{\tau,R}}   \sum_{\pi\subset \theta} \| T^\la_R f_\pi\|_{L^4(\chi_Q dx)}^4\\  
&=&  N_0^2\sum_{\pi\subset \Gamma_R} \| T^\la_R f_\pi\|_{L^4(\RR^3)}^4, 
\end{eqnarray*}
since $\sum_{Q\transl Q^{\theta}_{\tau,R}} \chi_Q\lesssim 1$ (here we use that $q_i\ge 1$ for i=1,2,3).
\smallskip

Moreover, by Plancherel's theorem we have
$$
 \| T^\la_R f_\pi\|_{L^2(\RR^3)}\lesssim  \|f_\pi\|_{L^2(\RR^3)},
$$
and by applying   Lemma \ref{4iersupp}, we find that
$$
 \| T^\la_R f_\pi\|_{L^\infty}=\|\sum_{\vth\subset \pi} f_\pi * \mu_\vth \|_{L^\infty}\lesssim  \sharp \{\vth\subset \pi\} \|f_\pi\|_{L^\infty}.
$$
Interpolating between these two estimates, we obtain
$$
 \| T^\la_R f_\pi\|_{L^4}\lesssim  \sharp \{\vth\subset \pi\}^{\frac 12} \|f_\pi\|_{L^4},
$$
and thus 
\begin{eqnarray*}
\|T^\la_R f\|_{GWZ,s}&\le &  
\big(N_0 \,   \sharp \{\vth\subset \pi\}\big)^{\frac 12}\big(\sum_{\pi\subset \Gamma_R} \| f_\pi\|_{L^4(\RR^3)}^4\big)^{\frac 14}\\
&\lesssim &   \sharp \{\vth\subset \theta\}^{\frac 12}\big(\sum_{\pi\subset \Gamma_R} \|  f_\pi\|_{L^4(\RR^3)}^4\big)^{\frac 14}.
\end{eqnarray*}
Again, by interpolation, we see that 
$$
\sum_{\pi\subset \Gamma_R} \|  f_\pi\|_{L^4(\RR^3)}^4\lesssim \|f\|^4_{L^4(\RR^3)},
$$
and thus we arrive at the desired estimate
\begin{equation}\label{finest2}
\|T^\la_R f\|_{GWZ,s}\lesssim\max_\theta \sharp \{\vth\subset \theta\}^{\frac 12} \|f\|_{L^4(\RR^3)}.
\end{equation}
The proof of Theorem \ref{mainthm} is thus complete. \qed

\section{The proof of Corollary \ref{mainlpeq}}\label{mainlpeqproof}

For $2\le p\le 4,$ the estimate \eqref{mainlpeq} follows by interpolation between the estimate \eqref{mainthmeq} of Theorem  \ref{mainthm} and the trivial estimate $\|T^\lambda_{R} \|_{L^2\to L^2} \leq C,$ which is immediate by Plancherel's theorem. The range $4/3\le p\le 2$ follows by duality.
\smallskip

Next, if  $p=1,$ we can apply the method of Seeger, Sogge and Stein. To this end, note that since we  can cover $\Ga^\pm_R$ by at most $CR^{1/2}$ caps $\theta,$  $\Ga^\pm_R$ can be  covered by at most $CR^{1/2}N(\la,R)$ L-boxes $\vth.$ By decomposing accordingly $m^\la_R=\sum_\vth m_\vth$ and recalling that by Lemma \ref{4iersupp} the inverse Fourier transforms $\mu_\vth$ of the $m_\vth$ have uniformly bounded $L^1$-norms, we see that
\begin{equation}\label{mainestL1}
\|T^\la_R\|_{L^1\to L^1} \le C R^{1/2}N(\la,R).
\end{equation}
The range $1\le p\le 4/3$ in \eqref{mainlpeq2} then follows by interpolation between this $L^1$-estimate and the previous $L^{4/3}$-estimate. The range $4\le p\le \infty$ follows by duality.
\qed

\section{Implications of  (SMD) and Examples}\label{sec:SMD}
In Case I, where $sR\tilde\rho_2> 1,$  the constant $\frak a$ from $\eqref{adefine}$ is strictly smaller than 1, so that the small mixed derivative condition (SMD) from  \eqref{smd} provides indeed an extra condition.  I.e., if we denote again by 
$\Gamma^\circ_R$ either $\Gamma^\pm_R,$ or $\Gamma_R,$  then we require that 
$$
 |\partial_{n(\xi_0)}\partial_{t(\xi_0)} \la \phi(\xi_0)|\leq C \frak a \rho_1^{-1} \rho_2^{-1}
	 \quad \text{for all}\  \xi_0\in \Ga^\circ_R,
$$
where  here
$$
\frak a=\frac 1{R\rho_2}<1.
$$
We shall show that this condition will allow us to control the variation of the gradient 
$$
H(\xi):=\la\nabla \phi(\xi)
$$
in the way as described by  the estimates \eqref{vargrad1} over the caps $\si$ in Sub-case I.a, and over the larger sectors $\Theta$ as by  \eqref{vargradb} in Sub-case I.b.
\color{black}

\medskip

\subsection{Control  of the variation of $\la\nabla \phi$  in Case I}\label{variation}
To unify the notation in this subsection, it will be convenient to denote the ``curved''   sectors $\Theta$ from Sub-case I.b    here by $\si$ too (compare Remark \ref{siTheta}).

Let $\phi\in \F^{\kappa_1,\kappa_2,\gamma,\C}(\mathbf \Ga^\circ_R),$ with $\Ga^\circ_R$  defined as usually,  
assume that the (SMD)-condition \eqref{smd} is satisfied, and let $\si$ be given. Recall that $\si$ has angular width $\al$, ``thickness'' $\Delta$  and vertical height about  1.

Further recall that  $\alpha=sR\rho_2^2$ and  $\Delta=(\alpha^2\vee\rho_1)\wedge R^{-1},$ and that, by \eqref{qjs2} and \eqref{qjsb},
$$q_1=c\la^{\ve}\rho_1^{-2}\Delta , \ q_2=c\la^{\ve}sR \ \text{ and  } q_3=c\la^{\ve}s^2R.$$
In particular, since $\rho_1\le 1/R$,
\begin{equation}\label{deltarange}
\rho_1\le \Delta \leq R^{-1}.
\end{equation}
The first inequality also allows to show that  $s^2R\le sR\leq \rho_1^{-1}\le \rho_1^{-2}\Delta $, so that
\begin{equation}\label{qorder}
q_3\leq q_2\le q_1.
\end{equation}
Let  $\xi_0$ be any point in $\si,$ and let $E=E_{\xi_0}$  be the orthonormal frame  at $\xi_0,$ i.e., $E=(E_1\vert E_2\vert E_3),$ where
  $$
  E_1=n(\xi_0), \, E_2=t(\xi_0), \, E_3= \xi_0/ |\xi_0 |.
  $$

\smallskip
We observe next that there is a constant $C_1\ge C$ such that  the estimate \eqref{smd} even implies that 
  \begin{equation}\label{smd2}
 |\partial_{n(\xi_0)}\partial_{t(\xi_0)} \la \phi(\xi)|\leq C_1 \frak a \rho_1^{-1} \rho_2^{-1}
	 \quad \text{for all}\  \xi\in \theta_{R^{-1/2}}(\xi_0).
\end{equation}
Indeed, an inspection of the proof of the stability Lemmata \ref{addkappa3n} --  \ref{stability2} reveals that  the assumption \eqref{30sep1631} for 
$\al_1=\al_2=1$  still implies the analogue of \eqref{30sep1632} for the same value of $\al_1$ and $\al_2,$
even in the presence of the additional factor $\frak a$. 
\smallskip

Recall also that by Lemma \ref{addkappa3n}, if we set $\ka_3:= \ka_2-1/2,$   then for every $\beta=(\beta_1,\beta_2,\beta_3)\in \NN ^3$ with $|\beta|=2$ we  have 
$$
|\partial_{n(\xi_0)}^{\beta_1}\partial_{t(\xi_0)}^{\beta_2} \partial_{\xi_0}^{\beta_3} \, \la  \phi(\xi)|\leq C_{\beta} \la R^{\beta_1\kappa_1+\beta_2\kappa_2+\beta_3\kappa_3-\gamma}  \quad \text{for all}\  \xi\in \theta_{R^{-1/2}}(\xi_0).
$$
In particular, when $\beta_3=0,$ then for any $\beta\in\NN^2$ of length $|\beta|=2$ which is different from $(1,1), $ we have 
\begin{equation}\label{smdad}
 |\partial_{n(\xi_0)}^{\beta_1}\partial_{t(\xi_0)}^{\beta_2} \la \phi(\xi)|\leq C_2 \rho_1^{-\beta_1} \rho_2^{-\beta_2}
  \quad \text{for all}\  \xi\in \theta_{R^{-1/2}}(\xi_0).
  \end{equation}

And, if $\beta_3=1,$ then we get
\begin{eqnarray}\label{smdad2}
  |\partial_{t(\xi_0)}^{\beta_2} \partial_{\xi_0}^{\beta_3} \,  \la  \phi(\xi)| &\le&  C_2 \la R^{\ka_2+\ka_3-\ga}=C_2 \la R^{2\ka_2-\ga-\frac 12} \le  C_2  \rho_2^{-2} R^{-\frac 12} \\
   |\partial_{n(\xi_0)}^{\beta_1} \partial_{\xi_0}^{\beta_3} \,  \la  \phi(\xi)| &\le&  C_2 \la R^{\ka_1+\ka_3-\ga}=C_2 \la R^{\ka_1+\ka_2-\ga-\frac 12} \le  C_2  \rho_1^{-1} \rho_2^{-1}R^{-\frac 12} \label{smdad3}
\end{eqnarray}
for all $\xi\in \theta_{R^{-1/2}}(\xi_0)$.
\smallskip

Assume next that  $\c:[\tau_1,\tau_2]\to \si$ is  any smooth circular ``tangential'' curve, i.e.,   
$\dot\c(\tau)=t(\c(\tau))$  for every $\tau\in [\tau_1,\tau_2].$  We then  write $E(\tau):=E_{\c(\tau)}$ for the orthonormal frame at $\xi_\tau=\c(\tau)$. 

\smallskip

We claim that along the curve $\c$, $H:=\la\nabla \phi$ satisfies

\begin{equation}\label{tangderiv}
			\alpha|\langle E_j(\tau),(H\circ\c)'(\tau) \rangle| \ll q_j, \qquad j=1,2,3.
\end{equation}

Indeed, by the chain rule, at $\xi_\tau$, we have by \eqref{smd},  
\begin{eqnarray*}
	|\langle E_1(\tau),(H\circ\c)'(\tau) \rangle| 
	&=& |\langle n(\xi_\tau),(H\circ\c)'(\tau) \rangle| \\
	&=& \lambda|\partial_{n(\xi_\tau)}\partial_{t(\xi_\tau)}\phi(\xi_\tau)| 
	\lesssim  \frak a \rho_1^{-1}\rho_2^{-1} \le \frac 1{R\rho_1\rho_2^2},
\end{eqnarray*}
so that with $\alpha=sR\rho_2^2$, $s\leq1$ and \eqref{deltarange}, 
$$
	\alpha|\langle E_1(\tau),(H\circ\c)'(\tau) \rangle| \lesssim s\rho_1^{-1}\le s\rho_1^{-2}\Delta\ll  sq_1\le q_1.
$$
Similarly, by \eqref{smdad}
\begin{eqnarray*}
	\alpha|\langle E_2(\tau),(H\circ\c)'(\tau) \rangle|
	=\alpha \lambda|\partial^2_{t(\xi_\tau)}\phi(\xi_\tau)| 
	\lesssim \alpha\rho_2^{-2} = sR \ll q_2,
\end{eqnarray*}
and by \eqref{smdad2},
\begin{eqnarray*}
	\alpha|\langle E_3(\tau),(H\circ\c)'(\tau) \rangle|
	\simeq \alpha\lambda|\partial_{t(\xi_\tau)}\partial_{\xi_\tau}\phi(\xi_\tau)| 
	\lesssim sR^{1/2} \leq s^2R \ll q_3.
\end{eqnarray*}

We will use these estimates to prove the following lemma:

\begin{lemma}\label{Hdiff}
Assume that $\phi$ satisfies condition (SMD).	Let $\xi,\xi'\in\sigma$ and $E=E_{\xi} $ be  the orthonormal  frame at $\xi$. If we write $H=\breve H_1E_1+\breve H_2E_2+\breve H_3E_3$, then for all $j=1,2,3$,
	\begin{equation}\label{Hdiffs}
		| \breve H_j(\xi)- \breve H_j(\xi')|\ll q_j.
	\end{equation}
\end{lemma}

Observe that the lemma implies the same kind of estimate \eqref{Hdiffs} if we choose for $E$ the orthonormal  frame  at any other point in $\sigma$.

Note also that Lemma  \ref{Hdiff}  implies the estimates \eqref{vargrad1} and  \eqref{vargradb}, which just express the estimates \eqref{Hdiffs} in different coordinates.

\begin{proof}

We will make use of the fact that
\begin{eqnarray}\label{deltest}
\Delta \leq \rho_1 (sR\rho_2)^2.
\end{eqnarray}
Note to this end that since $\rho_2^2\leq\rho_1$ by \eqref{rho1rho2}, we have 
$$\alpha^2=s^2R^2\rho_2^4\leq \rho_1 (sR\rho_2)^2,$$ 
and,  since $sR\rho_2\geq 1$ in Case I, we  also have 
$$\rho_1 \leq \rho_1 (sR\rho_2)^2.$$
These estimates prove \eqref{deltest}, since, according to Remark \ref{siTheta}, we have that $\Delta\le \al^2\vee \rho_1.$

	To prove the lemma, since  $H$ is homogeneous of degree $0,$ we may assume without loss of generality that  $\xi_3=\xi'_3.$ We will use a two-step argument.
	\smallskip
	
	In a first step, we move from $\xi$ in normal direction to a point $\xi''$, chosen so that $\xi''$ and $\xi'$ are connected by a tangential circular curve $\c$. To that end, let 
	$$\xi''=\xi+u\, n(\xi)=\xi+uE_1,$$
	where $|u|\leq\Delta$. Then, by \eqref{smdad},
	\begin{eqnarray*}
		|\breve H_1(\xi)-\breve H_1(\xi'')|&=&\lambda|\langle E_1,\nabla\phi(\xi)-\nabla\phi(\xi+uE_1)\rangle| \\
		&\leq& \sup\limits_{\tilde\xi\in \theta_{R^{-1/2}}(\xi)}\lambda|\partial_{E_1}^2\phi(\tilde\xi)|\Delta 
		\lesssim\rho_1^{-2}  \Delta \ll q_1.
	\end{eqnarray*} 
	Similarly, using  \eqref{smd2} and \eqref{deltest},
	\begin{eqnarray*}
		|\breve H_2(\xi)-\breve  H_2(\xi'')|
		&\leq&  \sup\limits_{\tilde\xi\in \theta_{R^{-1/2}}(\xi)}\lambda|\partial_{E_1}\partial_{E_2}\phi(\tilde \xi)|\Delta 
		\lesssim \frak{a}\rho_1^{-1}\rho_2^{-1} \Delta \\
		&\le& \frac{\Delta}{R\rho_1\rho_2^2} \leq s^2R \ll q_2.
	\end{eqnarray*} 
	
Finally, by \eqref{smdad3} and \eqref{deltest},
	\begin{eqnarray*}
		|\breve H_3(\xi)-\breve H_3(\xi'')|
		&\leq& \sup\limits_{\tilde\xi\in \theta_{R^{-1/2}}(\xi)} \lambda|\partial_{E_1}\partial_{E_3}\phi(\tilde\xi)|\Delta 
		\lesssim R^{-1/2}\rho_1^{-1} \rho_2^{-1} \Delta \\
		&\leq& s^2R^{3/2} \rho_2 \le s^2R\ll q_3.
	\end{eqnarray*}

	 \smallskip

	We now estimate  how much  $H$ can vary along the curve $\c$ connecting $\xi''$ with $\xi'.$  Write $\xi'=\c(\tau_1)$, $\xi''=\c(\tau_2)$, 
	where $|\tau_1-\tau_2|\lesssim\alpha$.
	Estimates \eqref{tangderiv} and \eqref{qorder} imply
\begin{equation}\label{Hcprime}
		|(H\circ\c)'(\tau)|\ll \frac{\max\{q_1,q_2,q_3\}}{\alpha} = \frac{q_1}{\alpha} \qquad\text{for all}\quad  \tau\in[\tau_1,\tau_2],
	\end{equation}
	and therefore 
$$
		|H(\xi')-H(\xi'')|\leq \int_{\tau_1}^{\tau_2} |(H\circ\c)'(\tau)|\ d\tau \ll q_1.
$$
In particular,
$$
|\breve H_1(\xi')-\breve H_1(\xi'')|\ll q_1.
$$

Next, using \eqref{tangderiv}, we have
	\begin{equation}\label{TH}
		\int_{\tau_1}^{\tau_2} |\langle E_2(\tau),(H\circ\c)'(\tau)\rangle| d\tau
		\ll q_2,
	\end{equation}
	where, regretfully,   the tangent direction $E_2(\tau)$ in the integrand still  depends on the integration variable $\tau$.
	
	However, we claim that the tangent directions do not change too much, that is,
	\begin{equation}\label{diffE2}
		|E_2-E_2(\tau)|\lesssim \alpha\lesssim q_2/q_1\qquad \text{for all } \tau\in [\tau_1,\tau_2].
	\end{equation}
	The second  inequality follows from \eqref{c1} and \eqref{c1b} (in Sub-cases I.a and I.b, respectively). For the first inequality, we look at the components
	$|\langle E_2-E_2(\tau), E_j\rangle|$
	for $j=1,2,3$. By Lemma \ref{EEcheck}, we have for $j\neq 2$
	$$
	|\langle E_2-E_2(\tau), E_j\rangle| = |\langle E_2(\tau), E_j\rangle| \lesssim \alpha^{|2-j|} =\alpha,
	$$
	and for $j=2$
	$$
	|\langle E_2-E_2(\tau), E_2\rangle| = |1-\langle E_2(\tau), E_2\rangle| \lesssim \alpha^{2} \leq \alpha.
	$$
	We conclude that 
	\begin{eqnarray*}
		|\breve H_2(\xi'')-\breve H_2(\xi')| 
		&\le& \int_{\tau_1}^{\tau_2} |\langle E_2,(H\circ\c)'(\tau)\rangle|\ d\tau \\
		&\leq& \int_{\tau_1}^{\tau_2}\big( |\langle E_2(\tau),(H\circ\c)'(\tau)\rangle|
		+ |E_2-E_2(\tau)|\,|(H\circ\c)'(\tau)|\big)\ d\tau  \\
		&\ll& q_2 + \frac{q_2}{q_1} q_1 =2q_2,
	\end{eqnarray*}
	by \eqref{TH}, and  \eqref{Hcprime} and \eqref{diffE2}.

	Finally, again by \eqref{tangderiv}, we have
	\begin{equation}\label{E3prelim}
		\int_{\tau_1}^{\tau_2} |\langle E_3(\tau),(H\circ\c)'(\tau)\rangle|
		\ll q_3.
	\end{equation}
But, again Lemma \ref{EEcheck} implies
\begin{equation}\label{DeltaE}
|\langle E_3-E_3(\tau), E_j(\tau)\rangle| = | \langle E_3, E_j(\tau)\rangle -\delta_{3,j}| \lesssim \alpha^{3-j}.
\end{equation}

Splitting the integral
		\begin{eqnarray*}
		|\breve H_3(\xi'')-\breve H_3(\xi')| 
		&\le& \int_{\tau_1}^{\tau_2} |\langle E_3,(H\circ\c)'(\tau)\rangle|\ d\tau \\
		&\leq& \int_{\tau_1}^{\tau_2}\big( |\langle E_3(\tau),(H\circ\c)'(\tau)\rangle|
		+ |\langle E_3-E_3(\tau),(H\circ\c)'(\tau)\rangle|\big)\ d\tau ,
	\end{eqnarray*}
we already estimated the first integral in \eqref{E3prelim}, and by  expanding the scalar product in the second integral, we obtain
	\begin{eqnarray*}
		&&\int_{\tau_1}^{\tau_2} |\langle E_3-E_3(\tau),(H\circ\c)'(\tau)\rangle| \ d\tau \\
		&\le& \sum_{j=1}^3 \int_{\tau_1}^{\tau_2} |\langle E_3-E_3(\tau),E_j(\tau)\rangle|\,|\langle E_j(\tau),(H\circ\c)'(\tau)\rangle| \ d\tau \\
		&\ll& \alpha^2q_1 + \alpha q_2 + q_3,
	\end{eqnarray*}
where made use of  \eqref{DeltaE} and \eqref{tangderiv}.

But, since $\alpha\leq R^{-1/2}$ in  Sub-case I.a, by  \eqref{c5} in Sub-case I.a, and by \eqref{c5b} in Sub-case I.b,  we have $\al^2q_1\lesssim q_3.$  Similarly,  by \eqref{c6}, respectively \eqref{c6b}, we have $\al q_2\lesssim q_3,$ and thus we  obtain altogether that
$$
|\breve H_3(\xi'')-\breve  H_3(\xi')| \ll q_3.
$$
\end{proof}

 \medskip
\subsection{The condition (SMD) for Examples \ref{mainex} and \ref{phi+phi}}\label{smdex}

From our discussion of Example \ref{mainex} in Subsection \ref{examples} it follows easily that here for all $\xi_0\in \Ga^\pm_R,$ 
\begin{equation}\label{smdexamplega}
 |\partial_{n(\xi_0)}\partial_{t(\xi_0)} \la \phi^\gamma(\xi_0)|\leq C \la R^{-\gamma} R\le \frac 1{R\rho_1^2},
\end{equation}
if we regard $\phi^\ga$ as an element of $\F^{1, \ka_2,\ga},$ with $\ka_2\ge 1/2.$  Note that, again  by \eqref{rho1rho2}, 
$$
\frac 1{R\rho_1^2}\le \frac 1{R\rho_1\rho_2^2},
$$
so that condition (SMD) is satisfied.

\smallskip
More generally, if 
$$
\phi:=\phi_0+\phi^\gamma\in\F^{1,\frac\gamma 2,\gamma}
$$
(compare Example \ref{phi+phi}),  where $1\le\gamma\le 2$ and  $\phi_0\in\F_{class},$ and if $\la\ge R^\gamma,$ then 
one finds that for all $\xi_0\in \Ga^\pm_R,$ 
\begin{equation}\label{smdexamplega+}
 |\partial_{n(\xi_0)}\partial_{t(\xi_0)} \la \phi(\xi_0)|\lesssim \la R^{-\gamma} R+\la\lesssim  \frac 1{R\rho_1^2}+\la. 
\end{equation}

But, here 
$$
\la R\rho_1\rho_2^2\le \la R(\la R^{2-\gamma})^{-\frac 12} \la^{-1}=(\la R^{-\gamma})^{-\frac 12}\le 1,
$$
i.e., $\la\le 1/(R\rho_1\rho_2^2),$ and we see that condition (SMD) is again satisfied.

\subsection{Further examples  concerning (SMD)}

\begin{example}\label{ExSMD1}

Any real, 1-homogeneous function $\phi$ on $\RR^3_\times$ which is of the form  $\phi(\xi)=\psi(\xi_1^2+\xi_2^2-\xi_3^2,\xi_3),$ with a function  $\psi$ which is defined and smooth on $\RR\setminus\{0\}\times \RR\setminus\{0\},$ satisfies the condition (SMD) on $\Gamma_R^\pm.$ In particular, the function 
$$
\phi_\ga(\xi):= \xi_3\left|\frac{\xi_1^2+\xi_2^2-\xi_3^2}{\xi_3^2}\right|^\ga
$$
satisfies (SMD) on $\Gamma_R^\pm$ for any $\ga>0.$
\smallskip

{\rm Indeed, as in the proof of Lemma \ref{EEcheck}, after rotation and dilation we may assume that $\xi_0=\trans (a,0,1),$ with $|a|\ne 1.$ The orthonormal frame at $\xi_0$ is then
$$
 E=E_{\xi_0}= \left(\begin{array}{ccc} s & 0 & sa \\ 0 & 1 & 0 \\ -sa & 0 & s \end{array}\right),
$$
with $s:=(1+a^2)^{-1/2}.$ One then computes that in  the coordinates $\eta$ given by $\xi=E\eta,$ we have 
$$
\phi(\xi)=\tilde\phi(\eta)=\psi \big(s^2(\eta_1+a\eta_3)^2+\eta_2^2-s^2(\eta_3-a\eta_1)^2,s(\eta_3-a\eta_1)\big),
$$
where $\xi_0$ corresponds to $\eta_0:=\trans (0,0,s(a^2+1)).$
It is then clear that we here even have
$$
\partial_{\eta_1}\partial_{\eta_2}\tilde \phi(\eta_0)=0.
$$
}

\end{example}

\begin{example}  The phase function $\phi$, given in the $\eta$-coordinates by 
$$\tilde \phi(\eta):=\frac{\eta_1\eta_2}{\eta_3},$$
 is  homogeneous of degree 1, and, by Example \ref{Fclass},  we have that  $\phi\in\F_{class}\subset \F^{\frac\gamma2,\frac\gamma2,\gamma}$ when $\ga\ge 1.$  Let us assume  that $1\le \ga\le 2,$ and that, as usually, $\la\ge R^\ga.$ 
 
Then  $\phi$ does not satisfy (SMD) in the sense of that class. 

More precisely, note that here $\rho_1=\rho_2=\lambda^{-1/2},$ so that for $\la< R^2,$ i.e.,  $R\tilde\rho_2> 1,$ 
$$ \frak a \rho_1^{-1}\rho_2^{-1} = \frac1{R\rho_1\rho_2^{2}} = \la^{3/2}R^{-1},$$
while
$$ \lambda\partial_{\eta_1}\partial_{\eta_2}\tilde\phi(0,0,1)= \lambda,$$
which is  bounded by $\la^{3/2}R^{-1}$ if and only if $\lambda\geq R^2$.  Thus, condition (SMD) fails here for $\la<R^2,$ when we regard $\phi$ as a phase in $\F^{\frac\gamma2,\frac\gamma2,\gamma}.$

Observe, on the other hand, that for  $\la\ge R^2,$ we have $\frak a=1,$ and thus condition (SMD)  is automatically satisfied. 

\smallskip

The failure of (SMD) can sometimes be remedied by embedding $\phi$ into a larger class: 

For example, $\phi$ is also contained in the larger class $\F^{1,\frac\gamma2,\gamma}$, and in the sense of this class, it does satisfy (SMD) for $\la\ge R^\ga:$

Indeed, then  $\rho_1=\lambda^{-1/2}R^{\frac\gamma 2-1}$, $\rho_2=\lambda^{-1/2}$, so that for $R^\ga\le \la<R^2,$ $$ \frak a \rho_1^{-1}\rho_2^{-1} = \frac1{R\rho_1\rho_2^{2}} = \la^{3/2}R^{-\frac \ga 2}\ge \la\gtrsim \sup_{\xi\in \Ga_R} |\partial_{n(\xi)}\partial_{t(\xi)} \la \phi(\xi)|.$$
\end{example}

\section{An application to maximal averages along  2-surfaces}\label{sec:maxop}
 Our results have an important application  to maximal averages along smooth hypersurfaces in $\Bbb R^3.$

Suppose  $S$ is a smooth hypersurfaces  in $\Bbb R^d$ with surface measure $d\sigma,$  and that $\rho\in C_0^\infty(S)$ is a smooth non-negative compactly supported density.  We then  consider the associated averaging operators $A_t, t>0,$
given by
$$
A_tf(x):=\int_{S} f(x-ty)  \,d\mu (y),\qquad f\in \S(\RR^d),
$$
where $d\mu:=\rho d\sigma.$    The associated maximal operator is given by
$$
\mathcal M_Sf(x):=\sup_{t>0}|A_tf(x)|, \qquad  x\in \Bbb R^d.
$$

We say that $\M_S$  is bounded on $L^p(\RR^d)$ for a given $p\in [1,\infty],$ if an a priori estimate of the form
$$
\|\M_S f\|_p\le C \|f\|_p, \qquad  f \in \S,
 $$
 holds true. 
 
\smallskip
The study of such maximal operators had been initiated in E.M. Stein's work on the spherical maximal function \cite{St}. In a series of papers \cite{IKM}, \cite{BDIM}, \cite{BIM}, for almost all real analytic hypersurfaces $S$ (and even larger classes of finite type surfaces) in $\Bbb R^3$ satisfying a natural transversality assumption, the range of Lebesgue spaces $L^p(\Bbb R^3)$ on which the maximal operator $\mathcal M_S$ is bounded has been determined  explicitly in terms of Newton diagrams associated to $S,$ in some cases up to the critical exponent $p_c=p_c(S).$ The latter is determined by the property that $M_S$ is $L^p$ bounded for $p>p_c,$ and unbounded for $p<p_c.$ In  \cite{BIM},  a parallel  ``geometric conjecture'' had been stated and proved for the same class of surfaces, which roughly claims that $p_c$ can be determined by testing $\mathcal M_S$ on characteristic functions of symmetric convex bodies.

This conjecture, and a related conjecture about a description of $p_c$ in terms of Newton diagrams, has remained open for a small  class of ``exceptional'' surfaces only. These  exhibit singularities of type A in the sense of Arno'ld's classification. A prototypical surfaces from this class is the
graph 
$$
S_{(n)}:=\{(x_1,x_2,1+\psi(x_1,x_2)): |(x_1,x_2)|\le \ve_0\}
$$
 of 
$$
\psi(x_1,x_2):=\frac 1{1-x_1}x_2^2+x_1^n \qquad  ( n\ge 4)
$$
over a sufficiently small $\ve_0>0$ -neighborhood of the origin. Our conjectures claim for  this surface that if $\rho(0,0,1)\ne 0,$ then
\begin{equation}\label{pc}
p_c=\tilde p_c:=\max \left\{\frac 32,2\frac {n+1}{n+3}\right\}
\end{equation}
(cf. \cite[Example 1.10]{BIM}). Indeed, by  \cite[Proposition 1.11]{BIM}, the maximal operator 
$\M_{S_{(n)}}$ is unbounded on $L^p(\RR^3)$ for $p<\tilde p_c,$ and we shall prove here that our estimates of FIO-cone multipliers allow  to prove

\begin{thm}\label{pbiggerpc}
For $\ve_0>0$ sufficiently small, the maximal operator $\M_{S_{(n)}}$ is bounded on $L^p(\RR^3)$ if $p>\tilde p_c.$ 
\end{thm}

 The proof requires essentially the full thrust of our FIO-cone multiplier estimates for $4/3\le p\le 2.$

We note that classical methods based on interpolation between $L^1$ and $L^2$ estimates only allow to cover the smaller range  
$p>2\frac {n+2}{n+4}$ (if $n$ is an even integer, the surface $S_{(n)}$ is even convex, so that the results from \cite{NSW} apply to it, but these again only yield the latter range).

We expect  that  extensions of our theory of FIO-cone multipliers will eventually allow for a full proof of the geometric conjecture.

\medskip

We shall prepare the proof of this theorem by providing  in the next subsection an abstract result which will allow to reduce  certain types of maximal estimates  to Fourier multiplier estimates. Our main tool  will be Seeger's main theorem in \cite{See}.

\subsection{Reduction of maximal estimates to multiplier estimates}\label{Mtomult}

Given any  smooth and bounded Fourier multiplier $m$ on $\RR^d,$ we shall denote by $\M_m$ the corresponding maximal operator, i.e., 
$$
\M_m f(x):=\sup\limits_{t>0} |T_{m(t\, \cdot)}f(x)|, \qquad f\in \S(\RR^d).
$$
Fix any smooth bump function $\chi_1$ supported in the annulus $1\le |\xi|\le 2$ so that $\chi_1(\xi)=1$ for $|\xi|=1.$
 By $\F^{-1}$ we shall denote  the inverse Fourier transform.
 
\begin{thm}\label{maxtomult}
Let $m$ be a smooth Fourier multiplier supported in the annulus $1\le |\xi|\le 2,$ and denote by  $\dot m$ its radial derivative, i.e., $\dot m(\xi):=\xi\nabla m(\xi).$ Let also $1<p<\infty,$ $\ve>0$ and $\la\gg1$ be given, and put $\tilde m:= \la^{-1} \dot m.$ 

Assume that the following estimates  hold true for m, and for $\tilde m$ in place of $m$ as well:
\smallskip

(i)  \hskip2cm $\|T_{m}\|_{L^p(\RR^d)\to L^p(\RR^d)}\le A$,
\medskip

(ii) \hskip2cm $\int_{|x|>r} |\F^{-1}m(x)| dx \le B(1+r)^{-\ve} \qquad \text{for every} \ r>0.$

Then 
\begin{equation}\label{maxmest}
\|\M_m\|_{L^p(\RR^d)\to L^p(\RR^d)}\le C_{p} \la^{\frac 1p} A [\log (2+B/A)]^{|\frac 1p-\frac 12|}.
\end{equation}
In particular, if $A\le 1/2$ and $B\ge 2,$ then, for any $\delta>0,$ 
$$
\|\M_m\|_{L^p(\RR^d)\to L^p(\RR^d)}\le C_{p,\delta} \la^{\frac 1p} ( A^{1-\de} +A\log B).
$$
\end{thm}

\begin{proof}  For $t>0,$ let us  write $S_t:= T_{m(t\cdot)}.$ Then
$
\M_m f=\sup\limits_{1\le t< 2} \sup\limits_{l\in\ZZ}|S_{t2^l} f|.
$

Denote by $\{r_l\}_{l\in \ZZ}$ the family of Rademacher functions on $[0,1],$ so that in particular $r_l(\om)=\pm 1$. By means of Khintchine's inequalities we  can then estimate
\begin{eqnarray*}
\M_m f(x)&\le& \sup\limits_{1\le t< 2}\big(\sum_{l\in\ZZ}|S_{t2^l} f(x)|^2\big)^{\frac 12}\\
&\sim& \sup\limits_{1\le t< 2} \left [\int_0^1\big|\sum_{l\in\ZZ} r_l(\om) S_{t2^l} f(x)\big|^p d\om \right]^{\frac 1p}.
\end{eqnarray*}
Putting
$
F_\om(t,x):=\sum_{l\in\ZZ} r_k(\om) S_{t2^l} f(x),
$
we may thus estimate
$$
(\M_m f(x))^p\lesssim \int_0^1 \sup\limits_{1\le t< 2}|F_\om(t,x)|^p d\om.
$$
Next, by \cite[(4.11)]{IKM}), if $\rho\in C^{\infty}_0(\RR)$ is a bump function  supported in
$[1/2,4]$ such that $\rho(t)=1$ when $1\le t\le 2,$ then
\begin{eqnarray*}
&&\sup_{1\le t<2}|F_\om(t,x)|^p\\
&\le& p\left(\int^{\infty}_{-\infty}\Big|\rho(t)F_\om(t,x)\Big|^p
dt\right)^{1/p'}\left(\int^{\infty}_{-\infty}\Big|\frac{\partial}{\partial
t}\Big(\rho(t)F_\om(t,x)\Big)\Big|^p dt\, \right)^{1/p}.
\end{eqnarray*}
By H\"older's inequality, this easily  implies that 
\begin{eqnarray}\label{Mlacest}
\|\M_m f\|^p_p&\lesssim&
\left(\int_{\RR^d}\int_0^1\int^{4}_{1/2}|F_\om(t,x)|^p dt d\om dx\right)^{\frac{1}{p'}}
\left(\int_{\RR^d}\int_0^1\int^{4}_{1/2}|\partial_tF_\om(t,x)|^p dt d\om dx\right)^{\frac{1}{p}}\nonumber\\
 &+& \int_{\RR^d}\int_0^1\int^{4}_{1/2}|F_\om(t,x)|^p dt d\om dx.
\end{eqnarray}
For $t\in[1/4, 4]$ and $\om\in[0,1]$ fixed, let us estimate $\int_{\RR^d}|F_\om(t,x)|^p dx.$
To this end, note that 
\begin{equation}\label{suppmlat}
\supp m(t2^l\, \cdot)\subset \{\xi\in\RR^d: 2^{-l-2} \le |\xi|\le 2^{-l+3}\},\qquad \text{if} \  1/2\le t\le 4,
\end{equation}
and that 
$$
F_\om(t,x)=(T_{m_{t,\om}}f) (x),
$$
where $m_{t,\om}$ denotes the Fourier multiplier
$$
m_{t,\om}(\xi):=\sum_{l\in\ZZ} r_l(\om) m(t2^l\xi).
$$
We shall estimate the operator $T_{m_{t,\om}}$ by means of   \cite[Theorem 1]{See}. Given $\tau>0,$ we choose $k\in \ZZ$ so that $\tau=\tau_0 2^k$ and $1\le \tau_0<2.$ Then $\chi_1(\xi) m(\tau t2^l\xi)\ne 0$ for some $\xi$ only if 
$2^{-4}\le 2^{k+l}\le 2,$ i.e., 
$$
l=-k+j, \qquad \text{with } j\in I:=\{-4, -3, \dots , 2\}.
$$
In particular, then $t\tau 2^l= t\tau_0 2^j,$ with $j\in I,$ so that
$$
(\chi_1m_{t,\om})(\tau\xi) =\sum\limits_{j=-4}^2r_{j-k}(\om) \chi_1(\xi)m(t\tau_0 2^j\xi).
$$
Thus (i) implies that
\begin{equation}\label{isee}
\|T_{\chi_1 m_{t,\om}(\tau \cdot) }\|_{L^p(\RR^d)\to L^p(\RR^d)}\le  CA,
\end{equation}
and since $s:=t\tau_0 2^j \sim 1,$ (ii) implies 
\begin{equation}\label{iisee}
\sup\limits_{\tau>0}\int_{|x|>r} |\F^{-1}[\chi_1 m_{t,\om}(\tau\,\cdot)](x)| dx \le C B(1+r)^{-\ve} \qquad \text{for every} \ r>0,
\end{equation}

Indeed, we clearly have
$$
\|T_{\chi_1m(s\cdot)}\|_{L^p\to L^p}\le C \|T_{m(s\cdot)}\|_{L^p\to L^p}=C\|T_m\|_{L^p\to L^p},
$$
which implies \eqref{isee}. Moreover, since $m$ satisfies (ii),  it is easily seen that for $s\sim 1,$ 
$$
\int_{|x|>r} |\F^{-1}[m(s\,\cdot)](x)| dx \le C B(1+r)^{-\ve} \qquad \text{for every} \ r>0,
$$
with a constant $C$ depending only on $\ve $ and the dimension $d.$ And, if we set $F:=\F^{-1}[m(s\,\cdot)]$ and $\vp:=\F^{-1} \chi_1\in \S,$ then $\|F*\vp\|_1\lesssim  \|F\|_1\le C B.$ Next, assume that $r\gg 1,$  and split 
$$
|F*\vp(x)|\le \int_{|y|> r/2} |F(x-y)| |\vp(y)| dy+\int_{|y|\le r/2} |F(x-y)| |\vp(y)| dy=F_1(x)+F_2(x).
$$
Since $|\vp(y)|\le C_Nr^{-N}(1+|y|)^{-N} $ for any $N\in\NN,$ if $|y|>r/2,$ we have 
$$
\|F_1\|_1\le C_Nr^{-N} \|F\|_1\le CB(1+r)^{-\ve}.
$$
Moreover, if $|x|\ge r,$ then $|x-y|\ge r/2$ in the integral defining $F_2,$ so we may estimate
\begin{eqnarray*}
\int_{|x|>r} F_2(x) dx&=& \int_{|y|\le r/2}\left(\int_{|x|>r}  |F(x-y)| dx\right) |\vp(y)|dy\\
&\le&  \int_{|y|\le r/2}\left(\int_{|z|>r/2}  |F(z)| dz\right) |\vp(y)|dy\\
&\lesssim& B(1+r/2)^{-\ve}.
\end{eqnarray*}
These estimates imply \eqref{iisee}.

The estimates  \eqref{isee} and  \eqref{iisee} show that $m_{t,\om}$ satisfies the assumptions (i) and (ii) of   \cite[Theorem 1]{See} uniformly in $t\in[1/4, 4]$ and $\om\in[0,1],$ and thus we can conclude that 

\begin{equation}\label{Tmest}
\|T_{m_{t,\om}}\|_{L^p(\RR^d)\to L^p(\RR^d)}\le C_{p}  A [\log (2+B/A)]^{|\frac 1p-\frac 12|}\quad \text{ for all } t\in[1/4, 4],\om\in[0,1].
\end{equation}

In particular, we find that for all $t\in [1/4, 4],\om\in[0,1]$, 
\begin{equation}\label{Fest}
\int_{\RR^d}|F_\om(t,x)|^p dx\le \left(C_{p}  A [\log (2+B/A)]^{|\frac 1p-\frac 12|}\right)^p \|f\|_p^p.
\end{equation}
\medskip

As for $\partial_tF_\om,$  note that 
$$
\partial_t[m(t2^l\xi)]=2^l\xi\nabla m(t2^l\xi)=t^{-1} \dot m (t2^l\xi)=\frac \la {t} \tilde m(t2^l\xi),
$$
so that 
$$
\partial_tF_\om(t,x)=\frac \la {t} (T_{\tilde m_{t,\om}}f) (x),
$$
if we put 
$$
\tilde m_{t,\om}(\xi):=\sum_{l\in\ZZ} r_l(\om) \tilde m(t2^l\xi).
$$
Since $\tilde m$ is assumed to satisfy estimates (i) and  (ii)  as well, arguing in the same way as for $m$ we see that also \begin{equation}\label{Ttildemest}
\|T_{\tilde m_{t,\om}}\|_{L^p(\RR^d)\to L^p(\RR^d)}\le C_{p}  A [\log (2+B/A]^{|\frac 1p-\frac 12|}\quad \text{ for all } t\in[1/4, 4],\om\in[0,1].
\end{equation}
This implies that 
\begin{equation}\label{dtFest}
\int_{\RR^d}|\partial_t F_\om(t,x)|^p dx\le \left(C_{p} \la \,A [\log (2+B/A]^{|\frac 1p-\frac 12|}\right)^p \|f\|_p^p.
\end{equation}
The estimate \eqref{maxmest} is now an immediate consequence of \eqref{Mlacest} and the estimates \eqref{Fest} and \eqref{dtFest}.
\end{proof} 

\begin{remark}\label{plesstwo}
Seeger's theorem is actually needed only when $p<2.$ If $p\ge2,$  standard arguments based on Littlewood-Paley theory and the fact that the $\ell^p(\ZZ)$- norm is bounded by the $\ell^2(\ZZ)$-norm  allow to control  the maximal operator
$\M_m$ by the corresponding local maximal operator $\M_{m,\rm loc} f:=\sup\limits_{1\le t< 2} |T_{m(t\, \cdot)}f |,$ 
which in return can be controlled again by the variant of Sobolev's embedding used  in our arguments above.This leads then  even to the stronger estimate
$$
\|\M_m\|_{L^p(\RR^d)\to L^p(\RR^d)}\le C_p\la^{\frac 1p} A.
$$

\end{remark}

\subsection{Proof of Theorem \ref{pbiggerpc}}\label{proofmax}

Following the  discussion in Subsection 5.2 of \cite{BIM}, we may assume that 
$\M=\M_{S_{(n)}}$ is associated to  the  averaging
operators of  convolution with dilates  of a measure $\mu,$
whose Fourier transform at $\xi=(\xi_1,\xi_2,\xi_3)$  is given  by
$$
\hat\mu(\xi) :=\iint e^{-i(\xi_1x_1+\xi_2x_2+\xi_3(1+\psi\x))}
\eta(x_1,x_2)\,dx_1 dx_2,$$
where $\eta$ is a smooth bump function supported in a sufficiently small $\ve_0$-neighborhood of the origin.

Moreover, by means of suitable dyadic frequency decompositions with dyadic parameter $\la\gg 1,$ we may  reduce ourselves to estimating the maximal operators $\M^\la$ associated to the frequency localized measures
\begin{equation}\nonumber
    \widehat{\mu^\la}(\xi) :=\chi_0\left(\frac{\xi_1}{\lambda},\,\frac{\xi_2}{\lambda} \right) \chi_1\left(\frac{\xi_3}{\lambda} \right)\hat\mu(\xi),
\end{equation}
where the smooth cut-off  function $\chi_1$ vanishes near the origin and is identically one  near $1,$ whereas the smooth cut-off $\chi_0$ is supported in a sufficiently small neighborhood of the origin and is identically $1$ near the origin (the contributions by the remaining  terms to the maximal operator $\M$ are of order $O(\la^{-N})$ for every $N\in\NN$  
as $\la\to+\infty, $ provided the support of $\eta$ is sufficiently small).

 We shall  see that the estimates that we shall obtain for the corresponding maximal operators $\M^\la$ will sum over all such dyadic numbers $\la\gg 1$ if $p>\tilde p_c.$
\smallskip

We next write
\begin{equation*}
\xi_3=\la s_3,\quad \xi_1=\la s_3s_1,\quad \xi_2=\la s_3s_2,
\end{equation*}
and put $s':=(s_1,s_2), s:=(s',s_3).$ Then, choosing the supports of $\chi_0$ and $\chi_1$ sufficiently small, we have
$$
|s_3|\sim 1\quad \mbox{and}\quad  |s'|\ll1
$$
on the support of $\widehat{\mu^\la}$.  

Note also that in the re-scaled coordinates $\xi'$ given by $\xi:=\la\xi',$ we have 
$$
s_3=\xi'_3, \quad s_1=\frac {\xi'_1}{\xi'_3}, \quad s_2=\frac {\xi'_2}{\xi'_3}, \qquad\text{where} \quad |\xi'_3|\sim 1,\  |\xi'_2|\ll 1,
\ |\xi'_1|\ll 1.
$$
Recalling  that the operator norm of a Fourier multiplier operator and of the associated maximal operator is  invariant under any non-degenerate linear change of coordinates in the multiplier, we shall work in the coordinates $\xi'$ henceforth.

We may then write
$$
\widehat{\mu^\la}(\xi)
= e^{-i\la s_3} \chi_0(s_3s')\chi_1(s_3) \int_{\bR^2} e^{-i\la s_3\left(s_1x_1+s_2x_2+1+\frac 1{1-x_1}x_2^2+x_1^n\right)} \eta(x)\,dx.
$$

In a next step, following Section 6.1 in \cite{BIM}, we apply the method of stationary phase to the integration in $x_2, $ which yields 
\begin{eqnarray*}
\widehat{\mu^\la}(\xi)
&= &e^{-i\la s_3} \chi_0(s_3s')\chi_1(s_3) \big[ \la^{-1/2} \int_{\bR} e^{-i\la s_3\left(x_1^n+s_1x_1+1 -\frac {(1-x_1)s_2^2}4\right)}\tilde \eta(x_1,  s_2)\,dx_1\\
&&\hskip 8cm+r(\la,s) \big],
\end{eqnarray*}
with a slightly modified cut-off function $\chi_1,$  
where $\tilde \eta$ is another smooth bump function supported in a
sufficiently small neighborhood of  the origin, and  $r(\la,s)$ is a
remainder term of order $r(\la,s)=O(\la^{-\frac32})$ as $\la\to+\infty.$ As shown in \cite{BIM}, the  total contribution to the maximal operator $\M$ by the terms  $r(\la,s)$ with sufficiently  large dyadic $\la$s is  $L^p$-bounded for $p>3/2.$ 

Without loss of generality, let us therefore in the sequel assume that 
\begin{equation}\label{mula2}
\widehat{\mu^\la}(\xi)
= \la^{-1/2}   \chi_0(s_3s')\chi_1(s_3) \, e^{-i\la s_3(1-\frac {s_2^2}4)}  J(\la,s),
\end{equation}
where we have put
$$
J(\la,s):=\int_{\bR} e^{-i\la s_3\left(x_1^n+x_1(s_1+s_2^2/4)\right)}\tilde \eta(x_1,  s_2)\,dx_1.
$$

The ``Airy type integral'' $J(\la,s)$ can be evaluated by means of \cite[Lemma 2.2]{IM}, which leads to the following 

\begin{lemma}\label{airy}
Let us put
$$
u:=s_1 +s_2^2/4.
$$ 
If $\ve_0>0$ is chosen sufficiently small, then the following hold true:

\begin{itemize}
\item[(a)]  If $\la^{(n-1)/n} |u|\lesssim 1,$ then 
$$
J(\la,s)=\la^{-1/n} g(\la^{\frac {n-1}n} u, s_2, s_3),
$$
where $g(v,s_2,s_3)$ is a smooth function whose derivatives of any fixed order are uniformly bounded for 
$|v|\lesssim 1, |s_2|\lesssim 1$ and $s_3\sim 1.$

\item[(b)]  If $\la^{(n-1)/n} |u|\gg 1,$ let us assume first   that  $n$ is odd  and $u<0.$ Then 
\begin{eqnarray}\nonumber
J(\la,s)&=&\la^{-\frac 12} |u|^{-\frac {n-2}{2n-2}}\, \chi_0\Big(\frac u{\ve_0}\Big)a_+(\la s_3 |u|^{\frac n{n-1}},|u|^{\frac 1{n-1}},s_2)\, e^{i\la s_3 c_n|u|^{\frac n{n-1}}}\\
&+&\la^{-\frac 12} |u|^{-\frac {n-2}{2n-2}}\, \chi_0\Big(\frac u{\ve_0}\Big) a_{-}(\la s_3|u|^{\frac n{n-1}},|u|^{\frac1{n-1}},s_2) \,e^{-i\la s_3 c_n|u|^{\frac n{n-1}}} 
\label{airyest}\\
 &+&(\la s_3|u|)^{-1} E(\la s_3 |u|^{\frac n{n-1}}, |u|^{\frac1{n-1}},s_2),\nonumber
\end{eqnarray}
with  $c_n:= (n-1)\cdot n^{-n/(n-1)}.$  
Here, $a_\pm$ are smooth functions which are symbols of order $0$ with respect to the first argument $ \mu=\la s_3|u|^{\frac n{n-1}},$ i.e., 
\begin{equation}\label{aderest}
|\pa_\mu^\al \pa_v^\be\pa_{s_2}^\ga a_\pm(\mu,v,s_2)|\le C_{\al,\be,\ga}  |\mu|^{-\al}, \qquad \forall \al,\be,\ga\in\NN,
\end{equation}
and $E$ is smooth and satisfies estimates
\begin{equation}\label{Ederest}
|\pa_\mu^\al \pa_v^\be\pa_{s_2}^\ga E(\mu,v,s_2)|\le C_{N,\al,\be,\ga} |v|^{-\be} |\mu|^{-N}, \qquad \forall N,\al,\be,\ga\in\NN.
\end{equation}

If  $n$ is odd and $u>0,$ then  the same formula remains valid, even with $a_+\equiv 0, a_-\equiv 0.$ 

Finally,  if $n$ is even, we do have a similar result, but without the presence of the term containing $a_-$.

 \end{itemize}
\end{lemma}

We seize  here the opportunity  to correct   an error in  the  formulation of \cite[Lemma 2.2 (b)]{IM}:   the  amplitudes 
  $a_\pm (|u|^{\frac1{B-1}},s)$ therein should rather be  functions $a_\pm (\la |u|^{\frac B{B-1}}, |u|^{\frac 1{B-1}},s) $ of also $\mu:=\la |u|^{\frac B{B-1}},$ which  are symbols of order $0$ with respect to this variable $\mu.$
\medskip

\noindent {\it Sketch of proof.} 
Since the arguments from the proof of this lemma are particularly  simple for $J(\la,s),$ we shall briefly sketch the proof for the convenience of the reader, and refer to \cite{IM} for further details .

\smallskip
{\bf Case (a)} If $\la^{\frac {n-1}n} |u|\lesssim 1,$ we scale $x_1$ by the factor $\la^{-1/n},$ which leads to 
$$
J(\la,s)=\la^{-1/n} g(\la^{\frac {n-1}n} u, s_2, s_3),
$$
where 
$$
g(v, s_2,s_3):=\int_{\bR} e^{-i s_3\left(t^n+tv\right)}\tilde \eta(\la^{-1/n}t,  s_2)\,dt.
$$
Integrations by parts show that $g$ is a smooth function whose derivatives of any order are uniformly bounded for 
$|v|\lesssim 1, |s_2|\lesssim 1$ and $s_3\sim 1.$ 

\medskip
{\bf Case (b)} If $\la^{\frac {n-1}n} |u|\gg 1,$ assume first that $|u|\ge \ve_0.$ Then, since $|x_1|<\ve_0,$ we see that we can integrate by parts in the integral defining $ J(\la,s)$ and find that, for any $N\in\NN,$ 
$$
J(\la,s)=O((\la|u|)^{-N})=O(\la^{-N}),
$$
and similar estimates hold true also for any derivatives of  $J(\la,s)$ w.r. to $s$-variables. These contributions can thus be ignored in the sequel (respectively  be included into the last term of \eqref{airyest}).

\smallskip

Let us therefore henceforth assume that $|u|<\ve_0.$ We scale by the factor 
$$
v:=|u|^{1/(n-1)}=|s_1 +s_2^2/4|^{1/(n-1)},
$$
 which yields
$$
J(\la,s)=v \int_{\bR} e^{-i\la v^n s_3\left(t^n+(\sgn u)t)\right)}\tilde \eta(vt,  s_2)\,dt.
$$
Note that $\la v^n\gg 1,$ since we are in Case (b).

We decompose this into 
$$
J(\la,s)=J_0(\la,s)+J_\infty(\la,s),
$$
where
\begin{eqnarray*}
J_0(\la,s)&=&v \int_{\bR} e^{-i\la v^n s_3\left(t^n+(\sgn u)t)\right)}\chi_0(\frac tM)\tilde \eta(vt,  s_2)\,dt,\\
J_\infty(\la,s)&=&v \int_{\bR} e^{-i\la v^n s_3\left(t^n+(\sgn u)t)\right)}\left(1-\chi_0(\frac tM)\right )\tilde \eta(vt,  s_2)\,dt,
\end{eqnarray*}
with $M\gg 1$ sufficiently large. The estimates in  \eqref{airyest} are now easily obtained by applying the method of stationary phase to the integral in $ J_0(\la,s),$ and integrations by parts to $J_\infty(\la,s).$
\qed

Since  we are assuming that $|u|\ll 1,$ in Case (b) we have
$$
1\ll \la^{\frac {n-1}n}|u| \ll \la^{\frac {n-1}n}.
$$
Choosing $j(\la)\in\NN$ suitably so that $2^{j(\la)}\sim \la^{\frac {n-1}n},$ we may thus dyadically  decompose 
\begin{eqnarray}\label{Jdecomp}
J(\la,s)=J_0(\la,s)+\sum\limits_{j=1}^{j(\la)} J_j^+(\la,s)+\sum\limits_{j=1}^{j(\la)} J_j^-(\la,s)+\sum\limits_{j=1}^{j(\la)} J_j^E(\la,s),
\end{eqnarray}
where 
\begin{eqnarray*}
J_0(\la,s)&:=&\la^{-1/n} \chi_0(\la^{\frac {n-1}n}|u|) g(\la^{\frac {n-1}n} u, s_2, s_3),\\
J_j^\pm(\la,s)&:=&\chi_1\big(\frac{\la^{\frac {n-1}n}|u|}{2^j}\big) \la^{-\frac 12} |u|^{-\frac {n-2}{2n-2}}\, \chi_0\Big(\frac u{\ve_0}\Big) a_{\pm}(\la s_3|u|^{\frac n{n-1}},|u|^{\frac1{n-1}},s_2) \,e^{\pm i\la s_3 c_n|u|^{\frac n{n-1}}},\\
 J_j^E(\la,s)&:=&\chi_1\big(\frac{\la^{\frac {n-1}n}|u|}{2^j}\big) (\la s_3|u|)^{-1} E(\la s_3 |u|^{\frac n{n-1}}, |u|^{\frac1{n-1}},s_2),
\end{eqnarray*}
with suitably chosen smooth cut-off function $\chi_0$ and $\chi_1.$ We shall separately estimate the contribution of each of these terms to the maximal operator $\M_{S_{(n)}}.$

\medskip
To this end, for fixed $j$ let us introduce the following abbreviations: We put
$$
\ga:=\frac n{n-1}, \quad  R=R_j(\la):=\frac{\la^{\frac {n-1}n}}{2^j}.
$$
Since we assume that $n\ge 4,$  we may here assume that 
\begin{equation}\label{abbrevs}
1<\ga <2, \quad R\gg 1 \quad  \text{and} \quad \la\ge R^\ga.
\end{equation}
We can now  re-write 
\begin{eqnarray*}
J_0(\la,s)&=&\la^{\frac 1\ga-1} \chi_0(\la^{1/\ga}|u|) g(\la^{1/\ga} u, s_2, s_3),\\
J_j^\pm(\la,s)&=&J^\la_{\pm,R}:=\la^{-\frac 12} |u|^{\frac \ga 2 -1}  \chi_1(R|u|) \,a_{\pm}(\la s_3|u|^{\ga},
|u|^{\ga-1},s_2) \,e^{\pm i\la s_3 c_n|u|^\ga},\\
 J_j^E(\la,s)&=&J^\la_{E,R}:=\chi_1(R|u|) (\la s_3|u|)^{-1} E(\la s_3 |u|^{\ga}, |u|^{\ga-1},s_2),
\end{eqnarray*}
where 
$$
u=s_1 +s_2^2/4=\frac{\xi'_1\xi'_3+\frac 14 (\xi'_2)^2}{(\xi'_3)^2}.
$$
In view of \eqref{mula2}, let us introduce the Fourier multiplier operators $T_{m^\la_{\pm,R}},$ where 
$$
m^\la_{\pm,R}(\xi'):=  \la^{-1/2}   \chi_0(s_3s')\chi_1(s_3) \, e^{-i\la s_3(1-\frac {s_2^2}4)}  J^\la_{\pm,R}(\la,s),
$$
as well as the operator $T_{m^\la_0},$ with 
$$
m^\la_0(\xi'):=\la^{-1/2}   \chi_0(s_3s')\chi_1(s_3) \, e^{-i\la s_3(1-\frac {s_2^2}4)} J_0(\la,s),
$$
and $T_{m^\la_{E,R}},$ with 
$$
m^\la_{E,R}(\xi'):=  \la^{-1/2}   \chi_0(s_3s')\chi_1(s_3) \, e^{-i\la s_3(1-\frac {s_2^2}4)}  J^\la_{E,R}(\la,s).
$$

In a next step, we estimate these operators by means of Corollary  \ref{mainlp}. We begin with the operators $T_{m^\la_{\pm,R}}.$

Since $|u|\sim R^{-1} $ in  $J^\la_{\pm,R}$ and $\ga\ge 1,$ the estimates \eqref{aderest}  easily imply that the amplitude 
$$
 a^\la_{\pm R}(u,s):=R^{\frac \ga 2-1} |u|^{\frac \ga 2 -1}\,a_{\pm}(\la s_3|u|^{\ga},|u|^{\ga-1},s_2)
$$
satisfies  estimates of the form
\begin{equation}\label{ampest}
|\partial_u^k \pa_s^ \beta a^\la_{\pm, R}(u,s)|\le C_{k,\beta} R^k , \qquad k\in \NN, \beta\in \NN^3,
\end{equation}
which match with those of $\partial_u^k [\chi_1(R|u|)].$ 
By Lemma \ref{admissaex}, this shows in particular that for any fixed $s,$ the amplitude $a^\la_{\pm, R}(u,s)$ is admissible when viewed as a function of $\xi'$, uniformly in $\la$ and $s$.

\smallskip

We shall therefore write $m^\la_{\pm,R}(\xi')$ briefly as 
\begin{equation}\label{mlapm}
m^\la_{\pm,R}(\xi')=  \la^{-1} R^{1-\frac \ga 2} \chi_1(R|u|)  a^\la_{\pm R}(u,s)\chi_0(s_3s')\chi_1(s_3) \, e^{-i\la s_3(1-\frac {s_2^2}4\mp c_n|u|^\ga)}.
\end{equation}

Let us here perform the linear change of coordinates
$$
\xi'_1= -\eta_1, \ \xi'_2=\sqrt{2} \eta_2,\  \xi'_3=\eta_3.
$$
In these coordinates,
$$
u=u(\eta):=\frac{\frac {\eta_2^2}2-\eta_1\eta_3}{\eta_3^2},
$$ so we may write
\begin{equation}\label{mlapmeta}
m^\la_{\pm,R}=\la^{-1} R^{1-\frac \ga 2} \chi_1(R|u(\eta)|)  a^\la_{\pm R}(u(\eta),s)\chi_0(-\eta_1,\sqrt{2}\eta_2)\chi_1(\eta_3) \, e^{-i\la \phi(\eta)},
\end{equation}
with  phase 
$$
\phi(\eta):=\eta_3-\frac {\eta_2^2}{2\eta_3} \mp c_n\eta_3 \Big|\frac{\frac {\eta_2^2}2-\eta_1\eta_3}{\eta_3^2}\Big|^\ga,
$$
and we are in a frequency region where $|\eta_1|+|\eta_2|\ll 1$ and, say,  $\eta_3\sim 1.$ Thus, in this frequency region, 
$\phi$ assumes the form 
$$
\phi=\phi_0\mp c_n\phi^\ga,
$$
with $\phi_0\in \F_{class}$ and $\phi^\ga$ as in Example \ref{mainex}.
Thus, by Example  \ref{phi+phi}, we see that $\phi\in \F^{1,\frac \ga 2,\ga},$ with $1<\ga<2.$ Moreover, as we have seen in Subsection \ref{smdex}, $\phi$ satisfies also condition (SMD). 

In particular, since $\la\ge R^\ga,$  we  see by \eqref{rho} that 
$
\tilde \rho_1=(\la R^{2-\ga})^{-1/2}\le R^{-1}, 
$
so that 
$$
\rho_1=(\la R^{2-\ga})^{-1/2}, \quad \rho_2=\la^{-1/2},
$$
hence 
$$
N(\la,R)=\frac {R^{-3/2}}{\rho_1\rho_2}=\la R^{-\frac{1+\ga}2}.
$$

Consequently, we may apply Corollary \ref{mainlp} and obtain that for $4/3\le p\le 2,$ 
\begin{eqnarray*}
 \|T_{m^\la_{\pm,R}} \|_{L^p\to L^p} &\le& \la^{-1} R^{1-\frac \ga 2} C_\epsilon\,R^\epsilon N(\lambda,R)^{\frac 2p-1} 
 =C_\epsilon  \la^{-1} R^{1-\frac \ga 2}\big(\la R^{-\frac{1+\ga}2}\big)^{\frac 2p-1},
\end{eqnarray*}
i.e.,
\begin{equation}\label{Test1}
\|T_{m^\la_{\pm,R}} \|_{L^p\to L^p} \le C_\epsilon\,R^\epsilon \la^{ \frac 2p-2} R^{-\frac{1+\ga}{p}+\frac 32}.
\end{equation}
We next prove
\begin{equation}\label{kernelest1}
\int_{|x|>r} |\F^{-1}m^\la_{\pm,R}(x)| dx \le C \la^5(1+r)^{-1} \qquad \text{for every} \ r>0.
\end{equation}
We can argue here in a similar way as in the proof of Corollary \ref{mainlp}. We decompose $m^\la_{\pm,R}=\sum_\vth m_\vth$ into the contributions $m_\vth$ of all L-boxes $\vth$ contained in $\Gamma^\pm_R,$ and write again 
$\mu_\vth:=\F^{-1}m_\vth,$ so that
$$
F:=\F^{-1}m^\la_{\pm,R}=\sum_{\vth\subset \Gamma^\pm_R} \mu_\vth.
$$
Since $1/\rho_1= \la^{1/2}R^{1-\ga/2}\le \la^{1/2}\la^{(1-\ga/2)/\ga}=\la^{1/\ga}$ and $1/\rho_2=\la^{1/2},$ we see that 
$$
\sharp \{\vth\subset \Gamma^\pm_R\}\lesssim \la^{1/2+1/\ga}.
$$
Thus, since by Lemma \ref{4iersupp} $\|\mu_\vth\|_1\le C$ for all $\vth$s, we see that $\|F\|_1\le C \la^{1/2+1/\ga}\le C\la^2.$ 

Moreover, since 
$$
|\la\nabla \phi(\xi)|\lesssim \la R^{\ga+1}\le \la^{\frac{2\ga+1}\ga}\le \la^3,
$$
we see that $\mu_\vth$ is essentially supported in a ball of radius $\la^3$ for every  $\vth\subset \Gamma^\pm_R.$
Indeed, by arguing as in the proof of Corollary \ref{4iersupp2}, we find that for $r\gg\la^3,$ 
$$
\int_{|x|>r} |\mu_\vth(x)| dx \le C_N\la^{-N} r^{-N}
$$
for any $N\in\NN.$ This implies in particular that 
$
\int_{|x|>r} |F(x)| dx \le C_N  r^{-N}.
$

In combination, our estimates imply that, for any $r>0,$
$$
\int_{|x|>r} |F(x)| dx \le C  \la^5(\la^3+r)^{-1}\le C\la^5(1+r)^{-1},
$$
which proves \eqref{kernelest1}.

\bigskip

Let us next consider the multiplier $\dot m(\eta)=\partial_t[m(t\eta)]\vert_{t=1}$ (as defined  in Theorem \ref{maxtomult}) for the multiplier  $m:=m^\la_{\pm,R}.$

Note first  that  $s'$ and $u(\eta)$ are homogeneous of degree $0,$ and $\phi$ and $s_3$ are homogeneous of degree $1$ in $\eta.$ If we regard  $a^\la_{\pm R}(u,s)$ as a function of $\eta,$ we thus see that 
 $$
{\dot a}^\la_{\pm R}(\eta)=\partial_t[a^\la_{\pm R}(t\eta)]\vert_{t=1}=R^{1-\frac \ga 2} |u|^{\frac \ga 2 -1}\,\tilde a_{\pm}(\la s_3|u|^{\ga},|u|^{\ga-1},s_2),
$$
if we put ${\dot a}_{\pm}(\mu,v,s_2):=\mu\pa_\mu a_\pm(\mu,v,s_2).$ By \eqref{aderest}, the function ${\dot a}_{\pm}$ s controlled by the same kind of estimates as  $a_{\pm}.$

Observe next that $\eta\cdot\nabla m(\eta) [\chi_0(-\eta_1,\sqrt{2}\eta_2)\chi_1(\eta_3)]$ is  essentially  just a slightly modified version of the smooth cut-off $\chi(\eta):=\chi_0(-\eta_1,\sqrt{2}\eta_2)\chi_1(\eta_3).$ 
 
 Finally, one easily sees  that 
 $$
 \la^{-1} \partial_t[e^{-i\la\phi(t\eta)}]\vert_{t=1}=-i\phi(\eta) e^{-i\la\phi(\eta)}=-i\big(\eta_3-\eta_1+\eta_3u(\eta)\mp c_n\eta_3|u(\eta)|^\ga \big)e^{-i\la\phi(\eta)}.
 $$
The product of  the smooth term $\eta_3-\eta_1+\eta_3u(\eta)$ with the amplitude in \eqref{mlapmeta} can easily be absorbed into  the cut-off $\chi$ and thus again just leads to a modified version of $\chi.$ 
 As for the product of the amplitude with $\eta_3|u(\eta)|^\ga,$ note that 
 $$
|u(\eta)|^\ga\chi_1(R|u(\eta)|) =R^{-\ga} \chi_{1,\ga}(R|u(\eta)|),
 $$
 if we set $\chi_{1,\ga}(v):=v^\ga \chi_1(v),$ which  is  again an admissible  amplitude.
 
These observations show that $\dot m$ is a sum of four terms, where the first two of them are again of the form  \eqref{mlapmeta} (but with  slightly modified versions of the amplitude),  the third being of a similar form but coming with an extra factor $\la R^{-\ga},$ and the fourth with an extra factor $\la$ (the last two arise when $\eta\cdot\nabla$ hits  $e^{-i\la\phi}$).
\smallskip

We thus see that $\tilde m:=\la^{-1} \dot m$ is again essentially of the form  \eqref{mlapmeta}, and thus the estimates \eqref{Test1} and \eqref{kernelest1} hold for $\widetilde {m^\la_{\pm,R}}$ in place of $m^\la_{\pm,R}$ as well.

 The multiplier $m^\la_{\pm,R}$ satisfies thus the assumptions (i) and (ii) of Theorem \ref{maxtomult}, with constants 
\begin{equation}\label{AB}
A:= C_\epsilon\,R^\epsilon \la^{ \frac 2p-2} R^{-\frac{1+\ga}{p}+\frac 32},\quad B:=C\la^5.
\end{equation}
 Note that for $4/3\le p\le 2$ and $\la\ge R^\ga,$ we have 
$$
A\lesssim  R^{(\frac 2p-2)\ga-\frac{1+\ga}{p}+\frac 32+\epsilon}\le R^{(\ga-1)\frac 34 +\frac 32 -2\ga+\epsilon}
=R^{\frac{3-5\ga}4+\epsilon}\ll 1,
$$
if we choose $\epsilon$ sufficiently small.

Since Theorem \ref{maxtomult} holds as well for multipliers supported in the slightly larger annulus $1/2\le |\xi| \le 2,$ we find that, for $4/3\le p\le 2,$
$$
\|\M_{m^\la_{\pm,R}} \|_{L^p\to L^p}\le C_{\epsilon,\delta,p}\,\left\{ \la^{\frac 1p}\left( \la^{ \frac 2p-2} R^{-\frac{1+\ga}{p}+\frac 32+\epsilon}\right)^{1-\delta}+\la^{\frac 1p}\log \la \,\la^{ \frac 2p-2} R^{-\frac{1+\ga}{p}+\frac 32+\epsilon} \right\}.
$$
Hence, by choosing $\epsilon$ and $\delta$ sufficiently small, we see that if $4/3\le p\le 2$, then
\begin{equation}\label{Mest1}
\|\M_{m^\la_{\pm,R}} \|_{L^p\to L^p} \le C_{\epsilon,p}\, \la^{ \frac 3p-2+\delta} R^{-\frac{1+\ga}{p}+\frac 32+\delta}
\end{equation}
 for every $\delta>0.$

\smallskip

Assume now that  $p>\tilde p_c,$ so that in particular  $p>3/2.$
Since $\la\ge R^\ga,$  by choosing $\delta>0$ sufficiently small we can estimate
$$
\la^{ \frac 3p-2+\delta} R^{-\frac{1+\ga}{p}+\frac 32+\delta}\le \la^{-\delta}R^{(\frac 3p -2+2\delta)\ga}R^{-\frac{1+\ga}{p}+\frac 32+\delta}\le 
\la^{-\delta}R^{\frac 1p( 2\ga-1)-2\ga+\frac 32+\delta(2\ga+1)}.
$$
Note that  $\frac 1p( 2\ga-1)-2\ga+\frac 32<0$ if and only if 
$$
p>\frac {2\ga-1}{2\ga-3/2}=2\frac {n+1}{n+3}.
$$
Thus, by choosing $\delta$ sufficiently small, we see that for $p>\tilde p_c$ we can establish an estimate 
$$
\|\M_{m^\la_{\pm,R}} \|_{L^p\to L^p} \le C_p\la^{-\delta_p}
$$
for some $\delta_p>0.$ Since each of the sums in \eqref{Jdecomp} consists of at most $j(\la)=\mathcal{O}(\log \la)$ terms, the previous inequality implies that 
$$
\sum\limits_{\la\gg 1} \sum\limits_{j=1}^{j(\la)}\|\M_{m^\la_{\pm,R_j(\la)}} \|_{L^p\to L^p}\le C'_p<\infty
$$
(recall that we are here summing only over dyadic values of  $\la$).

\bigskip
Next, we consider  $T_{m^\la_0}.$ Recall that
$$
m^\la_0=\la^{\frac 1\ga-\frac 32} \chi_0(\la^{1/\ga}|u|) g(\la^{1/\ga} u, s_2, s_3)    \chi_0(s_3s')\chi_1(s_3) \, e^{-i\la s_3(1-\frac {s_2^2}4)},
$$
where $g(v,s_2,s_3)$ is a smooth function whose derivatives of any fixed order are uniformly bounded for 
$|v|\lesssim 1, |s_2|\lesssim 1$ and $s_3\sim 1.$

Passing to the same coordinates $\eta$ as before, we here have, in analogy with \eqref{mlapmeta},
\begin{equation}\label{mla0eta}
m^\la_{0,R}(\eta)=\la^{\frac 1\ga-\frac 32} \chi_0(\la^{1/\ga}|u(\eta)|) g(\la^{1/\ga} u(\eta), s_2, \eta_3)  
 \chi_0(-\eta_1,\sqrt{2}\eta_2)\chi_1(\eta_3)\, e^{-i\la \phi_0(\eta)},
\end{equation}
with the same classical phase 
$$
\phi_0(\eta)=\eta_3-\frac {\eta_2^2}{2\eta_3}
$$
as before. Obviously, if we put here $R:=\la^{1/\ga},$ so that $\la=R^\ga,$ then $g(\la^{1/\ga} u, s_2, s_3)$ satisfies the same kind of estimates as $a^\la_{\pm, R}$ in \eqref{ampest}. Moreover, by Example \ref{Fclass}, also 
$$
\phi_0\in \F_{class}\subset \F^{\frac\ga 2,\frac \ga 2,\ga}\subset  \F^{1,\frac \ga 2,\ga}
$$ 
 over $\mathbf \Gamma_R,$ since here $1\le \ga\le 2.$ Note also that here 
$$
\la^{\frac 1\ga-\frac 32}=\la^{-1} \la^{\frac 1\ga-\frac 12}=\la^{-1}R^{1-\frac \ga2}.
$$
Thus, a comparison with \eqref{mlapm} shows that if we apply  Corollary \ref{mainlp}  to $T_{m^\la_0},$ we obtain the same kind of multiplier and kernel estimates \eqref{Test1} and \eqref{kernelest1} for the multiplier $m^\la_0$ as for  $m^\la_{\pm,R},$  but with $R:=\la^{1/\ga}.$ This implies that for $p>\tilde p_c,$ 
$$
\sum\limits_{\la\gg 1}\|\M_{m^\la_0} \|_{L^p\to L^p}\le C'_p<\infty.
$$
\medskip

We are thus left with the multiplier operators $T_{m^\la_{E,R}}.$ Notice here that in  $J^\la_{E,R},$ the factor $(\la s_3|u|)^{-1}$ is of size 
$\la^{-1} R,$ whereas the factor $\la^{-\frac 12} |u|^{\frac \ga 2 -1}$ in $T_{m^\la_{\pm,R}}$ is of size 
$\la^{-\frac 12} R^{1-\frac \ga 2}.$  But, since $\la\ge R^\ga,$ we see that 
$$
\la^{-1} R\le \la^{-\frac 12} R^{1-\frac \ga 2}.
$$
Moreover, the estimates \eqref{Ederest} show that the factor $E(\la s_3 |u|^{\ga}, |u|^{\ga-1},s_2)$ satisfies the same kind of estimates \eqref{ampest} as the factor $a^\la_{\pm,R}.$ Thus, by choosing the phase $\phi:=0$ here and   comparing with $T_{m^\la_{\pm,R}},$ we see that  we obtain the same kind of multiplier and kernel estimates \eqref{Test1} and \eqref{kernelest1} for the multiplier $m^\la_{E,R}$ as for  $m^\la_{\pm,R}$  (actually, even better estimates hold true  for 
$m^\la_{E,R}$). We conclude that we also have 
$$
\sum\limits_{\la\gg 1} \sum\limits_{j=1}^{j(\la)}\|\M_{m^\la_{E,R_j(\la)}} \|_{L^p\to L^p}\le C'_p<\infty,
$$
if $p>\tilde p_c.$ 

Theorem \ref{pbiggerpc} is an immediate consequences of all these estimates.\qed

\end{document}